\documentclass[a4paper,11pt]{article}

\usepackage{amsmath,amsthm}
\usepackage{amssymb}
\usepackage{breqn}
\usepackage{enumerate}
\usepackage{graphicx}
\usepackage{bm}
\usepackage{bbm}
\usepackage[affil-it]{authblk}
\usepackage{tabu}
\usepackage{bold-extra}
\usepackage[hang,flushmargin]{footmisc}
\usepackage[hypertex]{hyperref}
\usepackage{mathtools}
\usepackage{relsize}
\usepackage{scalerel}
\usepackage{xcolor}
\usepackage{titlesec}
\usepackage{apptools}
\usepackage{appendix}

\newtheorem{theorem}{Theorem}
\newtheorem{corollary}[theorem]{Corollary}
\newtheorem{lemma}[theorem]{Lemma}
\newtheorem{proposition}[theorem]{Proposition}
\newtheorem{question}[theorem]{Question}
\newtheorem{prob}[theorem]{Problem}
\newtheorem{conjecture}[theorem]{Conjecture}
\newtheorem{claim}[theorem]{Claim}

\newtheorem{observation}[theorem]{Observation}
\theoremstyle{definition}
\newtheorem{defin}[theorem]{Definition}
\newtheorem{rem}[theorem]{Remark}

\AtAppendix{\counterwithin{theorem}{section}}

\newcommand{\bsf}[1]{\bm{\mathsf{#1}}}

\titleformat{\section}[hang]{\scshape\large\bfseries\filcenter}{\S\thesection}{4pt}{}
\titleformat{\subsection}[hang]{\scshape\bfseries}{\thesubsection.}{4pt}{}

\allowdisplaybreaks

\newcommand{\tss}[1]{\textsuperscript{#1}}

\newcommand{\on}[1]{
	\operatorname{#1}
}

\newcommand\mder{{\Delta\!\!\!\!\!\hbox{\raisebox{0.3ex}{\tiny\ \textbullet}}}\,}

\newcommand{\tdt}{\times\cdots\times}

\newcommand{\tightoverset}[2]{
  \mathop{#2}\limits^{\vbox to -.5ex{\kern-1.15ex\hbox{$#1$}\vss}}}

\newcommand\restr[2]{{
  \left.\kern-\nulldelimiterspace 
  #1 
  \vphantom{\big|} 
  \right|_{#2} 
}}

\makeatletter
\newcommand{\subalign}[1]{%
  \vcenter{%
    \Let@ \restore@math@cr \default@tag
    \baselineskip\fontdimen10 \scriptfont\tw@
    \advance\baselineskip\fontdimen12 \scriptfont\tw@
    \lineskip\thr@@\fontdimen8 \scriptfont\thr@@
    \lineskiplimit\lineskip
    \ialign{\hfil$\m@th\scriptstyle##$&$\m@th\scriptstyle{}##$\hfil\crcr
      #1\crcr
    }%
  }%
}
\makeatother

\newcommand\blfootnote[1]{%
  \begingroup
  \renewcommand\thefootnote{}\footnote{#1}%
  \addtocounter{footnote}{-1}%
  \endgroup
}

\newcommand\ssk[1]{
	\substack{#1}
}

\newcommand\ex{\mathop{\mathbb{E}}}

\newcommand{\exx}{
  \mathop{
    \mathchoice{\vcenter{\hbox{\larger[4]$\mathbb{E}$}}}
               {\kern0pt\mathbb{E}}
               {\kern0pt\mathbb{E}}
               {\kern0pt\mathbb{E}}
  }\displaylimits
}

\makeatletter
\newcommand*\bcdot{\mathpalette\bigcdot@{0.5}}
\newcommand*\bigcdot@[2]{\mathbin{\vcenter{\hbox{\scalebox{#2}{$\m@th#1\bullet$}}}}}
\makeatother

\makeatletter
\def\blfootnote{\gdef\@thefnmark{}\@footnotetext}
\makeatother

\newcommand\id{\mathbbm{1}}

\begin{document}
\begin{center}\Large\noindent{\bfseries{\scshape Quantitative inverse theorem for Gowers uniformity norms $\mathsf{U}^5$ and $\mathsf{U}^6$ in $\mathbb{F}_2^n$}}\\[24pt]\normalsize\noindent{\scshape Luka Mili\'cevi\'c\dag}\\[6pt]
\end{center}
\blfootnote{\noindent\dag\ Mathematical Institute of the Serbian Academy of Sciences and Arts\\\phantom{\dag\ }Email: luka.milicevic@turing.mi.sanu.ac.rs}

\footnotesize
\begin{changemargin}{1in}{1in}
\centerline{\sc{\textbf{Abstract}}}
\phantom{a}\hspace{12pt}~We prove quantitative bounds for the inverse theorem for Gowers uniformity norms $\mathsf{U}^5$ and $\mathsf{U}^6$ in $\mathbb{F}_2^n$. The proof starts from an earlier partial result of Gowers and the author which reduces the inverse problem to a study of algebraic properties of certain multilinear forms. The bulk of the work in this paper is a study of the relationship between the natural actions of $\on{Sym}_4$ and $\on{Sym}_5$ on the space of multilinear forms and the partition rank, using an algebraic version of regularity method. Along the way, we give a positive answer to a conjecture of Tidor about approximately symmetric multilinear forms in 5 variables, which is known to be false in the case of 4 variables. Finally, we discuss the possible generalization of the argument for $\mathsf{U}^k$ norms.
\end{changemargin}
\normalsize
\section{Introduction}

Let us begin by recalling the definition of Gowers uniformity norms~\cite{TimSze}.  

\begin{defin}Let $G$ be a finite abelian group. The \emph{discrete multiplicative derivative operator} $\mder_a$ for shift $a \in G$ is defined by $\mder_a f(x) = f(x + a)\overline{f(x)}$ for functions $f \colon G \to \mathbb{C}$.\\
\indent Let $f \colon G \to \mathbb{C}$ be a function. The \emph{Gowers uniformity norm} $\|f\|_{\mathsf{U}^k}$ is given by the formula
\[\|f\|_{\mathsf{U}^k} = \Big(\exx_{x,a_1,\dots, a_k} \mder_{a_1} \dots \mder_{a_k} f(x) \Big)^{1/2^k}.\]
\end{defin}

We now briefly discuss Gowers uniformity norms (the first part of the introduction of this paper is similar to that in~\cite{FarSymm}). It is well-known that $\|\cdot\|_{\mathsf{U}^k}$ is indeed a norm for $k \geq 2$. The inverse question for Gowers uniformity norms, a central problem in additive combinatorics, asks for a description of functions $f \colon G \to \mathbb{D} = \{z \in \mathbb{C} \colon |z| = 1\}$ whose norm is larger than some constant $c > 0$. Namely, for a given finite abelian group $G$ and the norm $\|\cdot\|_{\mathsf{U}^k}$ we seek a family $\mathcal{Q}$ of functions
from $G$ to $\mathbb{D}$ with the properties that
\begin{itemize}
\item whenever $f \colon G \to \mathbb{D}$ has $\|f\|_{\mathsf{U}^k} \geq c$ then we have correlation $\Big|\ex_{x} f(x) \overline{q(x)}\Big| \geq \Omega_{c}(1)$ for some obstruction function $q \in \mathcal{Q}$, and, 
\item the family of obstructions $\mathcal{Q}$ is roughly minimal in the sense that if for some obstruction function $q\in \mathcal{Q}$ we have $\Big| \ex_{x} f(x) \overline{q(x)}\Big| \geq c$ then $\|f\|_{\mathsf{U}^k} \geq  \Omega_{c}(1)$.
\end{itemize} 

Two classes of groups for which this problem has been most intensively studied are cyclic groups of prime order (denoted $\mathbb{Z}/N\mathbb{Z}$) and finite-dimensional vector spaces over prime fields (denoted $\mathbb{F}_p^n$). When $G = \mathbb{F}_p^n$, in the so-called `high characteristic case' $p \geq k$, Bergelson, Tao and Ziegler~\cite{BergelsonTaoZiegler} proved an inverse theorem in which they took phases of polynomials as the obstruction family. Tao and Ziegler~\cite{TaoZiegler} extended their results to the `low characteristic case' $p < k$, by proving an inverse theorem with phases of \emph{non-classical polynomials} as obstruction. We shall discuss non-classical polynomials slightly later, but for now it is enough to mention that these arise as the solutions of the extremal problem of finding functions $f\colon \mathbb{F}_p^n \to \mathbb{D}$ with $\|f\|_{\mathsf{U}^k} = 1$. On the other hand, when $G = \mathbb{Z}/N\mathbb{Z}$, an inverse theorem was proved by Green, Tao, and Ziegler~\cite{StrongUkZ} and in that setting one need the theory of nilsequences to describe obstructions. Let us also mention the theory of nilspaces, developed in papers by Szegedy~\cite{Szeg}, Camarena and Szegedy~\cite{CamSzeg} and with further improvements, generalizations and contributions by Candela~\cite{CandelaNotes1},~\cite{CandelaNotes2}, Candela and Szegedy~\cite{CandelaSzegedy1},~\cite{CandelaSzegedy2} and Gutman, Manners and Varj\'u~\cite{GMV1},~\cite{GMV2},~\cite{GMV3}, which can be used to give alternative proofs of these inverse results. In particular, Candela, Gonz\'alez-S\'anchez and Szegedy~\cite{nilspacesCharp} were recently able to give an  alternative proof the Tao-Ziegler inverse theorem.\\

When it comes to the question of bounds, it should be noted that all the works above use infinitary methods or regularity lemmas and therefore give ineffective results. Effective bounds were first proved for inverse question for $\|\cdot\|_{\mathsf{U}^3}$ norm by Green and Tao~\cite{StrongU3} for abelian groups of odd order and by Samorodnitsky when $G = \mathbb{F}_2^n$ in~\cite{SamorU3} (see also a recent work of Jamneshan and Tao~\cite{U3AsgarTao}). For $\|\cdot\|_{\mathsf{U}^4}$ norm for the vector space case $G = \mathbb{F}_p^n$, quantitative bounds we obtained by Gowers and the author~\cite{U4paper} when $p \geq 5$ and by Tidor~\cite{Tidor} for $p < 5$. Finally, for general values of $k$, quantitative bounds were achived by Manners~\cite{Manners} in the $G = \mathbb{Z}/N\mathbb{Z}$ setting and by Gowers and the author~\cite{multihomPaper} in $G = \mathbb{F}_p^n$ in the case of high characteristic.\\

The question of getting quantitative bounds in the low characteristic case in the vector space setting is still open. Nevertheless, as a part of the proof of inverse theorem in the high characteristic case~\cite{multihomPaper}, we have the following partial result which holds independently of the characteristic assumption. 
 
\begin{theorem}[Gowers and Mili\'cevi\'c~\cite{multihomPaper}]\label{partialInverseTheorem}Suppose that $f \colon \mathbb{F}^n_p \to \mathbb{D}$ is a function such that $\|f\|_{\mathsf{U}^{k}} \geq c$. Then there exists a multilinear form $\alpha \colon \underbrace{\mathbb{F}^n_p \times \mathbb{F}^n_p \times \dots \times \mathbb{F}^n_p}_{k-1} \to \mathbb{F}_p$ such that
\begin{equation}\label{partialCorrelationForm}\Big| \exx_{x, a_1, \dots, a_{k-1}} \mder_{a_1} \dots \mder_{a_{k-1}} f(x) \omega^{\alpha(a_1, \dots, a_{k-1})}\Big| \geq \Big(\exp^{(O_k(1))}(O_{k,p}(c^{-1}))\Big)^{-1}.\end{equation}
\end{theorem}

The notation $\exp^{(t)}$ stands for the composition of $t$ exponentials, i.e.\ the tower of exponentials of height $t$.\\
\indent An important notion in the following discussion is that of partition rank of a multilinear form, which was introduced by Naslund in~\cite{Naslund} and which we now recall. The \emph{partition rank} of a multilinear form $\alpha\colon G^d \to \mathbb{F}_p$, denoted $\on{prank} \alpha$, is the least number $m$ such that a multilinear form $\alpha \colon G^d \to \mathbb{F}_p$ can be written as
\[\alpha(x_1, \dots, x_d) = \sum_{i \in [m]} \beta_i(x_j \colon j \in I_i) \gamma_i(x_j \colon j \in [d] \setminus I_i),\]
where $\beta_i \colon G^{I_i} \to \mathbb{F}_p$ and $\gamma_i \colon G^{[d] \setminus I_i} \to \mathbb{F}_p$ are multilinear forms for $i \in [m]$. One of the ways to think about this quantity is that the distance between two forms can be measured by the partition rank of their difference.\\

The proof of high characteristic case of the inverse theorem then proceeds by studying the properties the multilinear form $\alpha$ satisfying~\eqref{partialCorrelationForm}. As remarked in~\cite{multihomPaper}, it is plausible this theorem could be used in the low characteristic case as well. For example, it turns out that condition~\eqref{partialCorrelationForm} itself implies that $\on{prank}(\alpha - \alpha') \leq O_c(1)$, where $\alpha'$ is any form obtained from $\alpha$ by permuting some of its variables. This follows from the symmetry argument of Green and Tao~\cite{StrongU3} and the inverse theorem for biased multilinear forms~\cite{Janzer2},~\cite{LukaRank} (i.e.\ the so-called partition vs.\ analytic rank problem), and we do not need any assumptions on the characteristic of the field for this conclusion. The characteristic becomes relevant when we want to pass from an approximately symmetric multilinear form (meaning that differences $\alpha - \alpha'$ above are of low partition rank, rather than 0) to an exactly symmetric form (meaning that it is invariant under permutations of variables). In the high characteristic case it is trivial to achieve this, but in the low characteristic case this task becomes considerably harder.\\ 
\indent Once that we know that there exists a symmetric multilinear form $\sigma$ such that $\on{prank}(\alpha - \sigma)$ is small, we may deduce that~\eqref{partialCorrelationForm} holds for $\sigma$ in place of $\alpha$, as it turns out that~\eqref{partialCorrelationForm} is robust under modification by forms of small partition rank. Final observation required in the high characteristic case is that symmetric multilinear forms are precisely the additive derivatives of polynomials. This fact allows us to reduce the proof to the case of the inverse theorem for $\|\cdot\|_{\mathsf{U}^{k-1}}$ norm, which we assume by induction.\\ 

Note that there are two places in the argument above which rely on the fact that $p \geq k$. The first one is the relationship between approximately symmetric and exactly symmetric multilinear forms, and the second one is the description of the additive derivatives of polynomials. Tidor~\cite{Tidor} was able to resolve these two issues for trilinear multilinear forms and carry out the strategy of elucidating the structure of multilinear form provided by Theorem~\ref{partialInverseTheorem} in $\mathbb{F}_2^n$, thus proving a quantitative inverse theorem for $\|\cdot\|_{\mathsf{U}^4}$ norm in the low characteristic case. For the first issue, it turns out that in the case of trilinear forms, approximately symmetric forms are close to exactly symmetric ones. For the second issue, we need a definition. We say that a multilinear form $\sigma \colon G^k \to \mathbb{F}_p$ is \emph{strongly symmetric} if it is symmetric and the multilinear form $(x,y_{p+1}, \dots, y_k) \mapsto  \sigma(x,\dots, x, y_{p+1}, \dots, y_k)$ is also symmetric (where $x$ occurs $p$ times). It turns out we again have a rather satisfactory description, namely that order $k$ additive derivatives of generalized polynomials of degree at most $k$ are precisely strongly symmetric multilinear forms in $k$ variables (this is Proposition 3.5~\cite{Tidor}, see also~\cite{TaoZiegler}). Once we know that a form in~\eqref{partialCorrelationForm} is symmetric, it is not hard to show that it is strongly symmetric using similar arguments. However, while the description of additive derivatives of generalized polynomials holds for all numbers of variables, it turns out that this is surprisingly not the case with approximate symmetry problem.

\begin{theorem}[\cite{FarSymm}]\label{4varscounter}Given a sufficiently large positive integer $n$ there exists a multilinear form $\alpha \colon \mathbb{F}_2^n \times \mathbb{F}_2^n \times \mathbb{F}_2^n \times \mathbb{F}_2^n \to \mathbb{F}_2$ which is 3-approximately symmetric in the sense that $\on{prank}(\alpha + \alpha') \leq 3$ for any $\alpha'$ obtained from $\alpha$ by permuting its variables, and $\on{prank}(\sigma + \alpha) \geq \Omega(\sqrt[3]{n})$ for all symmetric multilinear forms $\sigma$.\end{theorem} 

With all this in mind, the main remaining obstacle in the way of the quantitative inverse theorem for uniformity norms can be formulated as follows. 

\begin{prob}\label{corrtossymm}Suppose that $f \colon \mathbb{F}_p^n \to \mathbb{D}$ is a function and that $\alpha \colon G^{k} \to \mathbb{F}_p$ is a multilinear form. Assume that
\begin{equation}\Big|\exx_{a_1, \dots, a_k, x} \mder_{a_1, \dots, a_k} f(x) \omega^{\alpha(a_1, \dots, a_k)}\Big| \geq c.\label{mainproblemcorrassumption}\end{equation}
Show that there exists a strongly symmetric multilinear form $\sigma \colon G^k \to \mathbb{F}_p$ such that $\on{prank}(\alpha - \sigma)$ is quantitatively bounded in terms of $k,p$ and $c$.\end{prob}

It should be noted that in a qualitative sense this follows from the inverse theorem of Tao and Ziegler, and that, in the light of Theorem~\ref{4varscounter}, the assumption~\eqref{mainproblemcorrassumption} is essential.\\ 

Our main result in this paper is that, for $k \in \{4,5\}$, we may overcome the additional difficulties caused by irregular behaviour of approximately symmetric forms. 

\begin{theorem}\label{mainthm}Let $k \in \{4,5\}$, $f \colon \mathbb{F}_2^n \to \mathbb{D}$ a function and $\alpha \colon (\mathbb{F}_2^n)^k \to \mathbb{F}_2$ a multilinear form in $k$ variables. Suppose that 
\begin{equation}\Big|\exx_{a_1, \dots, a_k, x} \mder_{a_1, \dots, a_k} f(x) (-1)^{\alpha(a_1, \dots, a_k)}\Big| \geq c.\label{mainthmassn}\end{equation}
Then there exists a strongly symmetric multilinear form $\sigma \colon (\mathbb{F}_2^n)^k \to \mathbb{F}_2$ such that $\alpha + \sigma$ has partition rank at most $O(\exp^{(O(1))} c^{-1})$.
\end{theorem}

In fact, as we shall explain in the outline of the proof and in the concluding remarks, most of the arguments work for higher values of $k$, and the arguments of the proof of Theorem~\ref{mainthm} are capable of almost proving the general case. As we shall explain in the concluding remarks, the proof breaks only for multilinear forms with very special properties; see Conjecture~\ref{finaldiffconj}.\\

As a corollary, we deduce the quantitative inverse theorems for $\|\cdot\|_{\mathsf{U}^5}$ and $\|\cdot\|_{\mathsf{U}^6}$ norms in $\mathbb{F}_2^n$.

\begin{corollary}\label{inverseu5u6cor}Let $k \in \{5,6\}$ and let $f \colon \mathbb{F}_2^n \to \mathbb{D}$ be a function such that $\|f\|_{\mathsf{U}^k} \geq c$. Then there exists a generalized polynomial $q\colon \mathbb{F}_2^n \to \mathbb{T}$ of degree at most $k - 1$ such that 
\[\Big|\exx_x f(x)\exp\Big(2\pi i q(x)\Big)\Big| \geq  \Big(\exp^{(O(1))}(O(c^{-1}))\Big)^{-1}.\]
\end{corollary}

It should also be noted that the arguments do not depend crucially on the choice of field $\mathbb{F}_2$. This case is of principal interest in theoretical computer science, and in our case it simplifies the notation somewhat.\\ 

\noindent\textbf{Outline of the proof.} In the rest of introduction we sketch the proof of Theorem~\ref{mainthm} and state other results of this paper. The overall theme is that of passing from multilinear forms with approximate algebraic properties to forms with exact versions of those properties. Mainly, these properties will be about symmetry in some set of variables.\\

In order to state the results, we define a natural action of $\on{Sym}_k$ on $G^k$ given by permuting the coordinates, which is similar to the left regular representation of the group $\on{Sym}_k$. For a permutation $\pi \in \on{Sym}_k$ we misuse the notation and write $\pi \colon G^k \to G^k$ for the map defined by $\pi(x_{[k]}) = (x_{\pi^{-1}(1)},$ $x_{\pi^{-1}(2)},$ $\dots,$ $x_{\pi^{-1}(k)})$, where $x_{[k]}$ is abbreviation for $x_1, \dots, x_k$. This defines an action on $G^k$. Given a multilinear form $\alpha \colon G^k \to \mathbb{F}_2$ and a permutation $\pi$ inducing the map $\pi \colon G^k \to G^k$, we may compose the two maps and the composition $\alpha \circ \pi$ would also be a multilinear form.\\

First of all, we show that if a multilinear form is approximately symmetric in the first two variables $x_1$ and $x_2$, then it can be made exactly symmetric in $x_1$ and $x_2$ at the small cost. Note that this result works for any arity of the form.

\begin{theorem}\label{approximateSymmetry2}Let $k \geq 2$ be an integer and let $\alpha \colon G^k \to \mathbb{F}_2$ be a multilinear form. Suppose that $\on{prank}(\alpha + \alpha \circ (1\,\,2)) \leq r$. Then there exists a multilinear form $\alpha' \colon G^k \to \mathbb{F}_2$ such that $\on{prank}(\alpha + \alpha') \leq \exp^{(O(1))}(O(r))$ and $\alpha'$ is symmetric in the first $2$ variables.\end{theorem}

In general, it is natural to try to build up exact symmetry of an approximately symmetric form one variable at a time. Theorem~\ref{approximateSymmetry2} shows that the first step of such a strategy can always be carried out. On the other hand, if a multilinear form $\alpha$ is symmetric in variables $x_{[2\ell]}$ and if $\on{prank}(\alpha + \alpha \circ (1\,\,2\ell + 1)) \leq r$, we may easily pass to a multilinear form $\alpha'$ given by
\[\alpha'(x_{[k]}) = \sum_{i \in [2\ell + 1]} \alpha \circ (i\,\,2\ell + 1) (x_{[k]})\]
which differs from $\alpha$ by partition rank at most $kr$ and is exactly symmetric in $x_{[2\ell + 1]}$. Note that the fact that $2\ell + 1$ is odd plays a crucial role.\\
\indent With this in mind, since our forms have at most 5 variables, the key question in this paper is how to pass from forms symmetric in $x_{[3]}$ to those symmetric in $x_{[4]}$. Recall that in general, this is not possible for forms in 4 variables (counterexample in~\cite{FarSymm} is actually symmetric in $x_{[3]}$). The next result shows that we may achieve this if our form has the additional property that $\alpha(u,u,x_3,x_4)$ vanishes.\\

\begin{theorem}\label{mainSymmExtn4}Suppose that $\alpha \colon G^4 \to \mathbb{F}_2$ is a multilinear form such that:
\begin{itemize}
\item $\alpha$ is symmetric in the first 3 variables,
\item $\alpha(u,u,x_3, x_4) = 0$ for all $u, x_3, x_4 \in G$,
\item $\on{prank} (\alpha + \alpha \circ (3\,\,4)) \leq r$.
\end{itemize}
Then there exists a multilinear form  $\alpha' \colon G^4 \to \mathbb{F}_2$ such that $\on{prank} (\alpha + \alpha') \leq \exp(\exp(O(r)))$ and $\alpha'$ is symmetric.
\end{theorem}

The final result is about obtaining exact symmetry in variables $x_{[4]}$ for forms in 5 variables. Unlike the case of 4 variables, we may again achieve this without additional assumptions.

\begin{theorem}\label{mainSymmExtn5}Suppose that $\alpha \colon G^5 \to \mathbb{F}_2$ is a multilinear form such that:
\begin{itemize}
\item $\alpha$ is symmetric in the first 3 variables,
\item $\on{prank} (\alpha + \alpha \circ (3\,\,4)) \leq r$.
\end{itemize}
Then there exists a multilinear form  $\alpha' \colon G^5 \to \mathbb{F}_2$ such that $\on{prank} (\alpha + \alpha') \leq \exp(\exp(O(r^{O(1)})))$ and $\alpha'$ is symmetric in the first 4 variables.
\end{theorem}

A neat corollary is that in the case of 5 variables we may again pass from approximately symmetric multilinear forms to those that are exactly symmetric, giving an affirmative answer to a question of Tidor~\cite{Tidor}, which is surprising in the light of Theorem~\ref{4varscounter}.

\begin{corollary}Suppose that $\alpha \colon G^5 \to \mathbb{F}_2$ is a multilinear form such that $\on{prank} (\alpha + \alpha \circ \pi) \leq r$ holds for all permutations $\pi \in \on{Sym}_5$. Then there exists a symmetric multilinear form  $\alpha' \colon G^5 \to \mathbb{F}_2$ such that $\on{prank} (\alpha + \alpha') \leq \exp^{(O(1))}(O(r^{O(1)}))$.\end{corollary}

Since we have additional condition in Theorem~\ref{mainSymmExtn4}, we need another `approximate-to-exact' claim. We may think of a form $\alpha(x_{[k]})$ which is symmetric in $x_1$ and $x_2$ such that $\alpha(u,u,a_3, \dots, a_k) = 0 $ as being `without repeated coordinates', since this vanishing condition is equivalent to not having monomials $x_{1,i_1} x_{2,i_2} \dots x_{k,i_k}$ with $i_1 = i_2$ present in the expansion of $\alpha(x_{[k]})$. The next theorem thus concerns forms that are `approximately without repeated coordinates'. (Symmetric multilinear forms without repeated coordinates are called classical multilinear forms~\cite{TaoZiegler},~\cite{Tidor}.)

\begin{theorem}\label{copiesPropertyVanish} Let $2 \leq k \leq 5$. Let $\alpha \colon G^k \to \mathbb{F}_2$ be a multilinear form which is symmetric in the first $m \geq 2$ variables. Suppose also that the multilinear form $(d, a_3, \dots, a_k) \mapsto \alpha(d,d,a_3,\dots, a_k)$ has partition rank at most $r$. Then there exist a subspace $U \leq G$ of codimension at most $\exp^{(O(1))}(O(r))$ and a multilinear form $\alpha' \colon U^k \to \mathbb{F}_2$, also symmetric in the first $m$ variables such that $\alpha'(d,d,a_3, \dots, a_k) = 0$ for all $d,a_3, \dots, a_k \in U$ and  $\on{prank} (\alpha|_{U \tdt U} + \alpha') \leq \exp(\exp(O(r^{O(1)})))$.\end{theorem}

The proof of Theorem~\ref{mainthm} then proceeds by a back-and-forth argument which uses the condition~\eqref{mainthmassn} as follows: we first deduce approximate symmetry properties of $\alpha$ and then use Theorem~\ref{approximateSymmetry2} to replace $\alpha$ by a form symmetric in $x_1$ and $x_2$. Again use the condition~\eqref{mainthmassn} to deduce approximate symmetry in $x_2$ and $x_3$, and then pass to a form exactly symmetric in $x_{[3]}$, etc.\ each time using appropriate symmetry extension statement. The additional assumption in Theorem~\ref{mainSymmExtn4} makes the details of the argument more involved than this short sketch, but the overall structure of the argument is the one described.\\

The proofs of the `approximate-to-exact' theorems are based on an algebraic regularity method, following the overall philosophy of the approaches in~\cite{U4paper},~\cite{extensionsPaper},~\cite{multihomPaper} and~\cite{FarSymm}. The proofs in this paper proceed by proving appropriate weak algebraic regularity lemmas that incorporate the additional algebraic properties of the forms in the question. Note that, despite similarity in the spirit, the previous works (\cite{U4paper},~\cite{extensionsPaper},~\cite{multihomPaper} and~\cite{FarSymm}) did not need such specialized lemmas. These weak algebraic regularity lemmas (Lemmas~\ref{symmRespWRegLemma} and~\ref{symmRegLemmaTimes2}), coupled with multilinear algebra arguments, allow us to use the low partition rank expressions (e.g. $\alpha + \alpha \circ (3\,\,4)$ in Theorem~\ref{mainSymmExtn5}) to modify the given forms to the ones with exact properties (symmetric in $x_{[4]}$ in Theorem~\ref{mainSymmExtn5}). Using the inverse theorem for biased multilinear forms (partition vs.\ analytic rank problem) and weak regularity lemmas rather than strong ones enable us to prove reasonable bounds in these results. We shall return to the discussion of these arguments in the concluding remarks.\\

\noindent\textbf{Comparison with~\cite{TaoZiegler}.} The proof of the inverse theorem for Gowers uniformity norms in the low characteristic by Tao and Ziegler~\cite{TaoZiegler} has a similar principle of deducing the full theorem from a partial result. In their case, they start from an earlier result they proved with Bergelson, saying that if $\|f\|_{\mathsf{U}^k} \geq c$ then $f$ correlates with phase of a generalized polynomial of degree $d(k, p)$, which may be larger than $k-1$. Thus their remaining task is to reduce this degree to the optimal one, which is a significantly different situation to ours, where we start from Theorem~\ref{partialInverseTheorem} which is itself an optimal result when it comes to the degree of the obtained form $\alpha$, but we still have to deduce different further algebraic properties of $\alpha$. Their argument also has variants of `approximate-to-exact' claims, but they primarily deal with \emph{classical symmetric multilinear} forms, namely symmetric multilinear forms without monomials with multiple occurrences of the same coordinate, so there is no irregular behaviour like that exhibited by approximately symmetric multilinear forms. Also, they rely on Multidimensional Szemer\'edi theorem (see proof of Theorem 4.1 in~\cite{TaoZiegler}), which, in order to be quantitative, requires understanding of directional uniformity norms~\cite{Austin1},~\cite{Austin2},~\cite{DirGU}, which is evaded in our approach.\\

\noindent\textbf{Acknowledgements.} This work was supported by the Serbian Ministry of Education, Science and Technological Development through Mathematical Institute of the Serbian Academy of Sciences and Arts.

\section{Preliminaries}

Throughout the paper $G$ stands for our ambient group which is $\mathbb{F}_2^n$ for some large $n$.\\

\noindent\textbf{Notation.} We use the standard expectation notation $\ex_{x \in X}$ as shorthand for the average $\frac{1}{|X|} \sum_{x \in X}$, and when the set $X$ is clear from the context we simply write $\ex_x$. As in~\cite{multihomPaper},~\cite{LukaRank}, we use the following convention to save writing in situations where we have many indices appearing in predictable patterns. Instead of denoting a sequence of length $m$ by $(x_1, \dots, x_m)$, we write $x_{[m]}$, and for $I\subset[m]$ we write $x_I$ for the subsequence with indices in $I$.\\
\indent We extend the use of the dot product notation to any situation where we have two sequences $x=x_{[n]}$ and $y=y_{[n]}$ and a meaningful multiplication between elements $x_i y_i$, writing $x\cdot y$ as shorthand for the sum $\sum_{i=1}^n x_i y_i$. For example, if $\lambda=\lambda_{[n]}$ is a sequence of scalars, and $A=A_{[n]}$ is a suitable sequence of maps, then $\lambda\cdot A$ is the map $\sum_{i=1}^n\lambda_iA_i$.\\  
\indent Frequently we shall consider \emph{slices} of functions $f \colon G^{[k]} \to X$, by which we mean functions of the form $f_{x_I} \colon G^{[k] \setminus I} \to X$ that send $y_{[k]\setminus I}$ to $f(x_I, y_{[k]\setminus I})$, for $I \subset [k], x_I \in G^I$. (Here we are writing $(x_I,y_{[k]\setminus I})$ not for the concatenation of the sequences $x_I$ and $y_{[k]\setminus I}$ but for the `merged' sequence $z_{[n]}$ with $z_i=x_i$ when $i\in I$ and $z_i=y_i$ otherwise.) If $I$ is a singleton $\{i\}$ and $z_i\in G_i$, then we shall write $S_{z_i}$ instead of $S_{z_{\{i\}}}$. Sometimes, the index $i$ will be clear from the context and it will be convenient to omit it. For example, $f(x_{[k]\setminus\{i\}},a)$ stands for $f(x_1,\dots,x_{i-1},a,x_{i+1},\dots,x_k)$. If the index is not clear, we emphasize it by writing it as a superscript to the left of the corresponding variable, e.g.\ $f(x_{[k]\setminus\{i\}},{}^i\,a)$.\\ 

\indent More generally, when $X_1, \dots, X_k$ are finite sets, $Z$ is an arbitrary set, $f \colon X_1 \tdt X_k = X_{[k]} \to Z$ is a function, $I \subsetneq [k]$ and $x_i \in X_i$ for each $i \in I$, we define a function $f_{x_I} \colon X_{[k] \setminus I} \to Z$, by mapping each $y_{[k] \setminus I} \in X_{[k] \setminus I}$ as $f_{x_I}(y_{[k] \setminus I}) = f(x_I, y_{[k] \setminus I})$. When the number of variables is small -- for example, when we have a function $f(x,y)$ that depends only on two variables $x$ and $y$ instead of on indexed variables -- we also write $f_x$ for the map $f_x(y)=f(x,y)$.\\

Let us first recall the definition of higher-dimensional box norms. Let $X_1, \dots, X_k$ be arbitrary sets. The \emph{box norm} of a function $f \colon X_1 \tdt X_k \to \mathbb{C}$ (see for example Definition B.1 in the Appendix B of~\cite{GreenTaoPrimes}) is defined by
\[\|f\|_{\square^{k}}^{2^k} = \exx_{x_1, y_1 \in X_1, \dots, x_k, y_k \in X_k} \prod_{I \subset [k]} \operatorname{Conj}^{|I|} f(x_I, y_{[k] \setminus I}).\]
The following is a well-known generalized Cauchy-Schwarz inequality for the box norm, which we refer to as the Gowers-Cauchy-Schwarz inequality.

\begin{lemma}\label{unifBound}Let $f_I \colon X_1 \tdt X_k \to \mathbb{C}$ be a function for each $I \subset [k]$. Then
\[\Big|\exx_{x_1, y_1 \in X_1, \dots, x_k, y_k \in X_k} \prod_{I \subset [k]} \operatorname{Conj}^{|I|} f_I(x_I, y_{[k] \setminus I})\Big| \leq \prod_{I \subset [k]} \|f_I\|_{\square^k}.\]
\end{lemma}

The following lemma concerns quasirandomness of quadratic polynomials in the case of low characteristic.

\begin{lemma}\label{quadraticVarSize}Suppose that $\rho_1, \dots, \rho_k \colon U \times U \to \mathbb{F}_2$ are bilinear forms such that all non-zero linear combinations of $\rho_i, \rho_i \circ (1\,\,2)$ have rank at least $k + 1$. Then the number of solutions $u \in U$ to $\rho_i(u,u) = 0$ for all $i \in [k]$ is at least $2^{-k-1}$.\end{lemma}

\begin{proof}Let $X = \{u \in U \colon \rho_1(u,u) = \dots = \rho_k(u,u) = 0\}$. Note that
\[|U|^{-1}|X| = \exx_{\lambda_1, \dots, \lambda_k \in \mathbb{F}_2} \exx_{u \in U} (-1)^{\lambda_1 \rho_1(u,u) + \dots + \lambda_k \rho_k(u,u)} = 2^{-k} + 2^{-k}\sum_{\lambda \in \mathbb{F}_2^{k} \setminus\{0\}}\exx_{u \in U} (-1)^{\lambda_1 \rho_1(u,u) + \dots + \lambda_k \rho_k(u,u)},\]
so we need to estimate 
\[\exx_{u \in U} (-1)^{\lambda_1 \rho_1(u,u) + \dots + \lambda_k \rho_k(u,u)}\]
for non-zero $\lambda$. Write $\phi = \lambda_1 \rho_1 + \dots + \lambda_k \rho_k$. Then 
\begin{align*}\Big|\exx_{u \in U} (-1)^{\phi(u,u)}\Big| = \Big|\exx_{u, v \in U} (-1)^{\phi(u +v,u + v)}\Big| = &\Big|\exx_{u, v \in U} (-1)^{(\phi + \phi \circ (1\,2))(u, v) + \phi(u,u) + \phi(v,v)}\Big| \\
\leq &\|(-1)^{\phi + \phi \circ (1\,2)}\|_{\square} = \on{bias} (\phi + \phi \circ (1\,2)) \leq 2^{-(k+1)}.\end{align*}
The lemma now follows.\end{proof}

We make use of Sanders's version of Bogolyubov-Ruzsa lemma.

\begin{theorem}[Corollary A.2 in~\cite{Sanders}]\label{sandersBogRuzsa}Suppose that $A \subseteq \mathbb{F}_2^n$ is a set of the density $\delta$. Then there exists a subspace $V \leq \mathbb{F}_2^n$ of codimension $O(\log^4 (2\delta^{-1}))$ such that $V \subseteq 4A$.\end{theorem}

For a multilinear form $\alpha \colon G^k \to \mathbb{F}_2$ we have two important quantities that measure its structure. The first is the \emph{bias}, defined as
\[\on{bias} \alpha = \exx_{x_{[k]}} (-1)^{\alpha(x_{[k})}.\]
This quantity is closely related to the \emph{analytic rank}, introduced by Gowers and Wolf~\cite{TimWolf}, which is defined as $\on{arank} \alpha = -\log_2 (\on{bias} \alpha)$.\\
\indent The second one is the \emph{partition rank}, introduced by Naslund in~\cite{Naslund}, defined as the least non-negative integer $r$ such that there exist sets $\emptyset \not= I_i \subsetneq [k]$ and multilinear forms $\beta_i \colon G^{I_i} \to \mathbb{F}_2$ and $\gamma_i \colon G^{[k] \setminus I_i} \to \mathbb{F}_2$, where $i \in [r]$ such that 
\[\alpha(x_{[k]}) = \sum_{i \in [r]} \beta_i(x_{I_i}) \gamma_i(x_{[k] \setminus I_i}).\]
We need the following result on the relationship between the two mentioned quantities, proved in~\cite{LukaRank}. Very similar result was proved by Janzer~\cite{Janzer2}, and previous qualitative versions were proved by Green and Tao~\cite{GreenTaoPolys} and by Bhowmick and Lovett~\cite{BhowLov}, generalizing an approach of Green and Tao.

\begin{theorem}[Inverse theorem for biased multilinear forms (the partition rank vs the analytic rank problem)~\cite{LukaRank}]\label{biasedinversethm}For every positive integer $k$ there are constants $C = C_k, D = D_k > 0$ with the following property. Suppose that $\alpha \colon G^k \to \mathbb{F}_2$ is a multilinear form such that $\ex_{x_{[k]}} (-1)^{\alpha(x_{[k]})} \geq c$, for some $c > 0$. Then $\on{prank}\alpha \leq C \log_2^{D} c^{-1}$.\end{theorem}

The next lemma relates the partition rank of a multilinear form $\alpha$ to the partition rank of restrictions of $\alpha$.

\begin{lemma}Let $\alpha \colon G^k \to \mathbb{F}_2$ be a multilinear form. Let $U \leq G$ be a subspace of codimension $d$. Then $\on{prank} \alpha \leq \on{prank} \alpha|_{U \tdt U} + kd$.\end{lemma}

\begin{proof}Let $G = U \oplus \langle w_1, \dots, w_d\rangle$. Then there exists linear maps $\pi \colon G \to U$ and $\phi_1, \dots, \phi_d \colon G \to \mathbb{F}_2$ such that $x = \pi(x) + \sum_{i \in [d]} \phi_i(x) w_i$ for all $x \in G$. Then
\begin{align*}\alpha(x_1, \dots, x_k) = &\alpha(\pi(x_1) + \phi_1(x_1) w_1 + \dots + \phi_d(x_1) w_d, x_2, \dots, x_k) \\
= &\sum_{i \in [d]} \phi_i(x_1) \alpha(w_i, x_2, \dots, x_k) + \alpha(\pi(x), x_2, \dots, x_k)\\
= & \dots\\
= & \Big(\sum_{c \in [k]} \sum_{i \in [d]} \phi_i(x_c) \alpha(\pi(x_1), \dots, \pi(x_{c-1}), w_i, x_{c+1}, \dots, x_k)\Big) + \alpha(\pi(x_1), \dots, \pi(x_k)).\qedhere\end{align*}
\end{proof}

We need a variant of the above lemma which concerns forms whose restrictions are close to symmetric forms.

\begin{corollary}\label{symmetryInher}Let $\alpha \colon G^k \to \mathbb{F}_2$ be a multilinear form, let $U \leq G$ be a subspace of codimension $d$ and let $\sigma \colon U^k \to \mathbb{F}_2$ be a symmetric multilinear form such that $\on{prank}(\alpha|_{U\tdt U} + \sigma) \leq r$. Then there exists a symmetric multilinear form $\tilde{\sigma} \colon G^k \to \mathbb{F}_2$ such that $\on{prank}(\alpha + \tilde{\sigma}) \leq kd + r$.\end{corollary}

\begin{proof}Let the maps $\pi, \phi_1, \dots, \phi_d$ be as in the proof of the previous lemma. Define mutlilinear form $\tilde{\sigma} \colon G^k \to \mathbb{F}_2$ by $\tilde{\sigma}(x_1, \dots, x_k) = \sigma(\pi(x_1), \dots, \pi(x_k))$. Since $\sigma$ is symmetric on $U \tdt U$ we have that $\tilde{\sigma}$ is symmetric. On the other hand, since $\pi$ is a projection onto $U$, we see that $\tilde{\sigma}|_{U \tdt U} = \sigma$ and hence
\[\on{prank}\Big((\alpha + \tilde{\sigma})|_{U \tdt U}\Big) \leq r.\]
By the previous lemma we conclude that $\on{prank}(\alpha + \tilde{\sigma}) \leq kd + r$, as required.\end{proof}

We also need a result on images of high-rank maps.

\begin{lemma}[Lemma 2.5 in~\cite{extensionsPaper}]\label{nonzeroPtLemma}Let $\rho, \beta_1,\dots,\beta_r \colon G^{k} \to \mathbb{F}_2$ be multilinear forms and let $m \in \mathbb{N}$ be such that for all choices of $\lambda \in \mathbb{F}^r_2$, $\on{bias}\Big(\rho + \lambda \cdot \beta\Big) < 2^{- k (r + m)}$. Then for any multilinear forms $\gamma_i \colon G^{I_i} \to \mathbb{F}_2$, $\emptyset \not= I_i \subsetneq [k]$, $i=1,2,\dots,m$, we may find $x_{[k]} \in G^{[k]}$ such that
\begin{itemize}
\item $\rho(x_{[k]}) = 1$,
\item $(\forall i \in [r])\  \beta_i(x_{[k]}) = 0$, and
\item $(\forall i \in [m])\  \gamma_i(x_{I_i}) = 0$.
\end{itemize}
\end{lemma}

The following lemma is important for neglecting the low rank differences.

\begin{lemma}\label{closeformsreplacementinverse}Suppose that a function $f \colon G \to \mathbb{D}$ and a multilinear form $\alpha \colon G^k \to \mathbb{F}_2$ satisfy
\[\Big| \exx_{x, a_1, \dots, a_k} \mder_{a_1} \dots \mder_{a_k} f(x) (-1)^{\alpha(a_1, \dots, a_k)}\Big| \geq c.\]
Let $\beta \colon G^k \to \mathbb{F}_2$ be another multilinear form such that $\on{prank}(\alpha + \beta) \leq r$. Then
\[\Big| \exx_{x, a_1, \dots, a_k} \mder_{a_1} \dots \mder_{a_k} f(x) (-1)^{\beta(a_1, \dots, a_k)}\Big| \geq 2^{-2^{k+1}r}c.\]
\end{lemma}

\begin{proof}Using the notion of the large multilinear spectrum (Definition 8 in~\cite{DirGU}), and writing $c' = c^{2^{-k}}$, we see that the condition is that $\alpha \in \on{Spec}^{\text{ml}}_{c'} f$. We need to show that $\beta \in \on{Spec}^{\text{ml}}_{c''} f$ for some $c'' \geq 2^{-2r} c'$. This follows from the proof of Lemma 39 in~\cite{DirGU}, which has large bias assumption instead of low partition rank assumption.\end{proof}

Similarly, we are allowed to pass to arbitrary subspaces.

\begin{lemma}\label{boundedcodimsubspacepass}Suppose that a function $f \colon G \to \mathbb{D}$ and a multilinear form $\alpha \colon G^k \to \mathbb{F}_2$ satisfy
\[\Big| \exx_{x, a_1, \dots, a_k} \mder_{a_1} \dots \mder_{a_k} f(x) (-1)^{\alpha(a_1, \dots, a_k)}\Big| \geq c.\]
Let $U \leq G$ be a subspace. Then there exists a function $\tilde{f} \colon U \to \mathbb{D}$ such that
\[\Big| \exx_{x, a_1, \dots, a_k \in U} \mder_{a_1} \dots \mder_{a_k} \tilde{f}(x) (-1)^{\alpha(a_1, \dots, a_k)}\Big| \geq c.\]
\end{lemma}

\begin{proof}Let $W$ be a subspace such that $G = U \oplus W$. Using this decomposition we get
\[\Big| \exx_{\ssk{x, a_1, \dots, a_k \in U\\y, b_1, \dots, b_k \in W}} \mder_{a_1 + b_1} \dots \mder_{a_k+ b_k} f(x+y) (-1)^{\alpha(a_1 + b_1, \dots, a_k + b_k)}\Big| \geq c.\]
By triangle inequality we get
\[\exx_{y, b_1, \dots, b_k \in W} \Big| \exx_{x, a_1, \dots, a_k \in U} \mder_{a_1 + b_1} \dots \mder_{a_k+ b_k} f(x+y) (-1)^{\alpha(a_1 + b_1, \dots, a_k + b_k)}\Big| \geq c.\]
By averaging, there exists a choice of $y, b_1, \dots, b_k \in W$ such that
\[\Big| \exx_{x, a_1, \dots, a_k \in U} \mder_{a_1 + b_1} \dots \mder_{a_k+ b_k} f(x+y) (-1)^{\alpha(a_1 + b_1, \dots, a_k + b_k)}\Big| \geq c.\]
We introduce additional variables $a'_1, \dots, a'_k \in U$ and make a change of variables where we replace $x$ by $x + a_1 + \dots + a_k$ and $a_i$ by $a_i + a'_i$ to obtain
\begin{align*}c \leq & \Big| \exx_{\ssk{x, a_1, \dots, a_k \in U\\a'_1, \dots, a'_k \in U}} \mder_{a_1 + a'_1 + b_1} \dots \mder_{a_k + a'_k + b_k} f(x + a_1 + \dots + a_k +y) (-1)^{\alpha(a_1 + a'_1 + b_1, \dots, a_k  + a'_k + b_k)}\Big|\\
= & \Big| \exx_{\ssk{x, a_1, \dots, a_k \in U\\a'_1, \dots, a'_k \in U}} \prod_{I \subseteq [k]} \Big(\on{Conj}^{k-|I|} f\Big(x + y + \sum_{i \in I} (a'_i + b_i) + \sum_{i \in [k] \setminus I} a_i\Big) (-1)^{\alpha(a'_I, a_{[k] \setminus I})}\Big) \\
&\hspace{8cm}(-1)^{\alpha(a_1 + a'_1 + b_1, \dots, a_k  + a'_k + b_k) + \alpha(a_1 + a'_1, \dots, a_k  + a'_k)}\Big|.\end{align*}
By triangle inequality 
\begin{align*}c \leq & \exx_{x \in U}\Big| \exx_{\ssk{a_1, \dots, a_k \in U\\a'_1, \dots, a'_k \in U}} \prod_{I \subseteq [k]} \Big(\on{Conj}^{k-|I|} f\Big(x + y + \sum_{i \in I} (a'_i + b_i) + \sum_{i \in [k] \setminus I} a_i\Big) (-1)^{\alpha(a'_I, a_{[k] \setminus I})}\Big) \\
&\hspace{8cm}(-1)^{\alpha(a_1 + a'_1 + b_1, \dots, a_k  + a'_k + b_k) + \alpha(a_1 + a'_1, \dots, a_k  + a'_k)}\Big|.\end{align*}

For each $x \in U$, we may define auxiliary functions $f_{I, x} \colon G^k \to \mathbb{D}$ for $I \subseteq [k]$, each of the form
\[f_{I, x}(z_1, \dots, z_k) = f\Big(x +y + \sum_{i \in I} b_i + z_1 + \dots + z_k\Big) (-1)^{\alpha(z_1, \dots, z_k)} \ell_I(z_1, \dots, z_k),\]
where $\ell_I$ is a product of several factors, each depending on a proper subset of variables $z_1, \dots, z_k$, coming from terms that arise from expansion of
\[(-1)^{\alpha(a_1 + a'_1 + b_1, \dots, a_k  + a'_k + b_k) + \alpha(a_1 + a'_1, \dots, a_k  + a'_k)}\]
and whose values are of modulus 1. In the new notation, we get
\[c \leq \exx_{x \in U}\Big|\exx_{\ssk{a_1, \dots, a_k \in U\\a'_1, \dots, a'_k \in U}} \prod_{I \subseteq [k]} \on{Conj}^{k-|I|}f_{I,x}(a'_I, a_{[k] \setminus I})\Big|.\]

Using the Gowers-Cauchy-Schwarz inequality we get 
\[c \leq \exx_{x \in U} \prod_{I \subseteq [k]} \|f_{I, x}\|_{\square^k},\]
and H\H{o}lder's inequality implies that
\[c \leq \prod_{I \subseteq [k]} \Big(\exx_{x \in U} \|f_{I, x}\|_{\square^k}^{2^k}\Big)^{2^{-k}}.\]
Thus, there is a choice of subset $I \subseteq [k]$ such that $c \leq \ex_x \|f_{I, x}\|_{\square^k}^{2^k}$. By definition, writing $\tilde{b} = \sum_{i \in I} b_i$ and splitting $\ell_I(z_1, \dots, z_k)$ into further factors that do not depend on a variable in $z_{[k]}$, we have
\[f_{I, x}(y_1, \dots, y_k) = f(x + y_1 + \dots + y_k + \tilde{b}) (-1)^{\alpha(y_1, \dots, y_k)} \prod_{i \in [k]} \ell_i(y_{[k] \setminus \{i\}})\]
for some functions $\ell_i$ with $|\ell_i(y_{[k] \setminus \{i\}})| = 1$ for all $y_{[k] \setminus \{i\}}$. Let $\tilde{f} \colon U \to \mathbb{D}$ be given by $\tilde{f}(x) = f(x + \tilde{b})$.\\

Thus 
\begin{align*}c \leq &\exx_{x \in U} \|f_{I, x}\|_{\square^k}^{2^k} = \exx_{\ssk{x, a_1, \dots, a_k \in U\\b_1, \dots, b_k \in U}} \prod_{I \subseteq [k]} \Big(\on{Conj}^{k-|I|} \tilde{f}\Big(x + \sum_{i \in I} a_i + \sum_{i \notin I} b_i\Big) (-1)^{\alpha(a_I, b_{[k] \setminus I})}\Big)\\
&\hspace{8cm} = \exx_{x, d_1, \dots, d_k \in U} \mder_{d_1} \dots \mder_{d_k}  \tilde{f}(x) (-1)^{\alpha(d_1, \dots, d_k)}.\qedhere\end{align*}

\end{proof}
%

Another fact that we need is a combination of the symmetry argument of Green and Tao~\cite{StrongU3} and the inverse theorem for biased multilinear forms.

\begin{lemma}[Symmetry argument, Corollary 97~\cite{multihomPaper}]\label{symmArgumentCor}Let $\alpha \colon G^k \to \mathbb{F}_2$ be a multilinear form. Suppose that
\[\Big|\exx_{a_1, \dots, a_k, x\in G} \mder_{a_1} \dots \mder_{a_{k}} f(x) (-1)^{\alpha(a_1, \dots, a_k)}\Big| \geq c\]
for some $c > 0$. Then for each $\pi \in \on{Sym}_{[k]}$ we have $\on{prank} (\alpha + \alpha \circ \pi) \leq  O\Big((\log c^{-1})^{O(1)}\Big)$.\end{lemma}

Recall from the introduction that a multilinear form $\sigma \colon G^k \to \mathbb{F}_2$ is \emph{strongly symmetric} if it is symmetric and the multilinear form $(x,y_3, \dots, y_k) \mapsto  \sigma(x,x, y_3, \dots, y_k)$ is also symmetric. The following lemma is a crucial property of strongly symmetric multilinear forms, namely that we may lift them to higher order strongly symmetric multilinear forms.

\begin{lemma}\label{liftingssforms}Suppose that $\sigma \colon G^k \to \mathbb{F}_2$ is a strongly symmetric multilinear form. Then there exists a strongly symmetric multilinear form $\tilde{\sigma} \colon G^{k + 1} \to \mathbb{F}_2$ such that 
\[\sigma(x_1, x_2, \dots, x_k) = \tilde{\sigma}(x_1, x_1,x_2, \dots, x_k).\]\end{lemma}

\begin{proof}Let $n = \dim G$ and fix a basis of $G$. Write $\beta$ as a linear combination of monomials. Thus
\[\sigma(x_1, x_2, \dots, x_k) = \sum_{i_1, \dots, i_k \in [n]} \lambda_{i_1, \dots, i_k} x_{1\,i_1} \dots x_{k\,i_k}.\]
Define coefficients $\mu_{j_1, \dots, j_{k+1}}$ as follows: if $j_1, \dots, j_{k+1}$ are all distinct, then put $\mu_{j_1, \dots, j_{k+1}} = 0$, and in the other case set $\mu_{j_1, \dots, j_{k+1}} = \lambda_{i_1, \dots, i_k}$ for any sequence $i_1, \dots, i_k$ with the property that there is some element $i_\ell$ which, if repeated, gives a permutation of the sequence $j_1, \dots, j_{k+1}$. We need to check that this is well-defined. Suppose that $i'_1, \dots, i'_k$ is another such sequence, and that $i'_m$ is the element that is repeated. If $i'_m = i_\ell$ then $i$ and $i'$ are the same up to reordering, so by symmetry of $\sigma$ we have the equality $\lambda_{i_1, \dots, i_k} = \lambda_{i'_1, \dots, i'_k}$, as desired. On the other hand, if $i_\ell \not= i'_m$, then that means that the number of times the element $i'_m$ appears in the sequence $i$ is one greater than the number of times it appears in the sequence $i'$, and, similarly, $i_\ell$ has one occurrence more in $i'$ than it has in $i$, while other elements appear an equal number of times in both sequences. Since $\sigma$ is strongly symmetric, we again have equality $\lambda_{i_1, \dots, i_k} = \lambda_{i'_1, \dots, i'_k}$.\\
\indent Coefficients $\mu$ are symmetric directly from the definition, so defining $\tilde{\sigma}$ as 
\[\tilde{\sigma}(x_0, x_1, x_2, \dots, x_k) = \sum_{i_0, i_1, \dots, i_k \in [n]} \mu_{i_0, i_1, \dots, i_k} x_{0\,i_0} x_{1\,i_1} \dots x_{k\,i_k}\]
produces a symmetric multilinear form. Finally, 
\[\tilde{\sigma}(x_1, x_1, x_2, \dots, x_k) = \sum_{i_1, \dots, i_k \in [n]} \mu_{i_1, i_1, \dots, i_k} x_{1\,i_1} \dots x_{k\,i_k} = \sum_{i_1, \dots, i_k \in [n]} \lambda_{i_1, i_2, \dots, i_k} x_{1\,i_1} \dots x_{k\,i_k} = \sigma(x_1, x_2, \dots, x_k).\]
This means that $\tilde{\sigma}(x_1, x_1, x_2, \dots, x_k)$ is symmetric, thus $\tilde{\sigma}$ is strongly symmetric.\end{proof}

Let us now give the formal definition of non-classical polynomials. Recall that $G = \mathbb{F}_2^n$ for some $n$. We write $\mathbb{T}$ for the circle group $\mathbb{R}/\mathbb{Z}$.

\begin{defin}[Definition 1.2~\cite{TaoZiegler}]The \emph{discrete additive derivative operator} $\Delta_a$ for shift $a \in G$ is defined by $\Delta_a f(x) = f(x + a) - f(x)$ for functions $f \colon G \to \mathbb{T}$. A function $q \colon G \to \mathbb{T}$ is said to be a \emph{non-classical polynomial of degree at most $d$} if one has
\[\Delta_{a_1}\dots\Delta_{a_{d+1}}q(x) = 0\]
for all $x, a_1, \dots, a_{d+1} \in G$.\end{defin}

A basic fact about non-classical polynomials is that they have an explicit description in terms of monomials. The notation $|\cdot|_2$ stands for the map from $\mathbb{F}_2$ to $\mathbb{R}$ which sends $0$ to $0$ and $1$ to $1$.

\begin{lemma}[Lemma 1.7(iii)~\cite{TaoZiegler}]A function $q \colon \mathbb{F}_2^n \to \mathbb{T}$ is a non-classical polynomial of degree at most $d$ if and only if it has a representation of the form
\[q(x_1, \dots, x_n) = \alpha + \sum_{\ssk{0 \leq i_1, \dots, i_n \leq 1; j\geq 0\\0 < i_1 + \dots + i_n \leq d - j}} \frac{c_{i_1, \dots, i_n, j}|x_1|_2^{i_1}\dots |x_n|_2^{i_n}}{2^{j+1}} + \mathbb{Z}\]
for some coefficients $c_{i_1, \dots, i_n, j} \in \{0,1\}$ and $\alpha \in \mathbb{T}$. Furthermore the coefficients $c_{i_1, \dots, i_n, j}$ and $\alpha $ are unique.
\end{lemma}

We do not use the lemma above in the paper, but we opted to include it for completeness. The way we obtain non-classical polynomials in this paper is through the following result.

\begin{lemma}[Tidor, Proposition 3.5~\cite{Tidor}]\label{ssformintegration}Given a strongly symmetric multilinear form $\sigma \colon G^k \to \mathbb{F}_2$, there exists a non-classical polynomial $q \colon G \to \mathbb{T}$ of degree at most $k$ such that
\[\Delta_{a_1} \dots \Delta_{a_k} q(x) = \frac{|\sigma(a_1, \dots, a_k)|_2}{2} + \mathbb{Z}\]
for all $a_1, \dots, a_k, x \in G$.\end{lemma}

\section{Properties of low partition rank decompositions}


We begin our work in this section with a lemma which allows us to conclude that a low partition rank decomposition which evaluates to 0 is necessarily trivial under the appropriate assumptions.

\begin{lemma}\label{zeromlformindexlemma}For given $k$ there exist constants $C = C_k \geq 1$ and $D = D_k \geq 1$ such that the following holds. Let $I_1 \cup \dots \cup I_r = [k]$ be a partition of the set $[k]$. Suppose that, for each $i \in [r]$, we are given a multilinear map $\alpha_i \colon G^{I_i} \to \mathbb{F}_2^{d_i}$ such that for each non-zero $\lambda \in \mathbb{F}_2^{d_i}$ we have $\on{prank} (\lambda \cdot \alpha_i) \geq r_i$. Let $\Pi_1, \dots, \Pi_s \colon G^k \to \mathbb{F}_2$ be functions such that each $\Pi_i(x_{[k]})$ is a product of multilinear forms of the shape $\gamma_1(x_{J_1}) \dots \gamma(x_{J_t})$ such that some $J_i$ does not contain any of the sets $I_1, \dots, I_r$. Suppose that there are some scalars $\lambda_{i_1, \dots, i_r} \in \mathbb{F}_2$ for $i_j \in [d_j]$ such that
\begin{equation}\sum_{i_1 \in [d_1], \dots, i_r \in [d_r]} \lambda_{i_1, \dots, i_r} \alpha_{1, i_1}(x_{I_1}) \dots \alpha_{r, i_r}(x_{I_r}) + \sum_{i \in [s]} \Pi_i(x_{[k]}) = 0\label{zerolcmlforms}\end{equation}
for all $x_{[k]}$. Provided $r_i \geq C(s + d_i)^D$ for each $i \in [r]$, we have $\lambda_{i_1, \dots, i_r} = 0$ for all indices $i_1, \dots, i_r$.\end{lemma}

\noindent\textbf{Remark.} Observe that we do not assume that factors of $\Pi_i$ have disjoint sets of variables.

\begin{proof}Note that for each $i \in [s]$ we have a multilinear form $\gamma_i(x_{J_i})$ that appears as a factor of $\Pi_i(x_{[k]})$ and $J_i$ contains none of the sets $I_1, \dots, I_r$. Fix indices $i_1 \in [d_1], \dots, i_r \in [d_r]$. We shall find $x_{[k]}$ so that all $\gamma_i(x_{J_i})$ vanish (implying $\Pi_i(x_{[k]}) = 0$), $\alpha_{j, i_j}(x_{I_j}) = 1$ for each $j \in [r]$ and $\alpha_{j, \ell}(x_{I_j}) = 0$ for each $j \in [r]$ and $\ell \not= i_j$. To find such an $x_{[k]}$, we first find suitable $x_{I_1}$, then $x_{I_2}$, etc. Assuming that $x_{I_1 \cup \dots \cup I_{j-1}}$ has been fixed for some $j \in [r]$, we look for $x_{I_j}$ such that $\alpha_{j, i_j}(x_{I_j}) = 1$, $\alpha_{j, \ell}(x_{I_j}) = 0$ for $\ell \not= i_j$ and $\gamma_i(x_{J_\ell}) = 0$ for all indices $\ell$ such that $J_\ell \subseteq I_1 \cup \dots \cup I_j$ and $J_\ell \cap I_j \not= \emptyset$. By Theorem~\ref{biasedinversethm} there exist quantities $C_k, D_k$, depending only on $k$, such that $\on{prank} \tau \geq C t^D$ implies $\on{bias} \tau < 2^{-t}$ for all multilinear forms $\tau$ in at most $k$ variables and all positive integers $t$. Thus, $\on{bias}\mu \cdot \alpha_j < 2^{-k(d_i +  s)}$ for all $\mu\not=0$ and we can find the desired $x_{I_j}$ using Lemma~\ref{nonzeroPtLemma}. We conclude that $\lambda_{i_1, \dots, i_r} = 0$ from~\eqref{zerolcmlforms}.\end{proof}

The following technical lemma also concerns linear combinations of mutlilinear forms, but is entirely linear-algebraic. It allows us to change a basis of multilinear forms to a more convenient one.

\begin{lemma}\label{mlformschbasislemma}Suppose that $I \cup J = [k]$ is a partition and that $\beta_1, \dots, \beta_r \colon G^I \to \mathbb{F}_2$ and $\gamma_1, \dots, \gamma_r \colon G^J \to \mathbb{F}_2$ are multilinear forms. Then we may find further mutlilinear forms $\tilde{\beta}_1, \dots, \tilde{\beta}_s \colon G^I \to \mathbb{F}_2$ and $\tilde{\gamma}_1, \dots, \tilde{\gamma}_s \colon G^J \to \mathbb{F}_2$ for some $s \leq r$ so that
\begin{itemize}
\item[\textbf{(i)}] for each $i \in [s]$, the form $\tilde{\beta}_i$ is a linear combination of forms $\beta_1, \dots, \beta_r$ and the form $\tilde{\gamma}_i$ is a linear combination of forms $\gamma_1, \dots, \gamma_r$,
\item[\textbf{(i)}] for all $x_{[k]}$ we have 
\[\sum_{i \in [r]} \beta_i(x_I) \gamma_i(x_J) = \sum_{i \in [s]} \tilde{\beta}_i(x_I) \tilde{\gamma}_i(x_J),\]
\item[\textbf{(ii)}] for each $i \in [s]$ there exists $x_J \in G^J$ such that $\tilde{\gamma}_j(x_J) = \id(i = j)$.
\end{itemize}\end{lemma}

\begin{proof}Let $v_1, \dots, v_s \in \mathbb{F}_2^r$ be a maximal independent set of elements in $\on{Im} \gamma \subseteq \mathbb{F}_2^r$. We may extend this sequence by $v_{s + 1}, \dots, v_r$ to obtain a basis of $\mathbb{F}_2^r$. Let $\tilde{\beta}_i = \sum_{j \in [r]} v_{i,j} \beta_j$ for each $i \in [r]$. Since the matrix $V = (v_{i,j})_{i,j \in [r]}$ is invertible, we may find its inverse $M = (\mu_{i,j})_{i,j \in [r]}$. Set $\tilde{\gamma}_i = \sum_{j \in [r]} \mu_{j,i} \gamma_j$. Property \textbf{(i)} holds trivially. We claim that when $i > s$ we have $\tilde{\gamma}_i(x_J) = 0$ for all $x_J \in G^J$. To see that, recall first that $\gamma(x_J) \in \langle v_1, \dots, v_s \rangle$, so there are scalars $\lambda_1, \dots, \lambda_s$ such that $\gamma(x_J) = \sum_{j \in [s]} \lambda_j v_j$.\\
\indent Next, observe that 
\[\tilde{\gamma}_i(x_J) = \sum_{j \in [r]} \mu_{j,i} \gamma_j(x_J) =  \sum_{j \in [r]} \mu_{j,i} \sum_{\ell \in [s]} \lambda_\ell v_{\ell, j} = \sum_{\ell \in [s]} \lambda_\ell \Big(\sum_{j \in [r]}  v_{\ell, j}\mu_{j,i}\Big) = \sum_{\ell \in [s]} \lambda_\ell \id(i = \ell) = 0,\]
as claimed. Hence, for all $x_{[k]}$ we have
\begin{align*}\sum_{i \in [s]} \tilde{\beta}_i(x_I) \tilde{\gamma}_i(x_J) = &\sum_{i \in [r]} \tilde{\beta}_i(x_I) \tilde{\gamma}_i(x_J) = \sum_{i \in [r]}\Big(\sum_{j \in [r]} v_{i,j} \beta_j(x_I)\Big)\Big(\sum_{\ell \in [r]} \mu_{\ell,i} \gamma_\ell(x_J)\Big) \\
= &\sum_{j \in [r]} \sum_{\ell \in [r]}\beta_j(x_I)\gamma_\ell(x_J) \Big(\sum_{i \in [r]} \mu_{\ell,i}v_{i,j}\Big) = \sum_{j \in [r]} \sum_{\ell \in [r]}\beta_j(x_I)\gamma_\ell(x_J)  \id(\ell = j)\\
 = &\sum_{j \in [r]} \beta_j(x_I)\gamma_j(x_J),\end{align*}
proving property \textbf{(ii)}.\\
\indent For property \textbf{(iii)}, for given $i \in [s]$, take $x_J$ such that $\gamma(x_J) = v_i$. Let $j \in [s]$. Then
\[\tilde{\gamma}_j(x_J) = \sum_{\ell \in [r]} \mu_{\ell, j} \gamma_\ell(x_J) = \sum_{\ell \in [r]} \mu_{\ell, j} v_{i, \ell} = \id(i = j).\qedhere\]
\end{proof}

\indent Using the mutlilinear forms change of basis lemma, we prove that for a given multilinear form $\phi$ of low partition rank, the lower-order forms in a decomposition of $\phi$ essentially come from $\phi$. We want to keep track of the structure of low partition rank decomposition and to that end we need to define a partial order on partitions. Given two partitions $A_1 \cup \dots \cup A_r = B_1 \cup \dots \cup B_s = [k]$, where all sets are non-empty, we write $A_{[r]} \leq B_{[s]}$ if every set $B_i$ is a union of some sets among $A_1, \dots, A_r$. Thus, the trivial partition $\{[k]\}$ is the maximum element of this partially ordered set and $\{\{1\}, \dots, \{k\}\}$ is the minimum element. A set $\mathcal{P}$ of partitions of $[k]$ is a \emph{down-set} if it has the property that whenever $A_{[r]} \leq B_{[s]}$ and $B_{[s]} \in \mathcal{P}$, then $A_{[r]} \in \mathcal{P}$.

\begin{proposition}\label{lcmlformsdecomppropo}Let $\phi \colon G^k \to \mathbb{F}_2$ be a multilinear form. Let $\mathcal{P}$ be a down-set of partitions. Suppose that
\[\phi(x_{[k]}) = \sum_{i \in [r]} \beta_{i, 1}(x_{I_{i, 1}}) \dots \beta_{i, d_i}(x_{I_{i, d_i}})\]
where each partition $\{I_{i,1}, \dots, I_{i, d_i}\}$ belongs to $\mathcal{P}$. Then there exists another decomposition
\[\phi(x_{[k]}) = \sum_{i \in [r']} \beta'_{i, 1}(x_{I'_{i, 1}}) \dots \beta'_{i, d'_i}(x_{I'_{i, d'_i}})\]
where $r' \leq O(r^{O(1)})$, such that each partition $\{I'_{i, 1}, \dots, I'_{i, d'_i}\}$ also belongs to $\mathcal{P}$ and each $\beta'_{i,j}(x_{I'_{i,j}})$ is of the shape $\phi(x_{I'_{i,j}}, y_{[k] \setminus I'_{i,j}})$ for some fixed $y_{[k] \setminus I'_{i,j}}$.\end{proposition}

\begin{proof}The proposition will follow from the next claim.

\begin{claim}For each non-empty down-set\footnote{The collection of sets $\mathcal{U}$ is a down-set in the usual sense, namely a collection of sets closed under taking subsets.} $\mathcal{U} \subseteq \mathcal{P}[k]$, there exists a decomposition
\begin{equation}\phi(x_{[k]}) = \sum_{i \in [r']} \beta'_{i, 1}(x_{I'_{i, 1}}) \dots \beta'_{i, d'_i}(x_{I'_{i, d'_i}})\label{claimrestrictiondecomposition}\end{equation}
where $r' \leq O(r^{O(1)})$, such that each partition $\{I'_{i, 1}, \dots, I'_{i, d'_i}\}$ belongs to $\mathcal{P}$ and if $I'_{i,j} \notin \mathcal{U}$, then $\beta'_{i,j}(x_{I'_{i,j}})$ is of the shape $\phi(x_{I'_{i,j}}, y_{[k] \setminus I'_{i,j}})$ for some fixed $y_{[k] \setminus I'_{i,j}}$.\end{claim}

\begin{proof}[Proof of the claim]We prove the claim by removing a maximal set from $\mathcal{U}$. The starting situation is when $\mathcal{U} = \mathcal{P}[k]$, in which case the claim is trivial. Suppose now that we have decomposition~\eqref{claimrestrictiondecomposition} for a given down-set $\mathcal{U}$, and let $I \in \mathcal{U}$ be a maximal set. Our goal is to remove $I$ from $\mathcal{U}$, that is to find a decomposition like that in~\eqref{claimrestrictiondecomposition} that has the desired properties for $\mathcal{U} \setminus \{I\}$.\\
\indent We may split the terms in decomposition~\eqref{claimrestrictiondecomposition} as
\begin{align}\phi(x_{[k]}) = \sum_{i \in [r_1]} \beta^{(1)}_{i, 1}(x_{I^{(1)}_{i, 1}}) \dots \beta^{(1)}_{i, d^{(1)}_i}(x_{I^{(1)}_{i, d^{(1)}_i}}) + &\sum_{i \in [r_2]} \beta^{(2)}_{i, 1}(x_{I^{(2)}_{i, 1}}) \dots \beta^{(2)}_{i, d^{(2)}_i}(x_{I^{(2)}_{i, d^{(2)}_i}})\nonumber\\
+ &\sum_{i \in [r_3]} \beta^{(3)}_{i, 1}(x_{I^{(3)}_{i, 1}}) \dots \beta^{(3)}_{i, d^{(3)}_i}(x_{I^{(3)}_{i, d^{(3)}_i}}),\label{claimrestrictiondecomposition2}\end{align}
where the terms $\beta^{(1)}_{i, 1}(x_{I^{(1)}_{i, 1}}) \dots \beta^{(1)}_{i, d^{(1)}_i}(x_{I^{(1)}_{i, d^{(1)}_i}})$ have the property that $I^{(1)}_{i, 1} = I$ for all $i \in [r_1]$, the terms $\beta^{(2)}_{i, 1}(x_{I^{(2)}_{i, 1}}) \dots \beta^{(2)}_{i, d^{(2)}_i}(x_{I^{(2)}_{i, d^{(2)}_i}})$ have the property that $I \subsetneq I^{(2)}_{i, 1}$ for all $i \in [r_2]$, and the terms $\beta^{(3)}_{i, 1}(x_{I^{(3)}_{i, 1}}) \dots \beta^{(3)}_{i, d^{(3)}_i}(x_{I^{(3)}_{i, d^{(3)}_i}})$ have the property that $I$ is not contained in any of the sets $I^{(3)}_{i, 1}, \dots, I^{(3)}_{i, d^{(3)}_i}$ for $i \in [r_3]$. For $i \in [r_1]$, let $\gamma_i$ be the multilinear form defined by $\gamma_i(x_{[k] \setminus I}) = \beta^{(1)}_{i, 2}(x_{I^{(1)}_{i, 2}}) \dots \beta^{(1)}_{i, d^{(1)}_i}(x_{I^{(1)}_{i, d^{(1)}_i}})$. Decomposition~\eqref{claimrestrictiondecomposition2} becomes
\begin{align}\phi(x_{[k]}) = \sum_{i \in [r_1]} \beta^{(1)}_{i, 1}(x_I) \gamma_i(x_{[k] \setminus I}) + &\sum_{i \in [r_2]} \beta^{(2)}_{i, 1}(x_{I^{(2)}_{i, 1}}) \dots \beta^{(2)}_{i, d^{(2)}_i}(x_{I^{(2)}_{i, d^{(2)}_i}})\nonumber\\
+ &\sum_{i \in [r_3]} \beta^{(3)}_{i, 1}(x_{I^{(3)}_{i, 1}}) \dots \beta^{(3)}_{i, d^{(3)}_i}(x_{I^{(3)}_{i, d^{(3)}_i}}).\nonumber\end{align}

Applying Lemma~\ref{mlformschbasislemma}, we obtain a positive integer $s \leq r_1$ and multilinear forms $\tilde{\beta}_1, \dots, \tilde{\beta}_s \in \langle  \beta^{(1)}_{i, 1} \colon i \in [r_1]\rangle$ and $\tilde{\gamma}_1, \dots, \tilde{\gamma}_s \in \langle \gamma_i \colon i \in [r_1]\rangle$ such that
\begin{align}\phi(x_{[k]}) = \sum_{i \in [s]} \tilde{\beta}_i(x_I) \tilde{\gamma}_i(x_{[k] \setminus I}) + &\sum_{i \in [r_2]} \beta^{(2)}_{i, 1}(x_{I^{(2)}_{i, 1}}) \dots \beta^{(2)}_{i, d^{(2)}_i}(x_{I^{(2)}_{i, d^{(2)}_i}})\nonumber\\
+ &\sum_{i \in [r_3]} \beta^{(3)}_{i, 1}(x_{I^{(3)}_{i, 1}}) \dots \beta^{(3)}_{i, d^{(3)}_i}(x_{I^{(3)}_{i, d^{(3)}_i}})\label{claimrestrictiondecomposition3}\end{align}
and for each $i \in [s]$ there exists $x_{[k] \setminus I} \in G^{[k] \setminus I}$ such that $\tilde{\gamma}_j(x_{[k] \setminus I}) = \id(i = j)$. Fix arbitrary $i \in [s]$ and take $y_{[k] \setminus I}$ such that $\tilde{\gamma}_j(y_{[k] \setminus I}) = \id(i = j)$. Evaluating decomposition~\eqref{claimrestrictiondecomposition3} at $(x_I, y_{[k] \setminus I})$ gives
\[\tilde{\beta}_i(x_I) = \phi(x_I, y_{[k] \setminus I}) + \sum_{i \in [r_2]} \beta^{(2)}_{i, 1}(x_I, y_{I^{(2)}_{i, 1} \setminus I}) + \psi(x_I)\]
for some multilinear form $\psi$ of partition rank at most $r_3$. Since index set of arguments of the form $\beta^{(2)}_{i, 1}$ does not belong to $\mathcal{U}$, we know that $\beta^{(2)}_{i, 1}(x_I, y_{I^{(2)}_{i, 1} \setminus I}) = \phi(x_I, z_{[k] \setminus I})$ for a suitable $z_{[k] \setminus I}$. Replacing each $\tilde{\beta}_i$ with a sum of at most $r_2 + 1$ forms coming from slices of $\phi$ and a form of partition rank at most $r_3$, and recalling that each $\tilde{\gamma}_i \in \langle \gamma_i \colon i \in [r_1]\rangle$, we conclude that 
\[\phi(x_{[k]}) = \sum_{i \in [r'']} \beta''_{i, 1}(x_{I''_{i, 1}}) \dots \beta''_{i, d''_i}(x_{I''_{i, d''_i}}),\]
where $r'' \leq r^2_1r_2 + r^2_1r_3 + r^2_1 + r_2 + r_3 \leq O(r^{O(1)})$, such that each partition $\{I''_{i, 1}, \dots, I''_{i, d''_i}\}$ belongs to $\mathcal{P}$ and if $I''_{i,j} \notin (\mathcal{U} \setminus \{I\})$, then $\beta''_{i,j}(x_{I''_{i,j}})$ is of the shape $\phi(x_{I''_{i,j}}, y_{[k] \setminus I''_{i,j}})$ for some fixed $y_{[k] \setminus I''_{i,j}}$, completing the proof of the step of the procedure. The procedure terminates after $2^k - 1$ steps when the collection $\mathcal{U}$ becomes $\{\emptyset\}$.\end{proof}

Use the case $\mathcal{U} = \{\emptyset\}$ of the claim above, which is equivalent to the proposition.\end{proof}

\section{Approximately symmetric multilinear forms}

This section is devoted to proofs of Theorems~\ref{approximateSymmetry2},~\ref{mainSymmExtn4} and~\ref{mainSymmExtn5}. In order to prove these theorems, we need a weak regularity lemma for multilinear forms that takes into account the symmetry properties of the given forms. Namely, we start with multilinear forms in variables $x_{[k]}$ that are symmetric in the first $m$ variables and are interested in how far the given forms are from being symmetric in variables $x_{[m+1]}$. The lemma allows us to express the given forms in terms of further multilinear forms that are also symmetric in the first $m$ variables and have one of the following three additional properties:
\begin{itemize}
\item additional form $\alpha(x_{[k]})$ has the property that $\on{prank} \Big(\alpha + \alpha \circ (m \,\,m+1) \Big)$ is small (this is the \emph{almost symmetric} case),
\item additional form $\alpha(x_{[k]})$ has the property that $\on{prank} \Big(\alpha + \alpha \circ (m \,\,m+1) \Big)$ is large, but $\on{prank} \Big(\sum_{i \in [m+1]} \alpha \circ (i \,\,m+1)\Big)$ is small (this is the \emph{partially symmetric} case),
\item additional form $\alpha(x_{[k]})$ has the property that the partition rank of non-zero linear combinations of forms $\alpha \circ (i \,\,m+1)$ are large (this is the \emph{asymmetric} case).
\end{itemize}
We shall denote the forms in almost symmetric, partially symmetric and asymmetric cases with greek letters $\sigma$, $\pi$ and $\rho$, respectively.\\
\indent For a positive quantity $R$, we write $\on{PR}_{\leq R}$ for the collection of all multilinear forms of the partition rank at most $R$.

\begin{lemma}\label{symmRespWRegLemma}Suppose that we are given some forms $\gamma_1, \dots, \gamma_r \colon G^k \to \mathbb{F}_2$ which are symmetric in the first $m$ variables, where $k \geq m + 1$. Let $C, D \geq 2$ be given.\footnote{The lower bounds $C, D \geq 2$ are here in order to simplify the calculation on the final bound on the quantity $R$.} Then there exist a positive integer $R \leq C^{D^{O_k(r)}}$, further forms $\sigma_1, \dots, \sigma_s, \pi_1, \dots, \pi_q, \rho_1, \dots, \rho_t$, where $s + q + t \leq r$, such that
\begin{itemize}
\item[\textbf{(i)}] $\sigma_i, \pi_i, \rho_i$ are linear combinations of forms $\gamma_1, \dots, \gamma_r$,
\item[\textbf{(ii)}] $\on{prank}\Big(\sigma_i + \sigma_i \circ (m\,\,m+1)\Big) \leq R$,
\item[\textbf{(iii)}] $\pi_i$ has the property that $\on{prank}\Big(\sum_{\ell \in [m + 1]} \pi_i \circ (\ell\,\,m + 1)\Big) \leq R$,
\item[\textbf{(iv)}] any linear combination of forms $\sigma_i$ for $i \in [s]$, $\pi_i \circ (\ell\,\,m + 1)$ for $i \in [q], \ell \in [m + 1]$ and $\rho_i \circ (\ell\,\, m+1)$ for $i \in [t], \ell \in [m + 1]$ has partition rank at least $(C(R + 2r))^D$, unless the non-zero coefficients appear only next to the forms $\pi_i \circ (\ell\,\,m + 1)$ and do not depend on $\ell$ for each $i$, making it a linear combination of sums $\sum_{\ell \in [m + 1]} \pi_i \circ (\ell\,\,m + 1)$ for $i \in [q]$,
\item[\textbf{(v)}] for all $j \in [r]$ we have
\[\gamma_j \in \langle \sigma_i \colon i \in [s]\rangle + \langle \pi_i \circ (\ell\,\,m + 1) \colon i \in [q], \ell \in [m + 1]\rangle + \langle \rho_i \circ (\ell\,\,m + 1) \colon i \in [t], \ell \in [m + 1]\rangle + \on{PR}_{\leq R}.\]
\end{itemize}
\end{lemma}

\begin{rem}\label{symmRespWRegLemmaRemark}When $m = 1$, we may rename the forms $\pi_{[q]}$ to $\sigma_{[s + 1, s + q]}$, as conditions \textbf{(ii)} and \textbf{(iii)} become the same. In the proof of Theorems~\ref{mainSymmExtn4} and~\ref{mainSymmExtn5} we shall also make use of $m = 0$ case, in which case the symmetry properties of forms play no role and we end up only with forms $\sigma_{[s]}$ which satisfy conditions \textbf{(iv)} and \textbf{(v)}.\end{rem}

\begin{proof} Let us begin by setting $t = r$, $q = 0, s= 0$, $\rho_1 = \gamma_1, \dots, \rho_t = \gamma_t$ and $R = 1$, which satisfies all conditions except possibly the fourth. At each step we modify the sequence of forms $\sigma_{[s]}, \pi_{[q]}, \rho_{[t]}$, decreasing the quantity $Q = s + 10q + 100t$ and preserving all properties but the fourth. The quantity $R$ will increase at each step as well.\\
\indent Suppose that at some step the fourth condition still fails. Hence a linear combination
\begin{equation}\label{lowrankcombpartsymmlemma}\sum_{i \in [s]} \lambda_i \sigma_i + \sum_{i \in [q] ,\ell \in [m+1]} \mu_{i, \ell} \pi_i \circ (\ell \,\, m + 1) + \sum_{i \in [t] ,\ell \in [m+1]} \nu_{i, \ell} \rho_i \circ (\ell \,\,m + 1)\end{equation}
has partition rank at most $(C(R + 2r))^D$ and it is not a linear combination of sums $\sum_{\ell \in [m + 1]} \pi_i \circ (\ell\,\,m + 1)$. We consider two separate cases depending on the behaviour of the coefficients in the above expression.\\

\indent \textbf{Case 1. For each $i \in [q]$ the coefficients $\mu_{i, \ell}$ do not depend on $\ell$ and for each $i \in [t]$ the coefficients $\nu_{i, \ell}$ do not depend on $\ell$.} By assumptions in this case there are coefficients $\mu_i$ and $\nu_i$ such that
\[\on{prank}\bigg(\sum_{i \in [s]} \lambda_i \sigma_i + \sum_{i \in [q]} \mu_i \Big(\sum_{\ell \in [m+1]} \pi_i \circ (\ell \,\, m + 1)\Big) + \sum_{i \in [t]} \nu_{i} \Big(\sum_{\ell \in [m+1]} \rho_i \circ (\ell \,\,m + 1)\Big)\bigg) \leq (C(R + 2r))^D.\]
If there is some $\lambda_i \not= 0$ we may simply remove $\sigma_i$ from the list $\sigma_{[s]}$, which results in decreasing the quantity $Q$, noting that the fifth condition still holds as long as we replace $R$ by $R + (C(R + 2r))^D$. Otherwise, assume that all $\lambda_i = 0$. Therefore, some $\nu_j \not= 0$. Remove $\rho_j$ from the list $\rho_{[t]}$ and add new form
\[\pi_{q + 1} = \sum_{i \in [q]} \mu_i \pi_i + \sum_{i \in [t]} \nu_{i} \rho_i.\]
This decreases $t$ by 1, increases $q$ by 1 and thus the quantity $Q$ decreases as well. All conditions except the fourth one are still satisfied, provided we replace $R$ by $R + (C(R + 2r))^D$.\\

\indent \textbf{Case 2. There exist indices $\ell, \ell' \leq m+1$ and $i$ such that $\mu_{i, \ell} = 1, \mu_{i, \ell'} = 0$ or  $\nu_{i, \ell} = 1, \nu_{i, \ell'} = 0$.} Let $\phi$ be the linear combination of forms in~\eqref{lowrankcombpartsymmlemma}. Note that
\begin{align}\phi(x_{[k]}) + \phi \circ (\ell \,\,\ell')(x_{[k]}) = &\sum_{i \in [s]} \lambda_i \Big(\sigma_i + \sigma_i \circ (\ell \,\,\ell')\Big) + \sum_{i \in [q] , a \in [m+1]} \mu_{i, a} \Big(\pi_i \circ (a \,\, m + 1) + \pi_i \circ (a \,\, m + 1) \circ (\ell \,\,\ell')\Big)\nonumber\\
&\hspace{2cm} + \sum_{i \in [t] ,a \in [m+1]} \nu_{i, a} \Big(\rho_i \circ (a \,\,m + 1) + \rho_i \circ (a \,\,m + 1) \circ (\ell \,\,\ell')\Big)\label{lowrankcombpartsymmlemma2}\end{align}
and has partition rank at most $2(C(R + 2r))^D$. Using the symmetry properties, we obtain 
\begin{align*}\on{prank} \bigg(& \sum_{i \in [q]} (\mu_{i, \ell} + \mu_{i, \ell'})\Big(\pi_i \circ (\ell \,\, m + 1) + \pi_i \circ (\ell' \,\, m + 1)\Big) \\
&\hspace{2cm}+ \sum_{i \in [t]}(\nu_{i, \ell} + \nu_{i, \ell'})\Big(\rho_i \circ (\ell \,\, m + 1) + \rho_i \circ (\ell' \,\, m + 1)\Big)\bigg) \leq 2(C(R + 2r))^D + sR.\end{align*}

Define multilinear form $\tilde{\sigma}$ as 
\begin{equation}\tilde{\sigma}(x_{[k]}) = \sum_{i \in [q]} (\mu_{i, \ell} + \mu_{i, \ell'})\pi_i(x_{[k]}) + \sum_{i \in [t]}(\nu_{i, \ell} + \nu_{i, \ell'})\rho_i(x_{[k]}).\label{tildesigmasymmreglemmadefin}\end{equation}

This form satisfies $\on{prank} (\tilde{\sigma} \circ (\ell \,\, m + 1) + \tilde{\sigma} \circ (\ell' \,\, m + 1))  \leq 2(C(R + 2r))^D + sR$ and is a linear combination of forms $\gamma_1, \dots, \gamma_r$. It follows that $\on{prank} (\tilde{\sigma} + \tilde{\sigma} \circ (m \,\, m + 1))  \leq 2(C(R + 2r))^D + sR$.\\

By assumptions in this case of the proof we see that $\tilde{\sigma}$ is a non-trivial linear combination of the forms in~\eqref{tildesigmasymmreglemmadefin}. Thus, we may remove a form $\pi_i$ or $\rho_i$ that appears with non-zero coefficient in~\eqref{tildesigmasymmreglemmadefin} from its list, set $\sigma_{s+1} = \tilde{\sigma}$ and replace $R$ by 
\begin{equation}(m+1)\Big(2(C(R + 2r))^D + sR\Big) + R\label{rchangeeqn}\end{equation}
to make sure that the fifth condition is still satisfied. As $Q$ decreases in this case as well, the proof is complete after at most $100r$ steps (this was this the initial value of the quantity $Q$). The bound on $R$ is given by starting from $R = 1$ and replacing $R$ by the value in~\eqref{rchangeeqn} at most $100r$ times.\end{proof}

In fact, we can say more about the linear combinations of the additional forms in Lemma~\ref{symmRespWRegLemma}.

\begin{observation}\label{symmetriclinearcombinations}Let $\sigma_{[s]}, \pi_{[q]}, \rho_{[t]}, C, D, R$ be as in the previous lemma and let $\tau$ be a multilinear form of partition rank at most $R$. (Recall that we assume that $C, D \geq 2$.) Suppose that 
\[\phi = \sum_{i \in [s]} \lambda_i \sigma_i + \sum_{i \in [q] ,\ell \in [m+1]} \mu_{i, \ell} \pi_i \circ (\ell \,\, m + 1) + \sum_{i \in [t] ,\ell \in [m+1]} \nu_{i, \ell} \rho_i \circ (\ell \,\, m + 1) + \tau\]
is symmetric in the first $m$ variables. Then we we have $\mu_{i, \ell} = \mu_{i, \ell'}$ and $\nu_{i, \ell} = \nu_{i, \ell'}$ for all $\ell, \ell' \in [m]$.\\
\indent Moreover, if $\phi$ is symmetric in the first $m+1$ variables, then 
\[\phi \in  \sum_{i \in [s]} \lambda_i \sigma_i + \sum_{i \in [t]} \tilde{\nu}_i \Big(\sum_{\ell \in [m+1]} \rho_i \circ (\ell\,\,m + 1)\Big) + \on{PR}_{\leq (r + 2)R}\]
for suitable scalars $\tilde{\nu}_i$.\end{observation}

\begin{proof}Let $\ell, \ell' \in [m]$ be two distinct indices. By the symmetry assumption, applying the transposition $(\ell\,\,\ell')$ we see that
\begin{align*}0 = &\sum_{i \in [q]} (\mu_{i, \ell} + \mu_{i, \ell'}) \pi_i \circ (\ell \,\, m + 1) + \sum_{i \in [q]} (\mu_{i, \ell} + \mu_{i, \ell'}) \pi_i \circ (\ell' \,\, m + 1) \\
&\hspace{2cm}+ \sum_{i \in [t]} (\nu_{i, \ell} + \nu_{i, \ell'}) \rho_i \circ (\ell \,\, m + 1) +  \sum_{i \in [t]} (\nu_{i, \ell} + \nu_{i, \ell'}) \rho_i \circ (\ell' \,\, m + 1) + \tau + \tau \circ (\ell \,\,\ell').\end{align*}

By property \textbf{(iv)} (we now use the assumption that $C \geq 2$) of the forms in the previous lemma, we have that 
\begin{align}&\sum_{i \in [q]} (\mu_{i, \ell} + \mu_{i, \ell'}) \pi_i \circ (\ell \,\, m + 1) + \sum_{i \in [q]} (\mu_{i, \ell} + \mu_{i, \ell'}) \pi_i \circ (\ell' \,\, m + 1) \nonumber\\
&\hspace{2cm}+ \sum_{i \in [t]} (\nu_{i, \ell} + \nu_{i, \ell'}) \rho_i \circ (\ell \,\, m + 1) +  \sum_{i \in [t]} (\nu_{i, \ell} + \nu_{i, \ell'}) \rho_i \circ (\ell' \,\, m + 1)\label{symmetricLinearCombinationObsn}\end{align}
has to be a linear combination of sums $\sum_{\ell \in [m + 1]} \pi_i \circ (\ell\,\,m + 1)$. However, $\pi_i = \pi_i \circ (m+1 \,\,m+1)$ does not appear in~\eqref{symmetricLinearCombinationObsn} for any $i$, hence the displayed expression has to be zero, making $\mu_{i, \ell} = \mu_{i, \ell'}$ and $\nu_{i, \ell} = \nu_{i, \ell'}$, as desired.\\

Suppose now that $\phi$ is symmetric in variables $x_{[m+1]}$. Note that
\begin{align*}\phi = \phi \circ (m\,\,m+1) = &\sum_{i \in [s]} \lambda_i \sigma_i \circ (m\,\,m+1) + \sum_{i \in [q]} \Big(\sum_{\ell \in [m-1]} \mu_{i, \ell} \pi_i \circ (\ell \,\, m + 1) + \mu_{i, m + 1} \pi_i \circ (m + 1 \,\,m) + \mu_{i, m} \pi_i\Big)\\
& \hspace{1cm}+ \sum_{i \in [t]} \Big(\sum_{\ell \in [m-1]} \nu_{i, \ell} \rho_i \circ (\ell \,\, m + 1) + \nu_{i, m+1} \rho_i \circ (m \,\,m + 1) + \nu_{i,m} \rho_i\Big)  + \tau\circ (m \,\,m+ 1 ).\end{align*}
Hence $0 = \phi + \phi \circ (m\,\,m+1)$ belongs to 
\[ \sum_{i \in [q]} (\mu_{i, m} + \mu_{i, m + 1}) \Big(\pi_i \circ (m  \,\,m+ 1) + \pi_i\Big) +  \sum_{i \in [t]}(\nu_{i, m} + \nu_{i, m+1}) \Big( \rho_i \circ (m \,\,m + 1) + \rho_i\Big) + \on{PR}_{\leq (s + 2)R}.\]
If $m = 1$, by the remark following the previous lemma, we may assume that forms $\pi_i$ do not appear, and if $m \geq 2$ then $\pi_i \circ (1  \,\,m+ 1)$ does not appear in the above expression. By property \textbf{(iv)} of Lemma~\ref{symmRespWRegLemma} (we now use the assumption that $C \geq 2, D \geq 2$) it follows that $\mu_{i, m} = \mu_{i, m + 1}$ (in the case $m \geq 2$) and $\nu_{i, m} = \nu_{i, m + 1}$. Let $\tilde{\nu}_i = \nu_{i,1} = \dots = \nu_{i, m+1}$. The observation follows from the property \textbf{(iii)} of Lemma~\ref{symmRespWRegLemma}.\end{proof}

\subsection{Symmetry in $x_1$ and $x_2$}

We prove Theorem~\ref{approximateSymmetry2} in this subsection.

\begin{proof}[Proof of Theorem~\ref{approximateSymmetry2}]During the proof we consider decompositions of the shape
\begin{align}\alpha(x_1, x_2, x_3, \dots, x_k) + \alpha(x_2, x_1, x_3, \dots, x_k) = &\sum_{i \in [r_1]} \beta_i(x_1, x_2, x_{I_{i}}) \beta'_{i, 1}(z_{I'_{i, 1}}) \cdots \beta'_{i, d_i}(z_{I'_{i, d_i}})\nonumber \\&\hspace{1cm}+ \sum_{i \in [r_2]} \gamma_i(x_1, x_{J_{i, 1}}) \gamma'_i(x_2, x_{J_{i, 2}})\gamma''_{i, 1}(x_{J'_{i, 1}}) \cdots \gamma''_{i, d'_i}(x_{J'_{i, d'_i}})\label{basicsymmexactpropeqn1}\end{align}
for some integers $r_1$ and $r_2$ and suitable multilinear forms where the partitions of variables in products belong to some down-set of partitions $\mathcal{P}$. We write $r' = r_1 + r_2$, which will satisfy $r' \leq O(\exp^{(O(1))}(r))$ at all times. Initially we have $r' \leq r$ and $\mathcal{P}$ contains all partitions except the trivial one $\{[k]\}$.\\
\indent The proof splits into three parts; in the first stage we remove the $\beta_i(x_1, x_2, x_{I_{i}}) \beta'_{i, 1}(z_{I'_{i, 1}}) \cdots \beta'_{i, d_i}(z_{I'_{i, d_i}})$ terms, in the second we remove the almost all of $ \gamma_i(x_1, x_{J_{i, 1}}) \gamma'_i(x_2, x_{J_{i, 2}})\gamma''_{i, 1}(x_{J'_{i, 1}}) \cdots \gamma''_{i, d'_i}(x_{J'_{i, d'_i}})$ terms and in the final step we remove remaining terms which will have linear forms in $x_1$ and $x_2$, thereby relating $\alpha$ to a form symmetric in the first two variables.\\

\noindent\noindent\textbf{Step 1. Removing $\beta$ forms.} Using Proposition~\ref{lcmlformsdecomppropo} we may assume that every $\beta_i$ comes from a slice of $\alpha + \alpha \circ (1\,\,2)$. This has the cost of replacing $r$ by $r' = O({r}^{O(1)})$. Misusing the notation, we still write $r_1$ and $r_2$ for the number of products in each of the two sums. Thus, for each $i \in [r_1]$ we obtain a multilinear form $\tilde{\beta}_i(x_1, x_2, x_{I_{i}})$ such that $\beta_i(x_1, x_2, x_{I_{i}}) = \tilde{\beta}_i(x_1, x_2, x_{I_{i}}) + \tilde{\beta}_i(x_2, x_1, x_{I_{i}})$. Considering the multilinear form $\alpha'$ defined by
\[\alpha'(x_{[k]}) = \alpha(x_{[k]}) + \sum_{i \in [r_1]} \tilde{\beta}_i(x_1, x_2, x_{I_{i}}) \beta'_{i, 1}(z_{I'_{i, 1}}) \cdots \beta'_{i, d_i}(z_{I'_{i, d_i}})\]
we see that $\on{prank}^{\mathcal{P}} (\alpha + \alpha') \leq O({r}^{O(1)})$ and
\begin{align}\alpha'(x_1, x_2, x_3, \dots, x_k) + \alpha'(x_1, x_2, x_3, \dots, x_k) = \sum_{i \in [r_2]} \gamma_i(x_1, x_{J_{i, 1}}) \gamma'_i(x_2, x_{J_{i, 2}})\gamma''_{i, 1}(x_{J'_{i, 1}}) \cdots \gamma''_{i, d'_i}(x_{J'_{i, d'_i}}).\label{basicsymmexactpropeqn2}\end{align}

\noindent\noindent\textbf{Step 2. Removing most $\gamma$ forms.} In this part of the proof we perform another iterative procedure in which we keep track of a down-set of partitions $\mathcal{P}'\subseteq \mathcal{P}$ with the property that $1$ and $2$ are in different sets in every partition and swapping $1$ and $2$ results in a partition still in $\mathcal{P}'$ and we find another mutlilinear form $\alpha'' \colon G^k \to \mathbb{F}_2$ such that $\on{prank}(\alpha' + \alpha'') \leq O(\exp^{(O(1))}(r))$ and we have an equality 
\begin{equation}\label{alphasymm2partialgammaPartitions}\alpha''(x_1, x_2, x_3, \dots, x_k) + \alpha''(x_2, x_1, x_3, \dots, x_k) = \sum_{i \in [s]} \gamma_i(x_1,x_{I_i}) \gamma'_i(x_2, x_{I'_i}) \delta_{i, 1}(x_{J_{i,1}}) \cdots \delta_{i, d_i}(x_{J_{i,d_i}})\end{equation}
where $s \leq O(\exp^{(O(1))}(r))$ and each partition $(\{1\} \cup I_i, \{2\} \cup I'_i, J_{i,1}, \dots, J_{i, d_i})$ belongs to $\mathcal{P}'$. (We misuse the notation as the multilinear forms $\gamma_i, \gamma'_i$ are not necessarily identical to those in the assumed decomposition~\eqref{basicsymmexactpropeqn2} and are being modified in each step of the procedure.) We now describe a step in the procedure. Let $\Big\{\{1\} \cup I,\{2\} \cup I', J_1, \dots, J_d\Big\}$ be a maximal partition in $\mathcal{P}'$, where the priority is given to partitions with $I \cup I' \not=\emptyset$. Let $\mathcal{P}''$ be the down-set of partitions obtained by removing $\Big\{\{1\} \cup I,\{2\} \cup I', J_1, \dots, J_d\Big\}$ and $\Big\{\{1\} \cup I',\{2\} \cup I, J_1, \dots, J_d\Big\}$ from $\mathcal{P}'$. Assume first that $I \cup I' \not=\emptyset$.\\

\noindent Before proceeding, we need to regularize the forms appearing in the expression above. We consider the following $d+2$ lists of forms:
\begin{itemize}
\item $\gamma_i(u, x_I)$ for all $i \in [s]$ such that $I_i = I$ and $\gamma'_i(u, x_I)$ for all $i \in [s]$ such that $I'_i = I$,
\item $\gamma_i(u, x_I)$ for all $i \in [s]$ such that $I_i = I'$ and $\gamma'_i(u, x_I)$ for all $i \in [s]$ such that $I'_i = I'$,
\item for each $j \in [d]$, make the list of all forms that depend on the variables $x_{J_j}$.
\end{itemize}

Let $C, D \geq 2$ be constants to be chosen later. Apply the $m = 0$ case of Lemma~\ref{symmRespWRegLemma} (see Remark~\ref{symmRespWRegLemmaRemark}) to each of these lists. We thus obtain further lists of forms:
\begin{itemize}
\item $\phi_i(u, x_I)$ for $i \in [t_1]$ where $t_1 \leq 2s$,
\item $\psi_i(u, x_{I'})$ for $i \in [t_2]$ where $t_2 \leq 2s$, and 
\item for each $j \in [d]$, a list $\mu_{j,i}(x_{J_j})$ with $i \in [q_j]$ for some $q_j \leq s$,
\end{itemize}
and a quantity $R \leq C^{D^{O_k(s)}}$ with properties
\begin{itemize}
\item each multilinear form in one of the initial $d+2$ lists can be expressed as a sum of a linear combination of forms in the corresponding new list and a multilinear form of partition rank at most $R$,
\item for each of $d+2$ new lists, the non-zero linear combinations have partition rank at least $(C(R + 2s))^D$.
\end{itemize}

Replacing the old forms by the new ones, we have scalars $\lambda_{i_1, i_2, j_1, \dots, j_d}, \lambda'_{i_1, i_2, j_1, \dots, j_d}$, where $i_1 \in [t_1],$ $i_2 \in [t_2],$ $j_1 \in [q_1], \dots,$ $j_d \in [q_d]$, and a multilinear form $L(x_{[k]})$ such that which is a sum of at most $O(s(s + R)^k)$ products whose partitions lie in $\mathcal{P}''$ such that 
\begin{align}&\alpha''(x_1, x_2, x_3, \dots, x_k) + \alpha''(x_2, x_1, x_3, \dots, x_k)\nonumber\\
&\hspace{1cm}= \sum_{i_1 \in [t_1], i_2 \in [t_2], j_1 \in [q_1], \dots, j_d \in [q_d]} \lambda_{i_1, i_2, j_1, \dots, j_d} \phi_{i_1}(x_1,x_{I}) \psi_{i_2}(x_2, x_{I'}) \mu_{1, j_1}(x_{J_1}) \cdots \mu_{d, j_d}(x_{J_d}) \nonumber\\
&\hspace{1cm} + \sum_{i_1 \in [t_1], i_2 \in [t_2], j_1 \in [q_1], \dots, j_d \in [q_d]} \lambda'_{i_1, i_2, j_1, \dots, j_d} \phi_{i_1}(x_2,x_{I}) \psi_{i_2}(x_1, x_{I'}) \mu_{1, j_1}(x_{J_1}) \cdots \mu_{d, j_d}(x_{J_d}) + L(x_{[k]}).\label{alphasymm2partialgammaPartitions2}\end{align}

Using the fact that
\begin{equation}\Big(\alpha''(x_1, x_2, x_{[3,k]}) + \alpha''(x_2, x_1, x_{[3,k]})\Big) + \Big(\alpha''(x_1, x_2, x_{[3,k]}) + \alpha''(x_2, x_1, x_{[3,k]})\Big)\circ(1\,\,2) = 0\label{alphadoubleprimeeqnsymm2}\end{equation}
Lemma~\ref{zeromlformindexlemma} implies $\lambda_{i_1, i_2, j_1, \dots, j_d} = \lambda'_{i_1, i_2, j_1, \dots, j_d}$, provided $C$ and $D$ are sufficiently large compared only to the constants from Lemma~\ref{zeromlformindexlemma} and $k$. Thus, setting
\[\alpha'''(x_{[k]}) =  \alpha''(x_{[k]}) + \sum_{i_1 \in [t_1], i_2 \in [t_2], j_1 \in [q_1], \dots, j_d \in [q_d]} \lambda_{i_1, i_2, j_1, \dots, j_d} \phi_{i_1}(x_1,x_{I}) \psi_{i_2}(x_2, x_{I'}) \mu_{1, j_1}(x_{J_1}) \cdots \mu_{d, j_d}(x_{J_d})\]
it follows that $\on{prank}(\alpha'' + \alpha''') \leq O(\exp^{(O(1))}(r))$ and $\alpha''' + \alpha''' \circ(1\,\,2)$ satisfies~\eqref{alphasymm2partialgammaPartitions} for the down-set of partitions $\mathcal{P}''$ and $s \leq O(\exp^{(O(1))}(r))$.\\

\noindent\noindent\textbf{Step 3. Removing remaining forms.} Let us now treat the case $I = I' = \emptyset$. The same steps of the argument apply, except that this time the first two of the $d+2$ lists of forms become the same list of linear forms $\gamma_i(u)$ and $\gamma'_i(u)$ for $i \in [s]$ and we may omit the list of regularized forms $\psi_i$. This time, we have that
\begin{align*}&\alpha'''(x_1, x_2, x_3, \dots, x_k) + \alpha'''(x_2, x_1, x_3, \dots, x_k)\nonumber\\
&\hspace{1cm}= \sum_{i \in [t_1], j_1 \in [q_1], \dots, j_d \in [q_d]} \lambda_{i_1, i_2, j_1, \dots, j_d} \phi_{i_1}(x_1) \phi_{i_2}(x_2) \mu_{1, j_1}(x_{J_1}) \cdots \mu_{d, j_d}(x_{J_d}) + L(x_{[k]}).\end{align*}
By our choice of maximal partitions in the removal procedure, we have that every product appearing in $L(x_{[k]})$ has $x_1$ and $x_2$ as arguments of linear forms. Therefore, each partition of variables in a product in $L(x_{[k]})$ has a partition of $x_{[3,k]}$ which is not greater than $J_1, \dots, J_d$.\\
\indent We first show that $\lambda_{i, i, j_1, \dots, j_d} = 0$ for all indices (note that $i \in [t_1]$ appears twice). Take an element $w \in G$ such that $\phi_{i'}(w) = \id(i = i')$. Putting $x_1 = x_2 = w$ gives
\begin{align*}0 = &\alpha'''(w, w, x_3, \dots, x_k) + \alpha'''(w, w, x_3, \dots, x_k) \\
= &\sum_{i', i'' \in [t_1], j'_1 \in [q_1], \dots, j'_d \in [q_d]} \lambda_{i', i'', j'_1, \dots, j'_d} \phi_{i'}(w)\phi_{i''}(w) \mu_{1, j'_1}(x_{J_1}) \cdots \mu_{d, j'_d}(x_{J_d}) + L(w,w,x_{[3,k]})\\ 
= &\sum_{j'_1 \in [q_1], \dots, j'_d \in [q_d]} \lambda_{i, i, j'_1, \dots, j'_d} \mu_{1, j'_1}(x_{J_1}) \cdots \mu_{d, j'_d}(x_{J_d}) + L(w,w,x_{[3,k]}).
\end{align*}
Recall the property of partitions of products in $L(w,w,x_{[3,k]})$, Lemma~\ref{zeromlformindexlemma} shows that $\lambda_{i, i, j_1, \dots, j_d} = 0$.\\
\indent Finally, to see that $\lambda_{i_1, i_2, j_1, \dots, j_d} = \lambda_{i_2, i_1, j_1, \dots, j_d}$, use~\eqref{alphadoubleprimeeqnsymm2} and observe that the coefficient of $\phi_{i_1}(x_1)$ $\phi_{i_2}(x_2)$ $\mu_{1, j_1}(x_{J_1}) \cdots$ $\mu_{d, j_d}(x_{J_d})$ equals $\lambda_{i_1, i_2, j_1, \dots, j_d} + \lambda_{i_2, i_1, j_1, \dots, j_d}$. Lemma~\ref{zeromlformindexlemma} implies that it is zero. Set 
\[\alpha^{(4)}(x_{[k]}) =  \alpha'''(x_{[k]}) + \sum_{i_1 \in [t_1], i_2 \in [t_2], j_1 \in [q_1], \dots, j_d \in [q_d]} \lambda_{i_1, i_2, j_1, \dots, j_d} \phi_{i_1}(x_1) \phi_{i_2}(x_2) \mu_{1, j_1}(x_{J_1}) \cdots \mu_{d, j_d}(x_{J_d})\]
to see that $\on{prank}(\alpha''' + \alpha^{(4)}) \leq O(\exp^{(O(1))}(r))$ and $\alpha^{(4)} + \alpha^{(4)} \circ(1\,\,2)$ satisfies~\eqref{alphasymm2partialgammaPartitions} for the down-set of partitions $\mathcal{P}''$ and $s \leq O(\exp^{(O(1))}(r))$. This completes the description of the iterative procedure, completing the proof. We choose $C$ and $D$ depending only to the constants from Lemma~\ref{zeromlformindexlemma} and $k$ so that the partition rank condition of Lemma~\ref{zeromlformindexlemma} holds.\end{proof}

\subsection{Symmetry in $x_{[4]}$}

In this subsection we show how to pass from a multilinear form that is symmetric in variables $x_{[3]}$ to another one which is symmetric in variables $x_{[4]}$, under suitable conditions. The main results in this subsection are the Theorems~\ref{mainSymmExtn4} and~\ref{mainSymmExtn5}.\\
\indent Before embarking on the proof of Theorem~\ref{mainSymmExtn4}, we note that a weaker symmetry property is sufficient to deduce the stated one.

\begin{lemma}\label{weaksymmetryextends}Let $\alpha \colon G^k \to \mathbb{F}_2$ be a multilinear form symmetric in variables $x_{[3]}$. Let $\phi = \alpha + \alpha \circ (3\,\,4)$. Then $\phi$ satisfies the identity
\[\phi + \phi \circ (1\,\,3) + \phi\circ(1\,\,4) = 0.\]
In particular, if $\phi$ is symmetric in variables $x_1$ and $x_{4}$ then $\alpha$ is in fact symmetric in variables $x_{[4]}$.\end{lemma}

\begin{proof} Simple algebraic manipulation relying on the symmetry of $\alpha$ gives
\begin{align*}\phi + \phi \circ (1\,\,3) + \phi\circ(1\,\,4) =& \alpha + \alpha \circ (3\,\,4) + \alpha  \circ (1\,\,3) + \alpha \circ (3\,\,4)  \circ (1\,\,3) \\
&\hspace{5cm} + \alpha\circ(1\,\,4) + \alpha \circ (3\,\,4)\circ(1\,\,4) \\
= &\alpha + \alpha \circ (3\,\,4) + \alpha + \alpha \circ (1\,\,4) +  \alpha\circ(1\,\,4) + \alpha \circ (3\,\,4)\\
= & 0.\end{align*}

If $\phi$ is symmetric in variables $x_1$ and $x_{4}$ then we get
\begin{align*}0 = \phi + \phi \circ (1\,\,3) + \phi\circ(1\,\,4) = \phi \circ (1\,\,3) = &\alpha \circ (1\,\,3) + \alpha \circ (3\,\,4) \circ (1\,\,3)\\
= &\alpha + \alpha \circ (1\,\,4).\end{align*}
As $\alpha$ is already symmetric in variables $x_{[3]}$, it follows that it is symmetric in $x_{[4]}$, as required.\end{proof}

\begin{proof}[Proof of Theorem~\ref{mainSymmExtn4}]By assumptions, we have that
\begin{equation}\alpha(x_{[4]})+ \alpha \circ (3\,\,4)(x_{[4]}) = \sum_{i \in [r]} \beta_{i, 1}(x_{I_{i,1}}) \beta_{i, 2}(x_{I_{i,2}}) \dots \beta_{i, d_i}(x_{I_{i,d_i}})\label{main4decompEqn}\end{equation}
for some multilinear forms $\beta_{i, j}$ such that for each $i \in [r]$ the sets $I_{i,1}, \dots, I_{i,d_i}$ form a partition of $[4]$, and $d_i \geq 2$. Notice that for a given $i \in [r]$ we either have that some $I_{i,j}$ is a singleton, or we have $d_i = 2$ and both sets $I_{i,1}$ and $I_{i,2}$ have size 2. Let $\mathcal{I}$ be the set of indices $i \in [r]$ where we get a singleton set $I_{i,j(i)}$ for some $j(i) \leq d_i$. If we set $U = \{u \in G \colon \beta_{i, j(i)}(u) = 0\}$, then we have for all $x_1, x_2, x_3, x_4 \in U$ that
\begin{equation}\alpha(x_{[4]})+ \alpha \circ (3\,\,4)(x_{[4]}) = \sum_{i \in [r] \setminus \mathcal{I}} \beta_{i, 1}(x_{I_{i,1}}) \beta_{i, 2}(x_{I_{i,2}})\label{main4decompEqn2}\end{equation}
holds and $|I_{i,1}| = |I_{i,2}| = 2$ for all indices $i$ appearing in the displayed equality. Take the list of all bilinear forms appearing on the right-hand-side of~\eqref{main4decompEqn2} and apply Lemma~\ref{symmRespWRegLemma} to it, with parameters $C, D$ to be chosen later (which will depend only on constants in Lemma~\ref{zeromlformindexlemma}) and $m = 1$. We get a positive integer $R \leq C^{D^{O(r)}}$, bilinear forms $\sigma_1, \dots, \sigma_s, \rho_1, \dots, \rho_t$, where $s + t \leq r$, such that:
\begin{itemize}
\item[\textbf{(i)}] we have $\on{rank}\Big( \sigma_i + \sigma_i \circ(1\,\,2)\Big) \leq R$ for all $i \in [s]$,
\item[\textbf{(ii)}] all non-zero linear combinations of forms $\sigma_1, \dots, \sigma_s, \rho_1, \dots, \rho_t$ and $ \rho_1 \circ (1\,\,2), \dots, \rho_t\circ (1\,\,2)$ have rank at least $(C(R + 2r))^D$,
\item[\textbf{(iii)}] every bilinear form in the list differs from a linear combination of forms $\sigma_{[s]}, \rho_{[t]}$ and $ \rho_{[t]} \circ (1\,\,2)$ by a bilinear form of rank at most $R$.
\end{itemize}
Passing to a further subspace $U' \leq U$ of codimension at most $2rR$ in $U$, we may assume that properties \textbf{(i)} and \textbf{(iii)} are exact rather than approximate, i.e. $\sigma_i = \sigma_i \circ(1\,\,2)$ and every bilinear form in the list equals a linear combination of the newly-found forms. We may replace the bilinear forms in~\eqref{main4decompEqn2} with the newly-found forms. This leads to equality
\begin{align}&\alpha(x_{[4]})+ \alpha \circ (3\,\,4)(x_{[4]}) \nonumber\\
&\hspace{1cm}= \sum_{\ssk{i \in [s]\\(a,b,c,d) \in \mathcal{V}_1}} \lambda_{\ssk{i\\a,b,c,d}} \sigma_i(x_a, x_b) \sigma_i(x_c, x_d) + \sum_{\ssk{1\leq i < j \leq s\\(a,b,c,d) \in \mathcal{V}_2}} \lambda_{\ssk{i,j\\a,b,c,d}} \sigma_i(x_a, x_b) \sigma_j(x_c, x_d)\nonumber\\
&\hspace{2cm} +  \sum_{\ssk{i \in [t]\\(a,b,c,d) \in \mathcal{V}_3}} \mu_{\ssk{i\\a,b,c,d}} \rho_i(x_a, x_b) \rho_i(x_c, x_d) +  \sum_{\ssk{1\leq i < j \leq t\\(a,b,c,d) \in \mathcal{V}_4}} \mu_{\ssk{i,j\\a,b,c,d}} \rho_i(x_a, x_b) \rho_j(x_c, x_d)\nonumber\\
&\hspace{2cm} +  \sum_{\ssk{i \in [s],j \in [t]\\(a,b,c,d) \in \mathcal{V}_5}} \nu_{\ssk{i,j\\a,b,c,d}} \sigma_i(x_a, x_b) \rho_j(x_c, x_d),\label{main4decompEqnReg1}\end{align}
for all $x_1, x_2, x_3, x_4 \in U'$, where
\begin{itemize}
\item $\mathcal{V}_1 = \{(1,2,3,4), (1,3,2,4), (1,4,2,3)\}$,
\item $\mathcal{V}_2 = \{(1,2,3,4), (1,3,2,4), (1,4,2,3), (2,3,1,4), (2,4,1,3), (3,4,1,2)\}$,
\item $\mathcal{V}_3 = \{(1,2,3,4), (2,1,3,4), (1,2,4,3), (2,1,4,3), (1,3,2,4), (3,1,2,4),$ $(1,3,4,2), (3,1,4,2), (1,4,2,3),$ $(4,1,2,3), (1,4,3,2), (4,1,3,2)\}$,
\item $\mathcal{V}_4 = \on{Sym}_{[4]}$, and
\item $\mathcal{V}_5 = \{(1,2,3,4),(1,2,4,3), (1,3,2,4), (1,3,4,2), (1,4,2,3), (1,4,3,2),$ $(2,3,1,4), (2,3,4,1),$ $(2,4,1,3),$ $(2,4,3,1),$ $(3,4,1,2), (3,4,2,1)\}$.
\end{itemize}

Write $\phi(x_{[4]}) = \alpha(x_{[4]})+ \alpha \circ (3\,\,4)(x_{[4]})$. This form has the following properties:
\begin{itemize}
\item[\textbf{(i)}] $\phi = \phi \circ (1\,\,2)$,
\item[\textbf{(ii)}] $\phi = \phi \circ (3\,\,4)$,
\item[\textbf{(iii)}] $\phi + \phi \circ (1\,\,3) + \phi\circ(1\,\,4) = 0$,
\item[\textbf{(iv)}] $\phi(u,v,u,v) = 0$ for all $u,v \in U'$.
\end{itemize}

We use these properties of $\phi$ to make $\alpha$ symmetric in all variables. Note that applying Lemma~\ref{zeromlformindexlemma} to the sum $\phi + \phi \circ (1\,\,2)$, which vanishes by symmetry property \textbf{(i)}, implies that we have equalities $\lambda_{\ssk{i\\a,b,c,d}} = \lambda_{\ssk{i\\a',b',c',d'}}$ whenever $(a,b,c,d)$ and $(a',b',c',d')$ can be obtained from each other by composing with transposition $(1\,\,2)$. Using Lemma~\ref{zeromlformindexlemma} similarly for symmetry property \textbf{(ii)} and for other coefficients we obtain additional equalities. In the rest of the proof we obtain further identities between coefficients. We treat each of the 5 sums in~\eqref{main4decompEqnReg1} separately.\\

\noindent\textbf{Step 1. Coefficients $\lambda_{\ssk{i\\a,b,c,d}}$.} Fix any $i \in [s]$. To see that $\lambda_{\ssk{i\\1,2,3,4}} = 0$, observe that the coefficient of $\sigma_i(x_1, x_2)\sigma_i(x_3, x_4)$ in the expression $\phi + \phi \circ (1\,\,3) + \phi\circ(1\,\,4)$ equals 
\[\lambda_{\ssk{i\\1,2,3,4}} + \lambda_{\ssk{i\\1,4,2,3}} + \lambda_{\ssk{i\\1,3,2,4}}.\]
By property \textbf{(iii)} of $\phi$ and Lemma~\ref{zeromlformindexlemma}, we conclude that
\[\lambda_{\ssk{i\\1,2,3,4}} + \lambda_{\ssk{i\\1,4,2,3}} + \lambda_{\ssk{i\\1,3,2,4}} = 0.\]
Since $\lambda_{\ssk{i\\1,4,2,3}} = \lambda_{\ssk{i\\1,3,2,4}}$, we finally get $\lambda_{\ssk{i\\1,2,3,4}} = 0$.\\
\indent Next, we show that $\lambda_{\ssk{i\\1,4,2,3}} = \lambda_{\ssk{i\\1,3,2,4}} = 0$. To that end, let $X = \{u \in U' \colon (\forall i \in [s]) \sigma_i(u,u) = 0, (\forall i \in [t]) \rho_i(u,u) = 0\}$. By Lemma~\ref{quadraticVarSize} this set has size $|X| \geq 2^{-s-t-1}|U'|$. Since all non-zero linear combinations of forms $\sigma_{[s]}$, $\rho_{[t]}$ and $\rho_{[t]}\circ(1\,\,2)$ have rank at least $(C(R + 2r))^D - 2rR$ on $U'$, we may find $u,v \in X$ such that $\sigma_i(u,v) = 1$, $\sigma_j(u,v) = 0$ for all $j \not=i$ and $\rho_j(u,v) = \rho_j(v,u) = 0$ for all $j \in [t]$. Putting $x_1 = x_3 = u, x_2 = x_4 = v$ in~\eqref{main4decompEqnReg1}, we deduce that
\[0 = \lambda_{\ssk{i\\1,3,2,4}} \sigma_i(u,u)\sigma_i(v,v) + \lambda_{\ssk{i\\1,4,2,3}} \sigma_i(u,v)\sigma_i(u,v) = \lambda_{\ssk{i\\1,4,2,3}},\]
as required.\\

\noindent\textbf{Step 2. Coefficients $\mu_{\ssk{i\\a,b,c,d}}$.} As remarked before the first step, we have scalars $\overline{\mu}_i, \overline{\mu}'_i, \overline{\mu}''_i, \overline{\mu}'''_i$ such that
\begin{align*}\overline{\mu}_i &= \mu_{\ssk{i\\1,2,3,4}} = \mu_{\ssk{i\\2,1,3,4}} = \mu_{\ssk{i\\1,2,4,3}} = \mu_{\ssk{i\\2,1,4,3}},\\
\overline{\mu}'_i &= \mu_{\ssk{i\\3,1,2,4}} = \mu_{\ssk{i\\1,4,3,2}} = \mu_{\ssk{i\\4,1,2,3}} = \mu_{\ssk{i\\1,3,4,2}},\\
\overline{\mu}''_i &= \mu_{\ssk{i\\1,3,2,4}} = \mu_{\ssk{i\\1,4,2,3}}\,\,\text{and}\,\,\,\, \overline{\mu}'''_i= \mu_{\ssk{i\\3,1,4,2}} = \mu_{\ssk{i\\4,1, 3,2}}.\\\end{align*}
Fix any $i \in [t]$. We now show that $\overline{\mu}_i + \overline{\mu}'_i + \overline{\mu}''_i = 0$ and $\overline{\mu}_i + \overline{\mu}'_i + \overline{\mu}'''_i = 0$. Observe that the coefficient of $\rho_i(x_3, x_1)\rho_i(x_4, x_2)$ in the expression $\phi + \phi \circ (1\,\,3) + \phi\circ(1\,\,4)$ equals 
\[\mu_{\ssk{i\\3,1,4,2}} + \mu_{\ssk{i\\1,3,4,2}} + \mu_{\ssk{i\\1,2,3,4}} = \overline{\mu}'''_i + \overline{\mu}'_i + \overline{\mu}_i.\]
By Lemma~\ref{zeromlformindexlemma} we see that $\overline{\mu}_i + \overline{\mu}'_i + \overline{\mu}'''_i = 0$. Similarly, we consider the coefficient of $\rho_i(x_1, x_3)\rho_i(x_2, x_4)$ in the expression $\phi + \phi \circ (1\,\,3) + \phi\circ(1\,\,4)$, which equals 
\[\mu_{\ssk{i\\1,3,2,4}} + \mu_{\ssk{i\\3,1,2,4}} + \mu_{\ssk{i\\2,1,4,3}} = \overline{\mu}''_i + \overline{\mu}'_i + \overline{\mu}_i\]
which vanishes by Lemma~\ref{zeromlformindexlemma}. As the last equality between coefficients in this step, we claim that $\overline{\mu}_i = \overline{\mu}'_i$. Let $X = \{u \in U' \colon (\forall i \in [s]) \sigma_i(u,u) = 0, (\forall i \in [t]) \rho_i(u,u) = 0\}$. By Lemma~\ref{quadraticVarSize} this set has size $|X| \geq 2^{-s-t-1}|U'|$. We thus have a choice of $u,v \in X$ such that $\rho_i(u,v) = 1$, $\rho_i(v,u) = 0$, $\rho_j(u,v) = \rho_j(u,v) = 0$ for $j \not = i$ and $\sigma_j(u,v) = \sigma_j(v,u) = 0$ for $j \in [t]$. Then setting $x_1 = x_3 = u, x_2 = x_4 = v$ we get
\[0 = \phi(u,v,u,v) = \sum_{(a,b,c,d) \in \mathcal{V}_3} \mu_{\ssk{i\\a,b,c,d}} \rho_i(x_a, x_b) \rho_i(x_c, x_d) = \overline{\mu}_i \rho_i(u,v)\rho_i(u,v) + \overline{\mu}'_i \rho_i(u,v)\rho_i(u,v) = \overline{\mu}_i + \overline{\mu}'_i .\]
Thus, we have $\overline{\mu}'_i = \overline{\mu}_i$ and $\overline{\mu}''_i = \overline{\mu}'''_i  = 0$.\\

Define a multilinear form $\alpha^{(1)} \colon G^4 \to \mathbb{F}_2$ by
\[\alpha^{(1)}(x_1, x_2, x_3, x_4) = \alpha(x_1, x_2, x_3, x_4) + \sum_{i \in [t]} \overline{\mu}_i \Big(\sum_{\pi \in \on{Sym}_{[3]}}\rho_i(x_{\pi_1}, x_{\pi_2})\rho_i(x_{\pi_3}, x_4)\Big).\]
By our work in the first two steps, we conclude that
\begin{align}\alpha^{(1)}(x_{[4]})+ \alpha^{(1)} \circ (3\,\,4)(x_{[4]}) = &\sum_{\ssk{1\leq i < j \leq s\\(a,b,c,d) \in \mathcal{V}_2}} \lambda_{\ssk{i,j\\a,b,c,d}} \sigma_i(x_a, x_b) \sigma_j(x_c, x_d) +  \sum_{\ssk{1\leq i < j \leq t\\(a,b,c,d) \in \mathcal{V}_4}} \mu_{\ssk{i,j\\a,b,c,d}} \rho_i(x_a, x_b) \rho_j(x_c, x_d)\nonumber\\
&\hspace{2cm} +  \sum_{\ssk{i \in [s],j \in [t]\\(a,b,c,d) \in \mathcal{V}_5}} \nu_{\ssk{i,j\\a,b,c,d}} \sigma_i(x_a, x_b) \rho_j(x_c, x_d)\label{main4decompEqnReg2}\end{align}
holds for all $x_1, x_2, x_3, x_4 \in U'$. Let us set $\phi^{(1)} = \alpha^{(1)} + \alpha^{(1)}\circ(3\,\,4)$. Since $\sum_{\pi \in \on{Sym}_{[3]}}\rho_i(x_{\pi_1}, x_{\pi_2})\rho_i(x_{\pi_3}, x_4)$ is symmetric in variables $x_{[3]}$, the form $\phi^{(1)}$ still satisfies symmetry conditions \textbf{(i)}, \textbf{(ii)} and \textbf{(iii)} which we used for $\phi$.\\

\noindent\textbf{Step 3. Coefficients $\lambda_{\ssk{i,j\\a,b,c,d}}$.} By symmetry properties of $\phi$ we know that the coefficients $\lambda_{\ssk{i,j\\1,3,2,4}},$ $\lambda_{\ssk{i,j\\1,4,2,3}},$ $\lambda_{\ssk{i,j\\2,3,1,4}}$ and $\lambda_{\ssk{i,j\\2,4,1,3}}$ are equal. Write $\overline{\lambda}_{i,j}$ for this value. We show that the remaining coefficients in this step $\lambda_{\ssk{i,j\\1,2,3,4}}$ and $\lambda_{\ssk{i,j\\3,4,1,2}}$ both vanish.\\
\indent To see this, consider the coefficients of $\sigma_i(x_1, x_2)\sigma_j(x_3, x_4)$ and $\sigma_i(x_3, x_4)\sigma_j(x_1, x_2)$ in $\phi^{(1)} + \phi^{(1)} \circ (1\,\,3) + \phi^{(1)}\circ(1\,\,4)$. These are respectively
\[\lambda_{\ssk{i,j\\1,2,3,4}} + \lambda_{\ssk{i,j\\2,3,1,4}} + \lambda_{\ssk{i,j\\2,4,1,3}} = \lambda_{\ssk{i,j\\1,2,3,4}}\]
and
\[\lambda_{\ssk{i,j\\3,4,1,2}} + \lambda_{\ssk{i,j\\1,4,2,3}} + \lambda_{\ssk{i,j\\1,3,2,4}} = \lambda_{\ssk{i,j\\3,4,1,2}}.\]
Lemma~\ref{zeromlformindexlemma} shows that both expressions vanish, showing that $\lambda_{\ssk{i,j\\1,2,3,4}} = \lambda_{\ssk{i,j\\3,4,1,2}} = 0$.\\
\indent Define $\alpha^{(2)} \colon G^4 \to \mathbb{F}_2$ by 
\begin{align*}\alpha^{(2)}(x_1, x_2, x_3, x_4) =& \alpha^{(1)}(x_1, x_2, x_3, x_4) \\
&\hspace{1cm}+ \sum_{\ssk{i,j \in [s]\\i \not=j}} \overline{\lambda}_{i,j} \Big(\sigma_i(x_1, x_2)\sigma_j(x_3, x_4) + \sigma_i(x_2, x_3)\sigma_j(x_1, x_4)+ \sigma_i(x_3, x_1)\sigma_j(x_2, x_4)\Big).\end{align*} 
Note that the additional term in the equality above is symmetric in $x_{[3]}$ on the subspace $U'$ and so is $\alpha^{(2)}(x_{[4]})$. Defining $\phi^{(2)}(x_{[4]}) = \alpha^{(2)}(x_{[4]})+ \alpha^{(2)} \circ (3\,\,4)(x_{[4]})$, we observe that $\phi^{(2)}|_{U' \tdt U'}$ satisfies symmetry conditions \textbf{(i)}, \textbf{(ii)} and \textbf{(iii)} which we used for $\phi$. It follows from the definition of $\alpha^{(2)}$ that
\begin{align}\alpha^{(2)}(x_{[4]})+ \alpha^{(2)} \circ (3\,\,4)(x_{[4]}) = &\sum_{\ssk{1\leq i < j \leq t\\(a,b,c,d) \in \mathcal{V}_4}} \mu_{\ssk{i,j\\a,b,c,d}} \rho_i(x_a, x_b) \rho_j(x_c, x_d)\nonumber\\
&\hspace{2cm} +  \sum_{\ssk{i \in [s],j \in [t]\\(a,b,c,d) \in \mathcal{V}_5}} \nu_{\ssk{i,j\\a,b,c,d}} \sigma_i(x_a, x_b) \rho_j(x_c, x_d)\label{main4decompEqnReg3}\end{align}
holds for all $x_1, x_2, x_3, x_4 \in U'$.\\

In order to simplify the notation in the remaining two steps, we introduce informal objects called \emph{places}. Each symmetric form $\sigma_i$ has one place and each asymmetric form $\rho_i$ has two places corresponding to the first and the second variable. Notice that the symmetry properties of $\phi$ imply that the coefficients of products $\rho_i(x_a, x_b) \rho_j(x_c, x_d)$ and $\sigma_i(x_a, x_b) \rho_j(x_c, x_d)$ depend only on the choice of places for variables $x_3$ and $x_4$. For example, if we chose the place in $\sigma_i$ and a place corresponding to the second variable in $\rho_j$, then we have that products 
\[\sigma_i(x_1, x_3) \rho_j(x_2, x_4), \sigma_i(x_2, x_3) \rho_j(x_1, x_4), \sigma_i(x_1, x_4) \rho_j(x_2, x_3), \sigma_i(x_2, x_4) \rho_j(x_1, x_3)\]
receive the same coefficient. Furthermore, our work in step 3 could be rephrased as saying that the coefficients of products where we use two copies of the same place vanish.\\
\indent Using the places terminology, we write $\lambda_{A\,\,B}$ for the coefficient received by products where $x_3$ and $x_4$ are put at places $A$ and $B$. In the rest of the proof we show that 
\begin{equation}\lambda_{A\,\,B} + \lambda_{A\,\,C} + \lambda_{B\,\,C} = 0\label{placesEQN4vars}\end{equation}
 for all places $A, B, C$.\\

\noindent\textbf{Step 4. Coefficients $\mu_{\ssk{i,j\\a,b,c,d}}$.} Note that in this case the places $A, B$ and $C$ in~\eqref{placesEQN4vars} may assume to be distinct and furthermore that $A$ and $B$ belong to the form $\rho_i$ and $C$ is a place in $\rho_j$ (we do not assume that $i < j$, but only that $i \not=j$). Suppose that $C$ is the place corresponding to the first variable in $\rho_j$, the other case will follow analogously. Consider the coefficient of the product $\rho_i(x_3, x_4)\rho_j(x_1, x_2)$ in $\phi^{(2)} + \phi^{(2)} \circ (1\,\,3) + \phi^{(2)}\circ(1\,\,4) = 0$. This coefficient equals
\[\lambda_{A\,\,B} + \lambda_{B\,\,C} + \lambda_{A\,\,C}\]
and vanishes, proving~\eqref{placesEQN4vars}.\\
For each $i,j \in [t]$ with $i < j$ let $\overline{\mu}_{i,j, A} = \lambda_{A\,\,D}$, $\overline{\mu}_{i,j, B} = \lambda_{B\,\,D}$ and $\overline{\mu}_{i,j, C} = \lambda_{C\,\,D}$, where $A, B, C$ and $D$ are respectively the places corresponding to the first variable in $\rho_i$, the second variable in $\rho_i$, the first variable in $\rho_j$ and the second variable in $\rho_j$. We may now define $\alpha^{(3)} \colon G^4 \to \mathbb{F}_2$ by 
\begin{align}\alpha^{(3)}(x_1, x_2, x_3, x_4) =\alpha^{(2)}(x_1, x_2, x_3, x_4) + &\sum_{1 \leq i < j \leq t} \overline{\mu}_{i,j, A}\Big(\sum_{\pi \in \on{Sym}_{[3]}} \rho_i(x_4, x_{\pi(1)}) \rho_j(x_{\pi(2)}, x_{\pi(3)})\Big)\nonumber\\
& + \sum_{1 \leq i < j \leq t} \overline{\mu}_{i,j, B}\Big(\sum_{\pi \in \on{Sym}_{[3]}} \rho_i(x_{\pi(1)}, x_4) \rho_j(x_{\pi(2)}, x_{\pi(3)})\Big)\nonumber\\
& + \sum_{1 \leq i < j \leq t} \overline{\mu}_{i,j, C}\Big(\sum_{\pi \in \on{Sym}_{[3]}} \rho_i(x_{\pi(1)}, x_{\pi(2)}) \rho_j(x_4, x_{\pi(3)})\Big).\label{alpha3defn4vars}\end{align} 
It follows that $\alpha^{(3)}|_{U'\tdt U'}$ is still symmetric in the first three variables and the form $\phi^{(3)}(x_{[4]}) = \alpha^{(3)}(x_{[4]})+ \alpha^{(3)} \circ (3\,\,4)(x_{[4]})$ satisfies symmetry conditions \textbf{(i)}, \textbf{(ii)} and \textbf{(iii)} which we used for $\phi$ on the subspace $U'$. We claim that $\phi^{(3)}$ has simpler structure than $\phi^{(2)}$.

\begin{claim}\label{4varsphi3claim}The equality
\begin{align}\alpha^{(3)}(x_{[4]})+ \alpha^{(3)} \circ (3\,\,4)(x_{[4]}) = \sum_{\ssk{i \in [s],j \in [t]\\(a,b,c,d) \in \mathcal{V}_5}} \nu_{\ssk{i,j\\a,b,c,d}} \sigma_i(x_a, x_b) \rho_j(x_c, x_d)\label{main4decompEqnReg4}\end{align}
holds for all $x_1, x_2, x_3, x_4 \in U'$.\end{claim}

\begin{proof}[Proof of Claim~\ref{4varsphi3claim}]Fix $1\leq i < j \leq t$. Let $A, B, C$ and $D$ are respectively the places corresponding to the first variable in $\rho_i$, the second variable in $\rho_i$, the first variable in $\rho_j$ and the second variable in $\rho_j$. Let $(a,b,c,d) \in \mathcal{V}_4$. Recall that the coefficient of $\rho_i(x_a, x_b)\rho_j(x_c, x_d)$ in~\eqref{main4decompEqnReg3} is $\mu_{\ssk{i,j\\a,b,c,d}}$. Our goal is to show that this coefficient equals the one in the form
\[\Delta(x_{[4]}) = \Big(\alpha^{(3)}(x_{[4]}) + \alpha^{(2)}(x_{[4]})\Big) + \Big(\alpha^{(3)} \circ (3\,\,4)(x_{[4]}) + \alpha^{(2)} \circ (3\,\,4)(x_{[4]})\Big)\]
arising from definition~\eqref{alpha3defn4vars}. We look at different cases on places for $x_3$ and $x_4$.\\
\noindent\emph{Case 1: $x_3,x_4$ are at places $A, B$.} We have $\{a,b\} = \{3,4\}$ and $\{c,d\} = \{1,2\}$. The coefficient of $\rho_i(x_a, x_b)\rho_j(x_c, x_d)$ in $\Delta(x_{[4]})$ equals $\overline{\mu}_{i,j, A} + \overline{\mu}_{i,j, B} = \lambda_{A\,\,D} + \lambda_{B\,\,D} = \lambda_{A\,\,B} = \mu_{\ssk{i,j\\a,b,c,d}}$, where we used~\eqref{placesEQN4vars} for places $A, B$ and $D$ in the second equality.\\
\noindent\emph{Case 2: $x_3,x_4$ are at places $A, C$.} We have $\{a,c\} = \{3,4\}$ and $\{b,d\} = \{1,2\}$. The coefficient of $\rho_i(x_a, x_b)\rho_j(x_c, x_d)$ in $\Delta(x_{[4]})$ equals $\overline{\mu}_{i,j, A} + \overline{\mu}_{i,j, C} = \lambda_{A\,\,D} + \lambda_{C\,\,D} = \lambda_{A\,\,C} = \mu_{\ssk{i,j\\a,b,c,d}}$, where we used~\eqref{placesEQN4vars} for places $A, C$ and $D$ in the second equality.\\
\noindent\emph{Case 3: $x_3,x_4$ are at places $A, D$.} We have $\{a,d\} = \{3,4\}$ and $\{b,c\} = \{1,2\}$. The coefficient of $\rho_i(x_a, x_b)\rho_j(x_c, x_d)$ in $\Delta(x_{[4]})$ equals $\overline{\mu}_{i,j, A} = \lambda_{A\,\,D} = \mu_{\ssk{i,j\\a,b,c,d}}$.\\
\noindent\emph{Case 4: $x_3,x_4$ are at places $B, C$.} We have $\{b,c\} = \{3,4\}$ and $\{a,d\} = \{1,2\}$. The coefficient of $\rho_i(x_a, x_b)\rho_j(x_c, x_d)$ in $\Delta(x_{[4]})$ equals $\overline{\mu}_{i,j, B} + \overline{\mu}_{i,j, C} = \lambda_{B\,\,D} + \lambda_{C\,\,D} = \lambda_{B\,\,C} = \mu_{\ssk{i,j\\a,b,c,d}}$, where we used~\eqref{placesEQN4vars} for places $B, C$ and $D$ in the second equality.\\
\noindent\emph{Case 5: $x_3,x_4$ are at places $B, D$.} We have $\{b,d\} = \{3,4\}$ and $\{a,c\} = \{1,2\}$. The coefficient of $\rho_i(x_a, x_b)\rho_j(x_c, x_d)$ in $\Delta(x_{[4]})$ equals $\overline{\mu}_{i,j, B} = \lambda_{B\,\,D} = \mu_{\ssk{i,j\\a,b,c,d}}$.\\
\noindent\emph{Case 6: $x_3,x_4$ are at places $C, D$.} We have $\{c,d\} = \{3,4\}$ and $\{a,b\} = \{1,2\}$. The coefficient of $\rho_i(x_a, x_b)\rho_j(x_c, x_d)$ in $\Delta(x_{[4]})$ equals $\overline{\mu}_{i,j, C} = \lambda_{C\,\,D} = \mu_{\ssk{i,j\\a,b,c,d}}$.\end{proof}

\noindent\textbf{Step 5. Coefficients $\nu_{\ssk{i,j\\a,b,c,d}}$.} If two of the places $A, B$ and $C$ in~\eqref{placesEQN4vars} are the same, then without loss of the generality $A= B$ is the place in a symmetric form $\sigma_i$ and $C$ is one of the places in an asymmetric form $\rho_j$. Assume that $C$ corresponds to the first variable in $\rho_j$, the other case is similar. Equality~\eqref{placesEQN4vars} reduces to showing that $\lambda_{A\,\,A} = 0$ in this case. To see this, consider the coefficient of the product $\sigma_i(x_3, x_4)\rho_j(x_1, x_2)$ in $\phi^{(3)} + \phi^{(3)} \circ (1\,\,3) + \phi^{(3)}\circ(1\,\,4) = 0$. This coefficient equals
\[\lambda_{A\,\,A} + \lambda_{A\,\,C} + \lambda_{A\,\,C}\]
and vanishes, as desired.\\
\indent On the other hand, if $A, B$ and $C$ are different, then we may assume that $A$ is the place in $\sigma_i$ and $B$ and $C$ are respectively the first and the second variable in $\rho_j$. This time consider the coefficient of the product $\sigma_i(x_3, x_2)\rho_j(x_4, x_1)$ in $\phi^{(3)} + \phi^{(3)} \circ (1\,\,3) + \phi^{(3)}\circ(1\,\,4) = 0$. This coefficient equals
\[\lambda_{A\,\,B} + \lambda_{B\,\,C} + \lambda_{A\,\,C}\]
and vanishes, as desired.\\
\indent For each $i \in [s],j \in [t]$ let $\overline{\mu}_{i,j, B} = \lambda_{A\,\,B}$ and $\overline{\mu}_{i,j, C} = \lambda_{A\,\,C}$, where $A, B$ and $C$ are respectively the places corresponding to $\sigma_i$, the first variable in $\rho_j$ and the second variable in $\rho_j$. We may now define $\alpha^{(4)} \colon G^4 \to \mathbb{F}_2$ by 
\begin{align*}\alpha^{(4)}(x_1, x_2, x_3, x_4) =\alpha^{(3)}(x_1,& x_2, x_3, x_4) \\
+ &\sum_{i \in [s], j \in [t]} \overline{\mu}_{i,j, B}\Big(\sigma_i(x_1, x_2)\rho_j(x_4, x_3) + \sigma_i(x_2, x_3)\rho_j(x_4, x_1) + \sigma_i(x_3, x_1)\rho_j(x_4, x_2)\Big)\\
+ &\sum_{i \in [s], j \in [t]} \overline{\mu}_{i,j, C}\Big(\sigma_i(x_1, x_2)\rho_j(x_3, x_4) + \sigma_i(x_2, x_3)\rho_j(x_1, x_4) + \sigma_i(x_3, x_1)\rho_j(x_2, x_4)\Big).\end{align*} 
We show that $\alpha^{(4)}|_{U'\tdt U'}$ is symmetric.

\begin{claim}\label{4varsfinalsymm}The form $\alpha^{(4)}|_{U'\tdt U'}$ is symmetric.\end{claim}

\begin{proof}[Proof of Claim~\ref{4varsfinalsymm}]Let us reuse the notation $\Delta$, this time defining it as
\[\Delta(x_{[4]}) = \Big(\alpha^{(4)}(x_{[4]}) + \alpha^{(3)}(x_{[4]})\Big) + \Big(\alpha^{(4)} \circ (3\,\,4)(x_{[4]}) + \alpha^{(3)} \circ (3\,\,4)(x_{[4]})\Big).\]
Let $i \in [s], j \in [t]$ and $(a,b,c,d) \in \mathcal{V}_5$. We need to show that the coefficient of $ \sigma_i(x_a, x_b) \rho_j(x_c, x_d)$ in $\Delta(x_{[4]})$ equals $\nu_{\ssk{i,j\\a,b,c,d}} $. Let $A, B$ and $C$ be respectively the places corresponding to $\sigma_i$, the first variable in $\rho_j$ and the second variable in $\rho_j$.  As in the proof of Claim~\ref{4varsphi3claim}, we consider the cases for the places of $x_3$ and $x_4$.\\
\noindent\emph{Case 1: $x_3,x_4$ are both at place $A$.} We have $\{a,b\} = \{3,4\}$ and $\{c,d\} = \{1,2\}$. The coefficient of $\sigma_i(x_a, x_b)\rho_j(x_c, x_d)$ in $\Delta(x_{[4]})$ vanishes and we know that $0 = \lambda_{A\,A} = \nu_{\ssk{i,j\\a,b,c,d}}$, as desired.\\
\noindent\emph{Case 2: $x_3,x_4$ are at places $A, B$.} We have $\{b,c\} = \{3,4\}$ and $\{a,d\} = \{1,2\}$ (recalling that $\mathcal{V}_5$ is a proper subset of $\on{Sym}_{[4]}$). The coefficient of $\sigma_i(x_a, x_b)\rho_j(x_c, x_d)$ in $\Delta(x_{[4]})$ is $\overline{\mu}_{i,j, B} = \lambda_{A\,\,B} = \nu_{\ssk{i,j\\a,b,c,d}}$.\\
\noindent\emph{Case 3: $x_3,x_4$ are at places $A, C$.} We have $\{b,d\} = \{3,4\}$ and $\{a,c\} = \{1,2\}$ (again, recalling that $\mathcal{V}_5$ is a proper subset of $\on{Sym}_{[4]}$). The coefficient of $\sigma_i(x_a, x_b)\rho_j(x_c, x_d)$ in $\Delta(x_{[4]})$ is $\overline{\mu}_{i,j, C} = \lambda_{A\,\,C} = \nu_{\ssk{i,j\\a,b,c,d}}$.\\
\noindent\emph{Case 4: $x_3,x_4$ are at places $B, C$.} We have $\{c,d\} = \{3,4\}$ and $\{a,b\} = \{1,2\}$. The coefficient of $\sigma_i(x_a, x_b)\rho_j(x_c, x_d)$ in $\Delta(x_{[4]})$ is $\overline{\mu}_{i,j, B} + \overline{\mu}_{i,j, C} = \lambda_{A\,\,B} + \lambda_{A\,\,C}= \lambda_{B\,\,C}= \nu_{\ssk{i,j\\a,b,c,d}}$, where we used~\eqref{placesEQN4vars} in the second equality.\end{proof}

We have thus obtained a multilinear form $\alpha^{(4)}|_{U' \tdt U'}$ which is symmetric in on $U' \tdt U'$ such that $\on{prank}\Big((\alpha + \alpha^{(4)})|_{U' \tdt U'}\Big) \leq O(r^2)$. Theorem follows from Corollary~\ref{symmetryInher} (recall that $\on{codim}_G U' \leq O(Rr)$ which is the main contribution to the final bound). As in the proof of previous theorem, we choose $C$ and $D$ depending only to the constants from Lemma~\ref{zeromlformindexlemma} so that the partition rank condition of Lemma~\ref{zeromlformindexlemma} holds at all times.\end{proof}

\vspace{\baselineskip}
As a final result in this section we prove Theorem~\ref{mainSymmExtn5}.

\begin{proof}[Proof of Theorem~\ref{mainSymmExtn5}] By assumptions, we have that
\begin{equation}\alpha(x_{[5]})+ \alpha \circ (3\,\,4)(x_{[5]}) = \sum_{i \in [r]} \beta_{i, 1}(x_{I_{i,1}}) \beta_{i, 2}(x_{I_{i,2}}) \dots \beta_{i, d_i}(x_{I_{i,d_i}})\label{decompEqn}\end{equation}
where all partitions of the variables appearing above are non-trivial, that is $d_i \geq 2$.\\

\noindent\textbf{Step 1. Obtaining a decomposition with inherited properties.} We begin our work by deducing some properties of the low partition rank decomposition~\eqref{decompEqn} of $\alpha + \alpha \circ (3\,\,4)$. Using Proposition~\ref{lcmlformsdecomppropo} we may assume that every $\beta_{i, j}$ comes from a slice of $\alpha + \alpha \circ (3\,\,4)$. This has the cost of replacing $r$ by $O({r}^{O(1)})$, which we do with a slight misuse of notation. Writing $\beta(x_I)$ for $\beta_{i, j}(x_{I_{i,j}})$, we know that $\beta(x_I) = \alpha(x_I, y_{[5] \setminus I}) + \alpha \circ (3\,\,4)(x_I, y_{[5] \setminus I})$ for some fixed $y_{[5] \setminus I}$. In particular, the following holds.
\begin{itemize}
\item[\textbf{(S1)}] When $3, 4 \in I$, then $\beta(x_I) = \alpha_{y_{[5] \setminus I}}(x_I) + \alpha_{y_{[5] \setminus I}}\circ (3\,\,4)(x_I)$. Thus, we have the equality $\beta = \tilde{\beta} + \tilde{\beta} \circ (3\,\,4)$, where $\tilde{\beta}(x_I)$ is a multilinear form which is symmetric in the variables $x_{I \cap [3]}$. Therefore, replacing $\beta$ by $\tilde{\beta} + \tilde{\beta} \circ (3\,\,4)$, we may assume that every form $\beta(x_I)$ that has $3, 4 \in I$ is either symmetric in variables $x_{I \cap [3]}$ or symmetric in variables $x_{I \cap (\{1,2,4\})}$.
\item[\textbf{(S2)}] When $|I \cap \{3, 4\}| = 1$, say $\{3\} = I \cap \{3, 4\}$, then we see that $\beta$ is symmetric in variables $x_{I\cap [2]}$.
\item[\textbf{(S3)}] When $3, 4 \notin I$, then $\beta$ is symmetric in variables $x_{I \cap [2]}$.
\end{itemize}
We may therefore assume that all forms in the products in~\eqref{decompEqn} have a symmetry property described in one of \textbf{(S1)}, \textbf{(S2)} and \textbf{(S3)}.\\

\noindent\textbf{Step 2. Simplifying the partitions.} Observe that all non-trivial partitions of the set $[5]$ have either two sets which have sizes 3 and 2, or at least one set is a singleton. We refer to the former partitions as the \emph{relevant} ones. In particular, each summand in~\eqref{decompEqn} either has a relevant partition or it has a factor $\gamma(x_c)$ for some linear form $\gamma$. By passing to subspace defined by zero sets of such linear forms, we may assume that all partitions are in fact relevant.\\ 

\noindent\textbf{Step 3. Regularizing the forms.} In this step we apply our symmetry-respecting regularity lemma (Lemma~\ref{symmRespWRegLemma}).\\

\noindent Our goal is to remove some partitions of variables from expression~\eqref{decompEqn}. We shall first remove all relevant partitions $I_1 \cup I_2 = [5]$ such that $|I_1| = 2$ and $5 \in I_1$, and secondly we remove all relevant partitions $I_1 \cup I_2 = [5]$ such that $|I_1| = 3$ and $5 \in I_1$. Using slightly more general notation, we explain how to apply Lemma~\ref{symmRespWRegLemma} in either of the two cases. We say that the variables $x_1, \dots, x_4$ are \emph{active} and that $x_5$ is the \emph{passive} variable.\\
\indent Let $(\ell_1, \ell_2) \in \{(1,3), (2,2)\}$. The numbers $\ell_1$ and $\ell_2$ stand for the number of active variables in sets $I_1$ and $I_2$. We aim to remove the set $\mathcal{R}$ of all partitions $(A_1 \cup \{5\}, A_2)$ where $|A_i| = \ell_i$. Write $P_1 = \{5\}$ and $P_2 = \emptyset$, which stand for the indices of passive variables, and make the following list. Take all forms $\beta(x_{A \cup P_i})$ appearing in the decomposition~\eqref{decompEqn} such that $A \subseteq [4]$ of size $|A| = \ell_i$. By the symmetry properties \textbf{S1}, \textbf{S2} and \textbf{S3} we see that each of the chosen forms $\beta$ is equal to
\begin{itemize}
\item $\gamma(x_S, x_{P_i})$ where $\gamma$ is symmetric in $x_S$, $|S| = \ell_i$, with $S \subseteq [2]$, or
\item $\gamma(x_S, x_c, x_{P_i})$ where $\gamma$ is symmetric in $x_S$, $|S| = \ell_i - 1$, with $S \subseteq [2]$, $c \in \{3, 4\}$ or 
\item $\gamma(x_S, x_c, x_{P_i})$ where $\gamma$ is symmetric in $x_S$, $|S| = \ell_i - 1$, with $c \in \{3, 4\}$, $S \subseteq [4] \setminus \{c\}$.
\end{itemize}
Consider all $\gamma$ above and replace variables $x_S$ by $z_1, \dots, z_{\ell_i}$ when $x_c$ does not appear, and $x_S$ by $z_1, \dots, z_{\ell_i - 1}$ and $x_c$ by $z_{\ell_i}$ when $x_c$ appears. Every form is thus symmetric in $z_{[\ell_i - 1]}$. Let $C, D \geq 2$ quantities to be specified later, which will depend only on the constants in Lemma~\ref{zeromlformindexlemma}. Apply Lemma~\ref{symmRespWRegLemma} to the list. We thus obtain integers $s,q,t$ such that $s + q + t \leq r$ and multilinear forms $\sigma_{[s]}, \pi_{[q]}, \rho_{[t]}$ with properties \textbf{(i)}-\textbf{(v)} from that lemma for a parameter $R \leq C^{D^{O(r)}}$. We add superscript $\sigma^{i, \ell_i}_1, \dots$ to stress the dependence on $i$ and $\ell_i$. Note that by Remark~\ref{symmRespWRegLemmaRemark} when $\ell_i = 2$ we only have forms $\sigma_{[s]}$ and $\rho_{[t]}$ and when $\ell_i = 1$ we only have forms $\sigma_{[s]}$. Furthermore, since we are applying Lemma~\ref{symmRespWRegLemma} to forms in at most 3 variables, low partition rank means that we may pass to a subspace where we get exact properties. Thus by passing to a further subspace $U$, we may assume that forms $\sigma_{[s]}$ are symmetric in their active variables. To be precise, the subspace $U$ has codimension at most $3rR + r$ inside $G$ (recall that we have first passed to a subspace of codimension at most $r$ in order to remove the irrelevant partitions).\\

Note that the forms $\pi_i$ and $\rho_i$ have the property that they are symmetric in the first $\ell_i - 1$ variables, the variables $x_{P_i}$ play a passive role and $z_{\ell_i}$ is related to $z_{[\ell_i-1]}$ in the sense that we shall consider compositions with transpositions $(a\,\,\ell_i)$. With this in mind we adopt the following notation. For a multilinear form $\phi(z_{[\ell_i]}, x_{P_i})$ we separate its variables as $\phi(z_{[\ell_i - 1]}\,|\,z_{\ell_i}\,|\,x_{P_i})$ where the first group indicated the symmetric part, the second group has a single variable and the third group consists of passive variables. We refer to the variable in the second group as the \emph{asymmetric variable}.\\

Replace forms $\beta(x_I, x_{P_a})$ where $I \subset [4]$, $|I| = \ell_a$, by the newly found maps using property \textbf{(v)} of Lemma~\ref{symmRespWRegLemma}. If $3, 4 \notin I$ then we know that $\beta$ is symmetric in $x_I$ so the second part of Observation~\ref{symmetriclinearcombinations} implies that $\beta(x_I, x_{P_a})$ can be replaced by 
\[\sum_{i \in [s]} \lambda_i \sigma_i(x_I, x_{P_a}) + \sum_{i \in [t], c' \in I} \lambda''_{i, c'} \rho_i(x_{I \setminus c'}\,|\,x_{c'}\,|\, x_{P_a}) + \tau(x_I, x_{P_a})\]
where $\tau(x_I, x_{P_a})$ is a multilinear form of partition rank at most $(r+2)R$. On the other hand, if $I \cap \{3, 4\} \not= \emptyset$ take $c \in I\cap \{3, 4\}$ to be an index such that $\beta$ is symmetric in $x_{I \setminus \{c\}}$ . Properties \textbf{(iii)} and \textbf{(v)} of Lemma~\ref{symmRespWRegLemma} and Observation~\ref{symmetriclinearcombinations} imply that $\beta(x_I, x_{P_a})$ can be replaced by 
\[\sum_{i \in [s]} \lambda_i \sigma_i(x_I, x_{P_a}) + \sum_{i \in [q]} \lambda'_i \pi_i(x_{I \setminus c}\,|\,x_c\,|\, x_{P_a}) + \sum_{i \in [t], c' \in I} \lambda''_{i, c'} \rho_i(x_{I \setminus c'}\,|\,x_{c'}\,|\, x_{P_a}) + \tau(x_I, x_{P_a})\]
where $\tau(x_I, x_{P_a})$ is a mutlilinear form of partition rank at most $(r + 2)R$. Hence, forms $\pi_i$ always appear with $x_3$ or $x_{4}$ as the asymmetric variable.\\

In order to express $\alpha(x_{[5]})+ \alpha \circ (3\,\,4)(x_{[5]})$ in terms of these multilinear forms we need to set up further notation. We use richer indices than just natural numbers, of the form $(i, j, \mathsf{S}, A)$, $(i, j, \mathsf{PS}, A, v)$ and $(i, j, \mathsf{AS}, A, v)$, where $i \in [2]$, $j$ is a suitable index, $A \subseteq [4]$ of size $|A| = \ell_i$, letters $\mathsf{S}$, $\mathsf{PS}$ and $\mathsf{AS}$ indicate the type of form we take, and $v \in A$. Using such an index $\bm{\mathsf{i}}$, we define form $\psi_{\bm{\mathsf{i}}}$ as a multilinear form in variables $x_{A \cup P_i}$ (which are ordered by the values of indices) and
\[\psi_{\bm{\mathsf{i}}}(x_{A \cup P_i}) = \begin{cases}
\sigma^{i, \ell_i}_j (x_A, x_{P_i}), &\text{ if }\bm{\mathsf{i}} = (i, j, \mathsf{S}, A),\\
\pi^{i, \ell_i}_j (x_{A \setminus v} \,|x_v\, |\, x_{P_i}), &\text{ if }\bm{\mathsf{i}} = (i,j, \mathsf{PS}, A, v),\\
\rho^{i, \ell_i}_j (x_{A \setminus v} \,|x_v\,|\,x_{P_i}), &\text{ if }\bm{\mathsf{i}} = (i,j, \mathsf{AS}, A, v).
\end{cases}\]

Let $A(\bsf{i})$ be the set $A$ in the rich index $\bsf{i}$. Let $\mathcal{I}$ be the set of pairs $(\bsf{i}_1,\bsf{i}_2)$ such that 
\begin{itemize}
\item $\bsf{i}_1 = (1, \dots)$ and $\bsf{i}_2 = (2, \dots)$, and
\item $A(\bsf{i}_1) \cup A(\bsf{i}_2)$ is a partition of $[4]$.
\end{itemize}
By our work so far, we have suitable scalars $\lambda_{\bsf{i}_1, \bsf{i}_2}$ such that
\begin{equation}\alpha(x_{[5]})+ \alpha \circ (3\,\,4)(x_{[5]})  = \sum_{(\bsf{i}_1, \bsf{i}_2) \in \mathcal{I}} \lambda_{\bsf{i}_1, \bsf{i}_2} \psi_{\bsf{i}_1}(x_{A(\bsf{i}_1) \cup P_1})\psi_{\bsf{i}_2}(x_{A(\bsf{i}_2) \cup P_2}).\label{decompEqn3addlsymm}\end{equation}

In the rest of the proof we show how to make $\lambda_{\bsf{i}_1, \bsf{i}_2}$ vanish. More precisely, we shall find a multilinear form $\mu(x_{[5]})$ which will be given by an appropriate linear combination of products related to those appearing on the right hand side of~\eqref{decompEqn3addlsymm} and $\alpha + \mu$ will have the desired property. We refer to $\mu(x_{[5]})$ as the \emph{modification term}.\\
\indent To that end, we fix a pair of multilinear forms $\beta_1,  \beta_2$, where $\beta_i$ is one of the forms in the list $\sigma^{i, \ell_i}_j, \pi^{i, \ell_i}_j, \rho^{i, \ell_i}_j$. We consider together all pairs of rich indices $(\bsf{i}_1, \bsf{i}_2) \in \mathcal{I}$ that $\psi_{\bm{\mathsf{i}}_i}$ comes from $\beta_i$ for each $i \in [2]$. Our goal is to describe contribution to the modification term $\mu(x_{[5]})$ coming from the choice $\beta_1, \beta_2$.\\

Note that the form $\alpha(x_{[5]})+ \alpha\circ (3\,\,4)(x_{[5]})$ is symmetric in variables $x_{[2]}$ and in variables $x_{\{3, 4\}}$. Let $\tau \in \on{Sym}_{[2]} \times \on{Sym}_{\{3, 4\}}$ be any permutation acting on these two sets. The fact that 
\[\Big(\alpha(x_{[5]})+ \alpha \circ (3\,\,4)(x_{[5]})\Big) + \Big(\alpha(x_{[5]})+ \alpha \circ (3\,\,4)(x_{[5]})\Big) \circ \tau = 0,\]
identity~\eqref{decompEqn3addlsymm}, properties of Lemma~\ref{symmRespWRegLemma} and Lemma~\ref{zeromlformindexlemma} imply the following fact.

\begin{claim}\label{basicCoeffsSymmetryFact}Let $(\bsf{i}_1, \bsf{i}_2) \in \mathcal{I}$ be a rich index pair giving a product of forms $\Pi(x_{[5]})$ and let $\tau \in \on{Sym}_{[2]} \times \on{Sym}_{\{3, 4\}}$ be a permutation. Let $(\bsf{j}_1, \bsf{j}_2)$ be the rich index pair corresponding to the product $\Pi \circ \tau(x_{[5]})$. Then we have $\lambda_{\bsf{i}_1,  \bsf{i}_2} = \lambda_{\bsf{j}_1,  \bsf{j}_2}$.\end{claim}

\vspace{\baselineskip}

\noindent\textbf{Step 4. Simplifying the expression -- removing partially symmetric forms $\pi_i$.} Note that if the partial symmetric forms play a role, then some $\ell_i \geq 3$. Since we have only 4 active variables, there can only be one partially symmetric form present in the product.\\
\indent Write $\phi = \alpha + \alpha \circ (3\,\,4)$ and recall that $\phi$ satisfies 
\begin{equation}\label{phiIdentity} \phi + \phi \circ (1\,\,3) + \phi\circ(1\,\,4) = 0 \end{equation}
by Lemma~\ref{weaksymmetryextends}.\\

\indent In this step of the proof, we define a multilinear form $\tilde{\mu}(x_{[5]})$ as the sum of all terms in~\eqref{decompEqn3addlsymm} that have a partially symmetric form in the product, with variable $x_{4}$ at its asymmetric place. Thus $\on{prank}  \tilde{\mu} \leq O(r^2)$. We then add $\tilde{\mu}$ to the modification term $\mu$. Let $P(x_{[5]})$ be the contribution of the products involving a partially symmetric form in~\eqref{decompEqn3addlsymm}. By Claim~\ref{basicCoeffsSymmetryFact}, we have that
\[\tilde{\mu} + \tilde{\mu}\circ(3\,\,4) = P.\]
We now show that $\tilde{\mu}$ is symmetric in the first three variables.\\
\indent Recall that we fixed forms $\beta_1$ and $\beta_2$. Suppose that $\beta_a$ is the partially symmetric form among them and let $b$ be the other index in $\{1,2\}$. Thus, $\beta_b$ is a form among $\sigma^{b, \ell_b}$ and $\rho^{b, \ell_b} \circ (c\,\,c')$, and $\beta_a$ is $\pi^{a, \ell_a}_\ell$ for some index $\ell$. We need to show that coefficients of all products of $\beta_a$ with $x_4$ as its asymmetric variable and $\beta_b$ are the same. Observe that $\beta_a$ has at least 3 active variables, as otherwise it cannot be partially symmetric. Since the partitions we consider are relevant, this implies that $P_a = \emptyset$ and $P_b = \{5\}$. Hence, we need to show that the products
\begin{align*}\Pi_1(x_{[5]}) = \pi^{a, \ell_a}_\ell(x_2, x_3 |x_4) \beta_b(x_1, x_5),\,\, &\Pi_2(x_{[5]}) = \pi^{a, \ell_a}_\ell(x_1, x_3 |x_4) \beta_b(x_2, x_5),\\
&\hspace{2cm}\Pi_3(x_{[5]}) = \pi^{a, \ell_a}_\ell(x_1, x_2 |x_4) \beta_b(x_3, x_5)\end{align*}
all have the same coefficient in~\eqref{decompEqn3addlsymm}.\\
\indent Notice that coefficients of $\Pi_1(x_{[5]})$ and $\Pi_2(x_{[5]})$ are the same by Claim~\ref{basicCoeffsSymmetryFact}. On the other hand, the coefficient of $\Pi_1(x_{[5]})$ in~\eqref{phiIdentity} is zero by Lemma~\ref{zeromlformindexlemma}. The contributing products to this coefficient from~\eqref{decompEqn3addlsymm} are $\Pi_1(x_{[5]})$ and $\Pi_1(x_{[5]}) \circ (1 \,\,3) = \Pi_3(x_{[5]})$ (note that $\Pi_1(x_{[5]}) \circ (1 \,\, 4)$ does not have either of $x_{3}$ and $x_{4}$ at the asymmetric place, so it cannot appear in $\phi$). This proves that the coefficients of products $\Pi_1(x_{[5]}), \Pi_2(x_{[5]})$ and $\Pi_3(x_{[5]})$ are equal, showing that $\tilde{\mu}$ is symmetric in $x_{[3]}$.\\

\noindent\textbf{Step 5. Simplifying the expression -- removing the remaining products.} Let forms $\beta_1$ and $\beta_2$ now be the forms which are not partially symmetric. Similarly to the proof of Theorem~\ref{mainSymmExtn4}, we now introduce a set of objects called \emph{places}, each describing a possible position of an active variable in the product. For each symmetric form above we have one place, and for each asymmetric form we have two places, one for the variables is in the symmetric part of the form and another at the asymmetric position in the form. Using the symmetries in $x_{[2]}$ and $x_{\{3, 4\}}$ variables we see that the coefficients of all considered products depend entirely on the choice of the places for $x_{3}$ and $x_{4}$. More precisely, by Claim~\ref{basicCoeffsSymmetryFact}, given two places $A$ and $B$, we have coefficients $\lambda_{A\,A}$ and $\lambda_{A\,B}$ such that
\begin{itemize}
\item all products where $x_{3}$ and $x_{4}$ are at the same place $A$ (thus both variables occur in same $\sigma$ or symmetric part of the same $\rho$) get the coefficient $\lambda_{A\,A}$,
\item all products where $x_{3}$ and $x_{4}$ are at places $A$ and $B$ which occurs at the same form (e.g.\ $\rho(x_1, x_{3}\,|\,x_{4})\dots$; note that this includes both cases of $x_{3}$ being at $A$ and $x_{3}$ being at $B$) get the coefficient $\lambda_{A\,B}$,
\item all products where $x_{3}$ and $x_{4}$ are at places $A$ and $B$ which occurs at different forms (e.g.\ $\rho(x_1\,|\,x_{4}) \sigma(x_2, x_{3},x_5)$; note that this includes both cases of $x_{3}$ being at $A$ and $x_{3}$ being at $B$) get the coefficient $\lambda_{A\,B}$. (Note that we get the same form of the coefficient $\lambda_{A\,B}$ as in the previous case, the difference is in the pairs of places $A$ and $B$ for which this case occurs.)\end{itemize}

Since all partitions of variables appearing are relevant we know that one form has 3 variables and the other has 2 variables, so we never have two copies of the same form in a product. Thus, specifying places $A$ and $B$ for $x_3, x_4$ essentially determines the product (up to permutation in $\on{Sym}_{\{1,2\}} \times \on{Sym}_{\{3,4\}}$, which does not affect the coefficient of the product). Our goal now is describe a multilinear form $\tilde{\mu}(x_{[k]})$ that will be added to the modification term $\mu(x_{[k]})$ so that $(\alpha + \tilde{\mu}) + (\alpha + \tilde{\mu}) \circ (3\,\,4)$ has no products of forms $\beta_1$ and $\beta_2$.\\

We have three different cases depending on the nature of the forms: either both are symmetric, or one is symmetric and the other asymmetric, or both are asymmetric.\\

\noindent\textbf{Symmetric case.} Let the two forms be $\sigma$ and $\sigma'$ with their places $A$ and $B$. Without loss of generality $\sigma$ has at least two active variables. We claim that $\lambda_{A\,A} = 0$. Swithching the roles of $\sigma$ and $\sigma'$ if $\sigma'$ also has at least two acitve variables shows that $\lambda_{B\,B} = 0$. To see that $\lambda_{A\,A}$ vanishes, consider the coefficient of product $\sigma(x_{3}, x_1, \dots) \sigma'(x_{4}, \dots)$ in~\eqref{phiIdentity}. (Note that the positions of the remaining active variable $x_2$ and the passive variables $x_{P_1}$ and $x_{P_2}$ in these forms are uniquely determined, but we opt not to display them for the sake of clarity.) This coefficient vanishes, but at the same time comes from coefficients of products 
\[\sigma(x_{3}, x_1, \dots) \sigma'(x_{4}, \dots),\,\, \sigma(x_{1}, x_{3}, \dots) \sigma'(x_{4}, \dots)\,\,\text{and}\,\,\sigma(x_{3}, x_4, \dots) \sigma'(x_{1}, \dots)\]
in~\eqref{decompEqn3addlsymm}. Thus
\[\lambda_{A\,B} + \lambda_{A\,B} + \lambda_{A\,A} = 0\]
giving $\lambda_{A\,A} = 0$.\\
\indent We may then set $\tilde{\mu}(x_{[5]})$ to be the sum of $\lambda_{A\,B} \sum_{I \in \binom{[3]}{\ell_1}} \beta_1(x_I, x_{P_1}) \allowbreak\beta_2(x_{[3] \setminus I}, x_{4}, x_{P_2})$ for all pairs $\sigma$ and $\sigma'$ appearing in this case (with place $A$ at $\sigma$ and $B$ at $\sigma'$), where $\beta_1, \beta_2$ is our initial notation for these two maps. By definition, $\tilde{\mu}$ is symmetric in $x_{[3]}$ and the vanishing of $\lambda_{C\,C}$ for the relevant places $C$ implies that $\tilde{\mu}(x_{[5]}) + \tilde{\mu} \circ (3\,\,4)(x_{[5]})$ equals the contribution from products of forms in this case.\\

\noindent\textbf{Mixed case.} Let the two forms be $\sigma$ and $\rho$. Let $A$ be the place in $\sigma$, let $B$ be the symmetric place in $\rho$ and $C$ the asymmetric one. As in the previous case, we use notation for products that does not specify positions of $x_2$ and $x_5$ for the sake of clarity, as these are uniquely determined. Considering the products $\sigma(x_3, x_{4}, \dots) \rho(x_1, \dots |\cdot)$ (which makes sense only when there are at least two active variables in $\sigma$), $\sigma(x_1, \dots) \rho(x_3, x_{4}, \dots |\cdot)$ (which makes sense only when there are at least two variables in the symmetric part of $\rho$) and $\sigma(x_1, \dots) \rho(x_3, \dots |x_{4})$ (meaning that $x_4$ is at the place $C$) respectively in~\eqref{phiIdentity} gives
\[0 = \lambda_{A\,A} + \lambda_{A\,B} + \lambda_{A\,B} = \lambda_{B\,B} + \lambda_{A\,B} + \lambda_{A\,B} = \lambda_{B\,C} + \lambda_{A\,C} + \lambda_{A\,B}\]
thus 
\begin{equation}\lambda_{A\,A} = \lambda_{B\,B} = \lambda_{B\,C} + \lambda_{A\,C} + \lambda_{A\,B} = 0.\label{mixedcasecoeffs}\end{equation}
We now describe the contribution $\tilde{\mu}(x_{[5]})$ to the modification term $\mu(x_{[5]})$. For this choice of forms $\sigma$ and $\rho$ as $\beta_1, \beta_2$ we add $\lambda_{A\,B} \Pi(x_{[5]})$ to $\tilde{\mu}(x_{[5]})$ for all products $\Pi$ of $\sigma$ and $\rho$ that have $x_4$ at the symmetric place $B$ of $\rho$ and we add $\lambda_{A\,C} \Pi(x_{[5]})$ to $\tilde{\mu}(x_{[5]})$ for all products $\Pi$ of $\sigma$ and $\rho$ that have $x_4$ at the asymmetric place $C$ of $\rho$. (Note that there are 3 different choice of $\Pi$ for both places and both cases whether $\sigma$ has 2 or 3 variables.) By definition, $\tilde{\mu}(x_{[5]})$ is symmetric in $x_{[3]}$. Finally, using the information on coefficients~\eqref{mixedcasecoeffs} we see that $\tilde{\mu}(x_{[5]}) + \tilde{\mu} \circ(3\,\,4)(x_{[5]})$ equals the contribution to~\eqref{decompEqn3addlsymm} coming from the mixed case, as desired.\\

\noindent\textbf{Asymmetric case.} Let the two forms be $\rho$ and $\rho'$ (in particular this implies that both forms have 2 active variables). Let the symmetric places in $\rho$ and $\rho'$ be $A$ and $C$ respectively, and let the asymmetric place be $B$ and $D$ respectively. Looking at products $\rho(x_1, \dots| \cdot) \rho'(x_3, \dots|x_{4}),$  $\rho(\cdots| x_1) \rho'(x_3, \dots|x_{4}),$ $\rho(x_3, \dots|x_{4})\rho'(x_1, \dots| \cdot)$, $\rho(x_3, \dots|x_{4})\rho'(\cdots|x_1)$, $\rho(x_3, x_{4}, \dots |\cdot) \rho'(x_1, \dots | \cdot)$ (if there are at least two variables in the symmetric part of $\rho$), $\rho(x_1, \dots | \cdot)\rho'(x_3, x_{4}, \dots |\cdot)$ (if there are at least two variables in the symmetric part of $\rho'$) from~\eqref{phiIdentity} we deduce that
\begin{align}0 = \lambda_{C\,D} + \lambda_{A\, D} + \lambda_{A\,C} = \lambda_{C\,D} + \lambda_{B\,D} + \lambda_{B\,C} = &\lambda_{A\,B} + \lambda_{B\, C} + \lambda_{A\,C} = \lambda_{A\,B} +\lambda_{B\,D}+\lambda_{A\,D}\nonumber\\
= &\lambda_{A\,A} +\lambda_{A\,C}+\lambda_{A\,C} = \lambda_{C\,C} + \lambda_{A\,C} + \lambda_{A\,C}.\label{asymmcasecoeffs}\end{align}
Let us now describe the final contribution $\tilde{\mu}(x_{[5]})$ to the modification term $\mu(x_{[5]})$. For this choice of forms $\rho$ and $\rho'$ as $\beta_1, \beta_2$ we add $\lambda_{A\,B} \Pi(x_{[5]})$ to $\tilde{\mu}(x_{[5]})$ for all products $\Pi$ of $\rho$ and $\rho'$ that have $x_4$ at the asymmetric place $B$ of $\rho$, we add $\lambda_{A\,C} \Pi(x_{[5]})$ to $\tilde{\mu}(x_{[5]})$ for all products $\Pi$ of $\rho$ and $\rho'$ that have $x_4$ at the symmetric place $C$ of $\rho'$ and we add $\lambda_{A\,D} \Pi(x_{[5]})$ to $\tilde{\mu}(x_{[5]})$ for all products $\Pi$ of $\rho$ and $\rho'$ that have $x_4$ at the asymmetric place $D$ of $\rho'$. (Note that there are 6 different choices of $\Pi$ for all three places and both cases whether $\rho$ has 2 or 3 variables.) By definition, $\tilde{\mu}(x_{[5]})$ is symmetric in $x_{[3]}$. Finally, using the information on coefficients~\eqref{asymmcasecoeffs} we see that $\tilde{\mu}(x_{[5]}) + \tilde{\mu} \circ(3\,\,4)(x_{[5]})$ equals the contribution to~\eqref{decompEqn3addlsymm} coming from the asymmetric case, as desired.\\

Finally, set $\tilde{\alpha} = \alpha + \mu$. Each of the contributions to the modification term $\mu$ is symmetric in $x_{[3]}$ and so is $\tilde{\alpha}$. But we chose $\mu$ so that $\tilde{\alpha} + \tilde{\alpha} \circ (3\,\,4) = 0$ on the subspace $U$, so in fact $\tilde{\alpha}$ is symmetric in $x_{[4]}$ on $U$. Since $\on{prank} \mu \leq O(r^2)$, the theorem follows from Corollary~\ref{symmetryInher} (as in the proof of the previous theorem, recall that $\on{codim}_G U \leq O(Rr)$ which is the main contribution to the final bound). As in the proof of Theorem~\ref{mainSymmExtn4}, we choose $C$ and $D$ depending only to the constants from Lemma~\ref{zeromlformindexlemma} so that the partition rank condition of Lemma~\ref{zeromlformindexlemma} holds each time the lemma is applied.\end{proof}

\section{Multilinear forms which are approximately without repeated coordinates}

This section is devoted to the proof of Theorem~\ref{copiesPropertyVanish}. As in the case of the previous theorems, we need a way to regularize the forms appearing in low partition rank decompositions. We need a preliminary lemma first.

\begin{lemma}\label{times2boundedRankLemma}For a positive integer $k$ there exist constants $C = C_k, D = D_k \geq 1$ such that the following holds. Suppose that $\alpha \colon G^k \to \mathbb{F}_2$ is symmetric in the first two variables and that $\alpha$ has partition rank at most $r$. Let $\alpha^{\times 2} \colon G^{k-1} \to \mathbb{F}_2$ be the multilinear form defined by $\alpha^{\times 2}(d, x_{[3,k]}) = \alpha(d,d, x_{[3,k]})$. Then $\alpha^{\times 2}$ has partition rank at most $C r^{D}$.\end{lemma}

\begin{proof}Let
\[\alpha(x_1, x_2, y_{[3,k]}) = \sum_{i \in [r_1]} \beta_i(x_1, x_2, y_{I_i}) \beta_i'(y_{I_i^c}) + \sum_{i \in [r_2]} \gamma_i(x_1, y_{J_i}) \gamma_i'(x_2, y_{J_i^c})\]
for some integers $r_1, r_2 \leq r$ and suitable multilinear forms $\beta_i, \beta'_i, \gamma_i, \gamma'_i$. Then we have
\[\alpha^{\times 2}(u, y_{[3,k]}) = \sum_{i \in [r_1]} \beta_i(u, u, y_{I_i}) \beta_i'(y_{I_i^c}) + \sum_{i \in [r_2]} \gamma_i(u, y_{J_i}) \gamma_i'(u, y_{J_i^c}).\]
For each $i \in [r_2]$, let $\tilde{\gamma}_i(u, y_{\tilde{J}_i})$ be the form among $\gamma_i(u, y_{J_i})$ and $\gamma_i'(y_{u, J_i^c})$ which does not involve variable $y_3$. In particular, we have the inclusion of varieties
\[\Big\{(u, y_{[3,k]}) \colon (\forall  i \in [r_1]) \beta'(y_{I_i^c}) = 0\Big\} \cap \Big\{(u, y_{[3,k]}) \colon (\forall i \in [r_2]) \tilde{\gamma}_i(u, y_{\tilde{J}_i}) = 0\Big\} \subseteq \{\alpha^{\times 2} = 0\}.\]
We thus obtain
\begin{align*}2^{-(k-1)r} \leq &\exx_{u, y_{[3,k]}} \prod_{i \in [r_1]}\id(\beta'(y_{I_i^c}) = 0) \prod_{i \in [r_2]} \id(\tilde{\gamma}_i(u, y_{\tilde{J}_i}) = 0) \\
= &\exx_{u, y_{[3,k]}} \prod_{i \in [r_1]}\id(\beta'(y_{I_i^c}) = 0) \prod_{i \in [r_2]}\id(\tilde{\gamma}_i(u, y_{\tilde{J}_i}) = 0)(-1)^{\alpha^{\times 2}(u, y_{[3,k]})}.\end{align*}
Using Gowers-Cauchy-Schwarz inequality, this quantity can be bounded from above by $\|(-1)^{\alpha^{\times 2}}\|_{\square^{k-1}} =\Big( \on{bias}( \alpha^{\times 2})\Big)^{2^{-{k-1}}}$. The lemma follows from Theorem~\ref{biasedinversethm}.\end{proof}

We may now state and prove the regularity lemma that we need.

\begin{lemma}\label{symmRegLemmaTimes2}Let $\alpha_1, \dots \alpha_r \colon G^k \to \mathbb{F}_2$ be multilinear forms that are symmetric in the first two variables. Let $M, R_0 \geq 1$ be positive quantities. Then, we can find a positive integer $R$ satisfying $R_0 \leq R \leq (2R_0 + 1)^{O(M)^{2r}}$ and subspaces $\Lambda \leq \Lambda^{\times 2} \leq \mathbb{F}_2^r$ such that
\begin{itemize}
\item[\textbf{(i)}] if $\lambda \in \Lambda$ then $\on{prank} \lambda \cdot \alpha \leq R$ and if $\lambda \notin \Lambda$ then $\on{prank} \lambda \cdot \alpha \geq (2R +1)^M$,
\item[\textbf{(ii)}] if $\lambda \in \Lambda^{\times 2}$ then $\on{prank} \lambda \cdot \alpha^{\times 2} \leq R$ and if $\lambda \notin \Lambda^{\times 2}$ then $\on{prank} \lambda \cdot \alpha^{\times 2} \geq (2R +1)^M$. 
\end{itemize}
\end{lemma}

\begin{proof} Let $C_i$ and $D_i$ be the constants from the previous lemma for multilinear forms in $i$ variables and set $C = \max_{i \in [k]} C_i, D = \max_{i \in [k]} D_i$ which we may assume to satisfy $C, D \geq 1$. At each step we keep track of two subspaces $\Lambda \leq \Lambda^{\times 2} \leq \mathbb{F}_2^r$ and a quantity $R$ such that for each $\lambda \in \Lambda$ we have $\on{prank} \lambda \cdot \alpha \leq R$ and for each $\lambda \in \Lambda^{\times 2}$ we have $\on{prank} \lambda \cdot \alpha^{\times 2} \leq R$. The procedure terminates once $\Lambda$ and $\Lambda^{\times 2}$ have the desired properties, that is when we have $\on{prank} \lambda \cdot \alpha \geq (2R +1)^M$ for each $\lambda \notin \Lambda$ and $\on{prank} \lambda \cdot \alpha^{\times 2} \geq (2R +1)^M$ for each $\lambda \notin \Lambda^{\times 2}$. We begin by setting $R = R_0$, $\Lambda = \Lambda^{\times 2} = \{0\}$.\\

Suppose that the procedure has not yet terminated. Suppose first that for some $\lambda \notin \Lambda$ we have $\on{prank} \lambda \cdot \alpha \leq (2R +1)^M$. By Lemma~\ref{times2boundedRankLemma} we have $\on{prank} \lambda \cdot \alpha^{\times 2} \leq C (2R +1)^{DM}$. Replace $R$ by $2C (2R +1)^{DM}$, $\Lambda$ by $\Lambda + \langle \lambda \rangle$ and $\Lambda^{\times 2}$ by $\Lambda^{\times 2} + \langle \lambda \rangle$. In the second case, we have $\on{prank} \lambda \cdot \alpha^{\times 2} \leq (2R +1)^M$ for some $\lambda \notin \Lambda^{\times 2}$. This time replace $R$ by $2(2R +1)^M$ and $\Lambda^{\times 2}$ by $\Lambda^{\times 2} + \langle \lambda \rangle$ and keep $\Lambda$ the same. The procedure thus terminates after at most $2r$ steps and we obtain the desired subspaces.\end{proof}

\begin{proof}[Proof of Theorem~\ref{copiesPropertyVanish}] We first do the case $k = 4$ which turns out to be rather simple. In this case, $\alpha(d,d, x_3, x_4)$ is a trilinear form of partition rank at most $r$, so there exists a subspace $U \leq G$ of codimension at most $r$ (given by the zero-set of the linear forms that appear as factors in the low partition rank decomposition of $\alpha(d,d, x_3, x_4)$) such that $\alpha(d,d, x_3, x_4) = 0$ for all $d, x_3, x_4 \in U$. We may take $\alpha' = \alpha|_{U \tdt U}$ to finish the proof in this case.\\

Now consider the case $k = 5$. Suppose that $\alpha(x_{[m]}, y_{[m+1,5})$ is symmetric in variables $x_{[m]}$ and that $\alpha(d,d,$ $x_{[3,m]}, y_{[m+1,5]})$ has partition rank at most $r$. Thus
\begin{equation}\label{2copiesApproxEqn1}\alpha(d,d,x_{[3,m]}, y_{[m + 1, 5]}) = \sum_{i \in [r]} \beta_{i, 0}(d, x_{I_{i, 0}}, y_{J_{i, 0}}) \beta_{i, 1}(x_{I_{i, 1}}, y_{J_{i, 1}})\end{equation}
holds for all $d, x_3, \dots, y_5 \in G$, where $\beta_{i, 0}$ and $\beta_{i, 1}$ are suitable multilinear forms.\\ 
We begin our proof by deducing that we may assume that every form $\beta_{i, 1}$ is symmetric in $x$ variables and that $\beta_{i, 0}(d, x_{I_{i, 0}}, y_{J_{i, 0}})$ equals $\tilde{\beta}(d, d, x_{I_{i, 0}}, y_{J_{i, 0}})$ for a suitable multilinear form $\tilde{\beta}$ which is symmetric in the first $|I_{i, 0}| + 2$ variables. Proposition~\ref{lcmlformsdecomppropo} allows us to assume that every $\beta_{i, j}$ is a slice of the form $(d,x_{[3,m]}, y_{[m + 1, k]}) \mapsto \alpha(d,d,x_{[3,m]}, y_{[m + 1, k]})$ (and therefore has the desired properties) at the cost of replacing $r$ by $O(r^{O(1)})$.\\
\indent Furthermore, if we take all linear forms appearing in~\eqref{2copiesApproxEqn1} and set $U$ to be the subspace of codimension at most $O(r^{O(1)})$ where they vanish, we may without loss of generality assume that all forms in~\eqref{2copiesApproxEqn1} are bilinear. We now consider each possible value of $m \in \{2,3,4,5\}$ separately.\\ 

\noindent\textbf{Case 1: $m = 2$.} In this case equality~\eqref{2copiesApproxEqn1} becomes
\[\alpha(d,d,y_{[3, 5]}) = \sum_{i \in [s]} \beta_{i, 0}(d, y_{J_{i, 0}}) \beta_{i, 1}(y_{J_{i, 1}})\]
for some $s\leq O(r^{O(1)})$. We know that $\beta_{i, 0}(d, y_{J_{i, 0}}) = \tilde{\beta}_{i, 0}(d, d, y_{J_{i, 0}})$ for a multilinear form $\tilde{\beta}_{i, 0}$ which is symmetric in the first two variables. Then simply set
\[\sigma(x_{[5]}) = \sum_{i \in [s]}  \tilde{\beta}_{i, 0}(x_1, x_2, x_{J_{i, 0}}) \beta_{i, 1}(x_{J_{i, 1}})\]
which is symmetric in $x_1$ and $x_2$, satisfies $\on{prank} \sigma \leq  O(r^{O(1)})$ and $(\alpha + \sigma)(d,d,x_3, x_4, x_5) = 0$ for all $d,x_3, x_4, x_5 \in U$, as desired.\\

\noindent\textbf{Case 2: $m = 3$.} In this case equality~\eqref{2copiesApproxEqn1} becomes
\[\alpha(d,d, x_3, y_4, y_5) = \sum_{i \in [s_1]} \beta_{i}(d, x_3) \gamma_i(y_4, y_5) + \sum_{i \in [s_2]} \beta'_{i}(d, y_4) \gamma'_i(x_3, y_5) + \sum_{i \in [s_3]} \beta''_{i}(d, y_5) \gamma''_i(x_3, y_4)\]
for some $s_1, s_2, s_3 \leq O(r^{O(1)})$ and bilinear forms $\beta_i,\dots, \gamma''_i$. Recall that all these forms come from slice of $\alpha$ and we have trilinear forms $\tilde{\beta}_i(x_{[3]})$, symmetric in $x_{[3]}$, $\tilde{\beta}'_i(x_{[3]})$, symmetric in $x_{[2]}$, and $\tilde{\beta}'_i(x_{[3]})$, symmetric in $x_{[2]}$, such that $\beta_{i}(d, x_3) = \tilde{\beta}_i(d,d,x_3)$, $\beta'_{i}(d, y_4) = \tilde{\beta}'_i(d,d,y_4)$ and $\beta_{i}''(d, y_5) = \tilde{\beta}''_i(d,d,y_5).$ Let $\sigma \colon G^5 \to \mathbb{F}_2$ be the multilinear form defined as
\begin{align*}\sigma(x_{[5]}) = \sum_{i \in [s_1]}& \tilde{\beta}_{i}(x_1, x_2, x_3) \gamma_i(x_4, x_5) \\
&+ \sum_{i \in [s_2]} \Big(\tilde{\beta}'_{i}(x_1, x_2, x_4) \gamma'_i(x_3, x_5) + \tilde{\beta}'_{i}(x_1, x_3, x_4) \gamma'_i(x_2, x_5)+ \tilde{\beta}'_{i}(x_2, x_3, x_4) \gamma'_i(x_1, x_5)\Big)\\
&+ \sum_{i \in [s_3]} \Big(\tilde{\beta}''_{i}(x_1, x_2, x_5) \gamma''_i(x_3, x_4) + \tilde{\beta}''_{i}(x_1, x_3, x_5) \gamma''_i(x_2, x_4)+ \tilde{\beta}''_{i}(x_2, x_3, x_5) \gamma''_i(x_1, x_4)\Big).\end{align*}
This form is symmetric in $x_{[3]}$ and we have $\on{prank} \sigma \leq O(r^{O(1)})$ and $(\alpha + \sigma)(d,d,x_3, x_4, x_5) = 0$ for all $d,x_3, x_4, x_5 \in U$, as desired.\\

\noindent\textbf{Case 3: $m = 4$.} In this case equality~\eqref{2copiesApproxEqn1} becomes
\[\alpha(d,d, x_3, x_4, y_5) = \sum_{i \in [s_1]} \beta_{i}(d, x_3) \gamma_i(x_4, y_5) + \sum_{i \in [s_2]} \beta'_{i}(d, x_4) \gamma'_i(x_3, y_5) + \sum_{i \in [s_3]} \beta''_{i}(d, y_5) \gamma''_i(x_3, x_4)\]
for some $s_1, s_2, s_3 \leq O(r^{O(1)})$ and bilinear forms $\beta_i,\dots, \gamma''_i$. Similarly to the previous case, recall that all these forms come from slice of $\alpha$ and that we have trilinear forms $\tilde{\beta}_i(x_{[3]})$, symmetric in $x_{[3]}$, $\tilde{\beta}'_i(x_{[3]})$, symmetric in $x_{[3]}$ (note that this form is now symmetric in all variables in contrast to the previous case), and $\tilde{\beta}''_i(x_{[3]})$, symmetric in $x_{[2]}$, such that $\beta_{i}(d, x_3) = \tilde{\beta}_i(d,d,x_3)$, $\beta'_{i}(d, x_4) = \tilde{\beta}'_i(d,d,x_4)$ and $\beta_{i}''(d, y_5) = \tilde{\beta}''_i(d,d,y_5)$. Also, $\gamma''_i(x_3, x_4)$ is symmetric for all $i$.\\
\indent We now regularize some of the forms. Let $C, D \geq 2$ be two parameters to be chosen later. First, apply Lemma~\ref{symmRespWRegLemma} (with parameters $C, D$ and $m = 0$ so that the symmetry properties play no role, see Remark~\ref{symmRespWRegLemmaRemark}) to the list consisting of forms $\gamma_i(u,v)$, $i \in [s_1]$ and $\gamma'_i(u,v)$, $i \in [s_2]$. We thus obtain bilinear forms $\tilde{\gamma}_1, \dots, \tilde{\gamma}_t$, each being a linear combination of the forms in the given list, for some $t \leq O(r^{O(1)})$ and some $R \leq C^{D^{O(r^{O(1)}}}$ such that
\begin{itemize}
\item every bilinear form in the list differs from a linear combination of forms $\tilde{\gamma}_1, \dots, \tilde{\gamma}_t$ by a bilinear form of rank at most $R$,
\item non-zero linear combinations of $\tilde{\gamma}_1, \dots, \tilde{\gamma}_t$ have rank at least $(C(R + 2r))^{D}$.
\end{itemize}
Next, apply Lemma~\ref{symmRegLemmaTimes2} to the list of forms $\tilde{\beta}_i(x_{[3]})$, for $i \in [s_1]$, and $\tilde{\beta}'_i(x_{[3]})$, for $i \in [s_2]$, with parameters $M = D$ and $R_0 = r$. We misuse the notation and still write $R$ for the parameter produced by the lemma, which still satisfies the earlier bound $R \leq C^{D^{O(r^{O(1)}}}$. Let $\Lambda^{\times 2}$ be the subspaces provided by the lemma and let $e_1, \dots, e_q$ be a basis of an additive complement of $\Lambda^{\times 2}$ in $\mathbb{F}_2^{s_1 + s_2}$. Let $\tau_1, \dots, \tau_q$ be the trilinear forms given by linear combinations of the trilinear forms in the given list with coefficients corresponding to $e_1, \dots, e_q$. Then
\begin{itemize}
\item every non-zero linear combination of $\tau_1, \dots, \tau_q$ has partition rank at least $(2r + 1)^D$,
\item every non-zero linear combination of $\tau_1^{\times 2}, \dots, \tau_q^{\times 2}$ has rank at least $(2r + 1)^D$,
\item every ${\tilde{\beta}}_i^{\times 2}$ and every $\tilde{\beta}'_i {}^{\times 2}$ differs from a linear combination of $\tau_1^{\times 2}, \dots, \tau_q^{\times 2}$ by a bilinear form of rank at most $R$. 
\end{itemize}

Using these properties, we may pass to a further subspace $U' \leq U$ of codimension $\on{codim}_G U' \leq O(r^{O(1)}R)$ and we may find coefficients $\lambda_{ij}, \lambda'_{ij}$ for $i \in [q], j \in [t]$ such that 
\begin{align*}\alpha(d,d, x_3, x_4, y_5) = &\sum_{i \in [q], j \in [t]} \lambda_{ij}\tau_i(d, d, x_3) \tilde{\gamma}_j(x_4, y_5) + \sum_{i \in [q], j \in [t]} \lambda'_{ij}\tau_i(d, d, x_4) \tilde{\gamma}_j(x_3, y_5) + \sum_{i \in [s_3]} \beta''_{i}(d, y_5) \gamma''_i(x_3, x_4)\end{align*}
holds for all $d, x_3, x_4, y_5 \in U'$.\\
\indent Recall that $\alpha$ is symmetric in $x_{[4]}$. Thus, for all $d, x_3, x_4, x_5 \in U'$ we have
\begin{align*}&0 = \alpha(d,d, x_3, x_4, y_5) + \alpha(d,d, x_4, x_3, y_5) \\
&\hspace{1cm}= \sum_{i \in [q], j \in [t]} \Big((\lambda_{ij} + \lambda'_{ij})\tau_i^{\times 2}(d, x_3) \tilde{\gamma}_j(x_4, y_5) \Big) + \sum_{i \in [q], j \in [t]} \Big((\lambda'_{ij} + \lambda_{ij})\tau_i^{\times 2}(d, x_4) \tilde{\gamma}_j(x_3, y_5)\Big).\end{align*}

Applying Lemma~\ref{zeromlformindexlemma} shows that $\lambda_{ij} = \lambda'_{ij}$ for all $i,j$. To finish the work in this case, define 
\begin{align*}\sigma(x_{[5]}) = \sum_{i \in [q], j \in [t]}& \lambda_{ij}\Big(\sum_{\{a,b,c\} \in \binom{[4]}{3}}\tau_{i}(x_a, x_b, x_c) \tilde{\gamma}_i(x_{[4] \setminus \{a,b,c\}}, x_5)\Big)\\
&+\sum_{i \in [s_3]} \Big(\sum_{\{a, b\} \in \binom{[4]}{2}} \tilde{\beta}''_{i}(x_a, x_b, x_5) \gamma''_i(x_{[4] \setminus \{a,b\}})\Big).\end{align*}

Finally, pass to the subspace $U''$ consisting of all $u \in U'$ such that all $\gamma''_i(u,u) = 0$ (recall that these forms are symmetric). On $U''$ we have that $(\alpha + \sigma)^{\times 2}$ vanishes, and we know that $\sigma(x_{[5]})$ is symmetric in $x_{[4]}$ and has $\on{prank}\sigma \leq O(r^{O(1)})$.\\

\noindent\textbf{Case 4: $m = 5$.} In this case equality~\eqref{2copiesApproxEqn1} becomes
\[\alpha(d,d, x_3, x_4, x_5) = \sum_{i \in [s_1]} \beta_{i}(d, x_3) \gamma_i(x_4, x_5) + \sum_{i \in [s_2]} \beta'_{i}(d,x_4) \gamma'_i(x_3, x_5) + \sum_{i \in [s_3]} \beta''_{i}(d, x_5) \gamma''_i(x_3, x_4)\]
for some $s_1, s_2, s_3 \leq O(r^{O(1)})$ and bilinear forms $\beta_i,\dots, \gamma''_i$ and the forms $\gamma_i, \gamma'_i, \gamma''_i$ are symmetric. As before, there are symmetric trilinear forms $\tilde{\beta}_i, \tilde{\beta}'_i$ and $\tilde{\beta}''_i$ such that $\beta_i = \tilde{\beta}_i{}^{\times 2}$, etc. We regularize the forms as in the previous case. The lists are slightly different, but the details are the same so we are deliberately concise in order to avoid repetition.\\
\indent Let $C, D \geq 2$ be two parameters to be chosen later. We apply Lemma~\ref{symmRespWRegLemma} to the list consisting of forms $\gamma_i(u,v)$, $i \in [s_1]$, $\gamma'_i(u,v)$, $i \in [s_2]$, and $\gamma''_i(u,v)$, $i \in [s_3]$ and Lemma~\ref{symmRegLemmaTimes2} to the list of forms $\tilde{\beta}_i(x_{[3]})$, $i \in [s_1]$, $\tilde{\beta}'_i(x_{[3]})$, $i \in [s_2]$, and $\tilde{\beta}''_i(x_{[3]})$,  $i \in [s_3]$. We obtain $t, q \leq O(r^{O(1)})$, $R \leq C^{D^{O(r^{O(1)}}}$, bilinear forms $\tilde{\gamma}_1, \dots, \tilde{\gamma}_t$ and trilinear forms $\tau_1, \dots, \tau_q$ such that
\begin{itemize}
\item every form in the list of bilinear forms differs from a linear combination of $\tilde{\gamma}_1, \dots, \tilde{\gamma}_t$ by a bilinear form of rank at most $R$,
\item non-zero linear combinations of $\tilde{\gamma}_1, \dots, \tilde{\gamma}_t$ have rank at least $(C(R + 2r))^{D}$,
\item each of the forms ${\tilde{\beta}}_i^{\times 2}$, $\tilde{\beta}'_i {}^{\times 2}$ and $\tilde{\beta}''_i {}^{\times 2}$ differs from a linear combination of $\tau_1^{\times 2}, \dots, \tau_q^{\times 2}$ by a bilinear form of rank at most $R$,
\item every non-zero linear combination of $\tau_1, \dots, \tau_q$ has partition rank at least $(2r + 1)^D$,
\item every non-zero linear combination of $\tau_1^{\times 2}, \dots, \tau_q^{\times 2}$ has rank at least $(2r + 1)^D$,
\item the newly obtained forms are symmetric.
\end{itemize}

We may now pass to a further subspace $U' \leq U$ of codimension $\on{codim}_G U' \leq O(r^{O(1)}R)$ and we may find coefficients $\lambda_{ij}, \lambda'_{ij}, \lambda''_{ij}$ for $i \in [q], j \in [t]$ such that 
\begin{align*}\alpha(d,d, x_3, x_4, x_5) = \sum_{i \in [q], j \in [t]} \lambda_{ij}\tau_i(d, d, x_3) \tilde{\gamma}_j(x_4, x_5) + &\sum_{i \in [q], j \in [t]} \lambda'_{ij}\tau_i(d, d, x_4) \tilde{\gamma}_j(x_3, x_5)\\
 + &\sum_{i \in [q], j \in [t]} \lambda''_{ij}\tau_i(d, d, x_5) \tilde{\gamma}_j(x_3, x_4)\end{align*}
holds for all $d, x_3, x_4, x_5 \in U'$.\\

Recalling that $\alpha$ is symmetric and applying Lemma~\ref{zeromlformindexlemma} to expressions $\alpha(d,d, x_3, x_4, x_5) + \alpha(d,d, x_4, x_3, x_5)$ and $\alpha(d,d, x_3, x_4, x_5) + \alpha(d,d, x_5, x_4, x_3)$ shows that $\lambda_{ij} = \lambda'_{ij}= \lambda''_{ij}$ for all $i,j$. Let us also pass to a further subspace $U'' \leq U'$ consisting of all $u \in U'$ such that $\tilde{\gamma}_j(u,u) = 0$ for all $j \in [t]$, whose codimension is $\on{codim}_G U'' \leq O(r^{O(1)}R)$. Finally, consider 
\begin{align*}\sigma(x_{[5]}) = \sum_{i \in [q], j \in [t]}& \lambda_{ij}\Big(\sum_{\{a,b,c\} \in \binom{[5]}{3}}\tau_{i}(x_a, x_b, x_c) \tilde{\gamma}_i(x_{[5] \setminus \{a,b,c\}})\Big).\end{align*}

This is a symmetric multilinear form with $\on{prank} \sigma \leq O(r^{O(1)})$ and $(\alpha + \sigma)^{\times 2}$ vanishes on $U''$.\end{proof}

\section{Proof of inverse theorem}

In this section we combine Theorems~\ref{approximateSymmetry2},~\ref{mainSymmExtn4},~\ref{mainSymmExtn5} and~\ref{copiesPropertyVanish} with other additive-combinatorial arguments in order to prove our main result.

\begin{proof}[Proof of Theorem~\ref{mainthm}]We prove the claim for $k \in \{4,5\}$ and assume the theorem for $k - 1$. The proof will have the following structure.
\begin{itemize}
\item[\textbf{Step 1.}] We first show that whenever
\begin{equation}\Big|\exx_{a_1, \dots, a_k, x} \mder_{a_1, \dots, a_k} f(x) (-1)^{\alpha(a_1, \dots, a_k)}\Big| \geq c\label{mainthmassncopy}\end{equation}
is satisifed, we may pass to a carefully chosen subspace $U$ on which $\alpha$ has the additional property that the multilinear form $(a_2, \dots, a_k) \mapsto \alpha(u, a_2, \dots, a_k)$ is a sum of a strongly symmetric and a bounded partition rank form for each $u \in U$.
\item[\textbf{Step 2.}] Next, we prove that $\alpha|_{U' \tdt U'}$ is a sum of strongly symmetric and low partition rank form on a suitable subspace $U'$ of bounded codimension.
\item[\textbf{Step 3.}] Using the structure of multilinear form $\alpha|_{U' \tdt U'}$, we deduce that $\alpha$ is of the desired shape.
\end{itemize}

\noindent\textbf{Step 1.} We formulate the work in this step as the following lemma.

\begin{lemma}\label{subspaceAddlProperty}Let $V \leq G$ be a subspace and let $g \colon V \to \mathbb{D}$ be a function. Suppose that 
\begin{equation}\Big| \exx_{x, a_1, \dots, a_k \in V} \mder_{a_1} \dots \mder_{a_k} g(x) (-1)^{\alpha(a_1, \dots, a_k)}\Big| \geq c.\label{subspaceAddlPropertyAssumptionIneq}\end{equation}
Then there exits a subspace $U \leq V$ of codimension at most $O(\log^{O(1)} (2c^{-1}))$ and a function $g' \colon U \to \mathbb{D}$ such that
\begin{equation}\Big| \exx_{x, a_1, \dots, a_k \in U} \mder_{a_1} \dots \mder_{a_k} g'(x) (-1)^{\alpha(a_1, \dots, a_k)}\Big| \geq c^{2^k}\label{subspacepasscorrelationmder}\end{equation}
and, for any $u \in U$, the multilinear form $(a_2, \dots, a_k) \mapsto \alpha(u, a_2, \dots, a_k)$ (where $a_2, \dots, a_k \in U$) is a sum of a strongly symmetric form and a form of partition rank at most $r$ for some $r \leq O(\exp^{(O(1))} c^{-1})$.
\end{lemma}

\begin{proof}[Proof of Lemma~\ref{subspaceAddlProperty}]Let $B$ be the set of all $b \in V$ such that
\[\Big|\exx_{x, a_2, \dots, a_k \in V} \mder_b \mder_{a_2} \dots \mder_{a_k} g(x) (-1)^{\alpha(b, a_2, \dots, a_k)}\Big| \geq \frac{c}{2}.\]
From assumption~\eqref{subspaceAddlPropertyAssumptionIneq} we see that $|B| \geq \frac{c}{2}|V|$. By Theorem~\ref{sandersBogRuzsa}, $4B$ contains a subspace $U$ of codimension at most $O(\log^{O(1)} (2c^{-1}))$. Take any $u \in U$. In particular, since $U \subseteq 4B$, $u$ can be written as $b_1 + b_2 + b_3 + b_4$ for some $b_1, b_2, b_3, b_4 \in B$. For each $i \in [4]$ we then have
\[\Big|\exx_{x, a_2, \dots, a_k \in V} \mder_{a_2} \dots \mder_{a_k} \Big( \mder_{b_i}  g\Big)(x) \omega^{\alpha(b_i, a_2, \dots, a_k)}\Big| \geq \frac{c}{2}.\]
By induction hypothesis applied to the function $\mder_{b_i}f$, we deduce that the multilinear form $(a_2, \dots, a_k) \mapsto \alpha(b_i, a_2, \dots, a_k)$ is a sum of a strongly symmetric form and a form of partition rank at most $r$, for some positive quantity $r \leq O(\exp^{(O(1))} c^{-1})$. Thus, the multilinear form $(a_2, \dots, a_k) \mapsto \alpha(u, a_2, \dots, a_k)$ is a sum of a strongly symmetric form and a form of partition rank at most $4r$.\\
\indent We now pass to $U \tdt U$. Take any direct sum $U \oplus W = V$. Going back to~\eqref{subspaceAddlPropertyAssumptionIneq} and using this decomposition of $V$ we obtain
\[\Big|\exx_{\ssk{a'_1, \dots, a'_k, x \in U\\w_1, \dots, w_k, y \in W}} \mder_{a'_1 + w_1} \dots \mder_{a'_k + w_k} g(x + y) (-1)^{\alpha(a'_1 + w_1, \dots, a'_k + w_k)}\Big| \geq c.\]
Average over elements $w_1, \dots, w_k, y \in W$ and use the triangle inequality to find a choice such that
\[\label{chvarineqdecomp}\Big| \exx_{x, a'_1, \dots, a'_k \in U} \mder_{a'_1 + w_1} \dots \mder_{a'_k + w_k} g(x + y) (-1)^{\alpha(a'_1 + w_1, \dots, a'_k + w_k)}\Big| \geq c.\]

We now use Gowers-Cauchy-Schwarz inequality to simplify the expression above. For each subset $I \subseteq [k]$ let $\tilde{g}_I \colon U^k \to \mathbb{D}$ be the function defined by
\[\tilde{g}_I(u_1, \dots, u_k) = g\Big(u_1 + \dots + u_k + y + \sum_{i \in I} w_i\Big) (-1)^{\alpha\big((u_i + w_i)_{i \in I}, (u_i)_{i \in [k] \setminus I}\big)}.\]
With the functions defined this way we get
\[\exx_{x, a'_1, \dots, a'_k \in U} \mder_{a'_1 + w_1} \dots \mder_{a'_k + w_k} g(x + y) (-1)^{\alpha(a'_1 + w_1, \dots, a'_k + w_k)} = \exx_{u_1, u'_1, \dots, u_k, u'_k \in U} \prod_{I \subset [k]} \operatorname{Conj}^{|I|} g_I(u_I, u'_{[k] \setminus I}).\]
The modulus of this expression can be bounded from above using the Gowers-Cauchy-Schwarz inequality by $\|\tilde{g}_{\emptyset}\|_{\square}$. Therefore we obtain 

\[c^{2^k}\leq \|\tilde{g}_{\emptyset}\|_{\square}^{2^k} = \Big| \exx_{x, a_1, \dots, a_k \in U} \mder_{a_1} \dots \mder_{a_k} g(x + y) (-1)^{\alpha(a_1, \dots, a_k)}\Big|.\]

Define $g' \colon U \to \mathbb{C}$ by $g'(x) = g(x + y)$. Hence
\[\Big| \exx_{x, a_1, \dots, a_k \in U} \mder_{a_1} \dots \mder_{a_k} g'(x) (-1)^{\alpha(a_1, \dots, a_k)}\Big| \geq c^{2^k},\]
which takes the same shape as assumption~\eqref{mainthmassncopy}, but this time $\alpha$ has the additional property that, for any $u \in U$, the multilinear form $(a_2, \dots, a_k) \mapsto \alpha(u, a_2, \dots, a_k)$ is a sum of a strongly symmetric form and a form of partition rank at most $4r$.\end{proof}

\noindent\textbf{Step 2.} In this step we show that $\alpha$ can essentially be assumed to be symmetric and to satisfy $\alpha(d,d,a_3, \dots, a_k) = 0$. This is formulated precisely in the following claim.

\begin{claim}Suppose that a function $f \colon G \to \mathbb{D}$ and a multilinear form $\alpha \colon G^k \to \mathbb{F}_2$ satisfy~\eqref{mainthmassncopy}. For each $\ell \in [2,k]$ there exist a subspace $U \leq G$ of codimension at most $O(\exp^{(O(1))} c^{-1})$, a multilinear form $\beta \colon U^k \to \mathbb{F}_2$ and a function $g \colon U \to \mathbb{D}$ such that
\begin{itemize}
\item[\textbf{(i)}] $\beta$ is symmetric in the first $\ell$ variables,
\item[\textbf{(ii)}] $\beta(d,d,a_3, \dots, a_k) = 0$ for all $d, a_3, \dots, a_k \in U$,
\item[\textbf{(iii)}] $\Big|\ex_{x,a_1, \dots, a_k \in U} \mder_{a_1}\dots \mder_{a_k}g(x) (-1)^{\beta(a_1, \dots, a_k)}\Big| \geq \Omega((\exp^{(O(1))} c^{-1})^{-1})$,
\item[\textbf{(iv)}] $\alpha|_{U \tdt U} = \beta + \sigma + \delta$ for a strongly symmetric multilinear form $\sigma \colon U^k \to \mathbb{F}_2$ and a multilinear form $\delta \colon U^k \to \mathbb{F}_2$ of partition rank at most $O(\exp^{(O(1))} c^{-1})$.
\end{itemize}\end{claim}

\begin{proof}We prove the claim by induction on $\ell$. First we prove the base case $\ell = 2$. By Lemma~\ref{symmArgumentCor} we see that $\on{prank}(\alpha + \alpha \circ (1\,\,2)) \leq O\Big((\log c^{-1})^{O(1)}\Big)$. Theorem~\ref{approximateSymmetry2} produces a multilinear form $\beta \colon G^k \to \mathbb{F}_2$, symmetric in the first two variables which differs from $\alpha$ by a multilinear form of partition rank at most $O(\exp^{(O(1))} c^{-1})$. This immediately gives property \textbf{(iv)} with $\sigma = 0$. Lemmas~\ref{closeformsreplacementinverse} and~\ref{boundedcodimsubspacepass} imply property \textbf{(iii)}. Note that the proof of the base case is not yet complete as we have not addressed the $\beta(d,d,a_3, \dots, a_k) = 0$ property. We do this now in a more general form, which will also be used in the inductive step.\\

\noindent\textbf{Obtaining property \textbf{(ii)}.} Suppose that we are given a subspace $U \leq G$ of codimension at most $O(\exp^{(O(1))} c^{-1})$, a multilinear form $\alpha' \colon U^k \to \mathbb{F}_2$ and a function $g \colon U \to \mathbb{D}$ which satisfy conditions \textbf{(i)}, \textbf{(iii)} and \textbf{(iv)} (with $\beta$ replaced by $\alpha'$). Since $\alpha'$ satisfies property \textbf{(iii)}, we may use Lemma~\ref{subspaceAddlProperty} to pass to a further subspace $U' \leq U$ of codimension at most $O(\exp^{(O(1))} c^{-1})$ and to find a function $g' \colon U' \to \mathbb{D}$ such that
\begin{equation}\Big|\ex_{x,a_1, \dots, a_k \in U'} \mder_{a_1}\dots \mder_{a_k}g'(x) (-1)^{\alpha'(a_1, \dots, a_k)}\Big| \geq c'\label{step2mainproofgprimecorrneqn}\end{equation}
where $c' \geq \Omega((\exp^{(O(1))} c^{-1})^{-1})$, and for any $u \in U'$, the multilinear form $(a_2, \dots, a_k) \mapsto \alpha'(u, a_2, \dots, a_k)$ (where $a_2, \dots, a_k \in U'$) is a sum of a strongly symmetric form and a form of partition rank at most $r$ for some $r \leq O(\exp^{(O(1))} c^{-1})$.\\

By making a slight change of variables we obtain
\[\Big|\ex_{x,b,d,a_3, \dots, a_k \in U'} \mder_{b+d} \mder_d\mder_{a_3}\dots \mder_{a_k}g'(x) (-1)^{\alpha'(b+d, d, a_3, \dots, a_k)}\Big| \geq c'.\]
We may find $b \in U'$ such that 
\[\Big| \exx_{x, d, a_3, \dots, a_k} \mder_{b +d} \mder_d \mder_{a_3} \dots \mder_{a_k} g'(x) (-1)^{\alpha'(b + d, d, a_3, \dots, a_k)}\Big| \geq c'.\]
However, 
\[\mder_{b +d} \mder_d \mder_{a_3} \dots \mder_{a_k} g'(x) = \mder_d \mder_{a_3} \dots \mder_{a_k} h(x),\]
where $h(x) = \overline{g'(x) g'(x + b)}$. Thus
\[\Big| \exx_{x, d, a_3, \dots, a_k} \mder_d \mder_{a_3} \dots \mder_{a_k} h(x) (-1)^{\alpha'(b + d, d, a_3, \dots, a_k)}\Big| \geq c'.\]

We may apply the case $k-1$ of the theorem to the multilinear form $(d, a_3, \dots, a_k) \mapsto \alpha'(b + d, d, a_3, \dots, a_k)$ (note that this is still multilinear as $\alpha'$ is symmetric in the first two variables) to conclude that $(d, a_3, \dots, a_k) \mapsto \alpha'(b + d, d, a_3, \dots, a_k)$ is a sum of a strongly symmetric multilinear form and a multilinear form of partition rank at most $O(\exp^{(O(1))} c^{-1})$ on subspace $U'$. But, recall that a similar property holds for the multilinear form $(d, a_3, \dots, a_k) \mapsto \alpha'(b, d, a_3, \dots, a_k)$ by our choice of the subspace $U'$. Thus, we conclude that there exist a strongly symmetric mulitlinear form $\sigma' \colon (U')^{k-1} \to \mathbb{F}_2$ and a multilinear form $\delta' \colon (U')^{k-1} \to \mathbb{F}_2$ of partition rank at most $O(\exp^{(O(1))} c^{-1})$ such that for all $d, a_3, \dots, a_k \in U'$ 
\begin{equation}\label{alphaprimedoubledeqn}\alpha'(d,d,a_3, \dots, a_k) = \sigma'(d,a_3, \dots, a_k) + \delta'(d,a_3, \dots, a_k)\end{equation}
holds. By Lemma~\ref{liftingssforms} there exists a strongly symmetric multilinear form $\tilde{\sigma} \colon (U')^k \to \mathbb{F}_2$ such that 
\begin{equation}\label{tildesigmalifteqn}\tilde{\sigma}(d,d,a_3, \dots, a_k) = \sigma'(d,a_3, \dots, a_k).\end{equation}

Let $\tilde{\alpha} = \alpha'|_{U'\tdt U'} + \tilde{\sigma}$, which is still symmetric in the first $\ell$ variables. Using Lemma~\ref{ssformintegration} we may find phase $s \colon U' \to \mathbb{D}$ of a non-classical polynomial such that $\mder_{a_1} \dots \mder_{a_k} s(x) = (-1)^{\tilde{\sigma}(a_1, \dots, a_k)}$ for all $a_1, \dots, a_k, x \in U'$. Putting $\tilde{g}(x) = g'(x) s(x)$, from~\eqref{step2mainproofgprimecorrneqn} we obtain
\begin{equation}\Big|\ex_{x,a_1, \dots, a_k \in U'} \mder_{a_1}\dots \mder_{a_k} \tilde{g}(x) (-1)^{\tilde{\alpha}(a_1, \dots, a_k)}\Big| \geq c'.\label{tildealphacorrelationeqn}\end{equation}
On the other hand, from~\eqref{alphaprimedoubledeqn} and~\eqref{tildesigmalifteqn}, for all $d, a_3, \dots, a_k \in U'$ we see that 
\begin{align*}\tilde{\alpha}(d,d,a_3, \dots, a_k) = &\alpha'(d,d,a_3, \dots, a_k) + \tilde{\sigma}(d,d,a_3, \dots, a_k) \\
=& \sigma'(d,a_3, \dots, a_k) + \delta'(d,a_3, \dots, a_k) +  \sigma'(d,a_3, \dots, a_k) = \delta'(d,a_3, \dots, a_k)\end{align*}
which has partition rank at most $O(\exp^{(O(1))} c^{-1})$. By Theorem~\ref{copiesPropertyVanish} we conclude that there exist a subspace $U'' \leq U'$ of codimension $O(\exp^{(O(1))} c^{-1})$ and a multilinear form $\beta \colon (U'')^k \to \mathbb{F}_2$, symmetric in first $\ell$ variables, such that $\beta(d,d,a_3, \dots, a_k) = 0$ for all $d, a_3, \dots, a_k \in U''$ and $\on{prank}(\tilde{\alpha}|_{U'' \tdt U''} + \beta) \leq O(\exp^{(O(1))} c^{-1})$. Since we need to pass to further subspace $U''$, we use~\eqref{tildealphacorrelationeqn} and apply Lemma~\ref{boundedcodimsubspacepass} which provides us with a function $h \colon U'' \to \mathbb{D}$ such that
\[\Big|\ex_{x,a_1, \dots, a_k \in U''} \mder_{a_1}\dots \mder_{a_k} h(x) (-1)^{\tilde{\alpha}(a_1, \dots, a_k)}\Big| \geq c'.\]
Since $\on{prank}(\tilde{\alpha}|_{U'' \tdt U''} + \beta) \leq O(\exp^{(O(1))} c^{-1})$, Lemma~\ref{closeformsreplacementinverse} allows us to conclude that $\beta$ satisfies property \textbf{(iii)} with function $h$ on subspace $U''$. Finally, writing $\tilde{\delta} = \tilde{\alpha}|_{U'' \tdt U''} + \beta$, we have
\begin{align*}\alpha|_{U'' \tdt U''} = &\alpha'|_{U'' \tdt U''} + \sigma|_{U'' \tdt U''} + \delta|_{U'' \tdt U''} \\
= &\tilde{\alpha}|_{U'' \tdt U''} + (\tilde{\sigma}|_{U'' \tdt U''} + \sigma|_{U'' \tdt U''}) + \delta|_{U'' \tdt U''}\\
= &\beta + (\tilde{\sigma}|_{U'' \tdt U''} + \sigma|_{U'' \tdt U''}) + (\delta|_{U'' \tdt U''} + \tilde{\delta})\end{align*}
proving property \textbf{(iv)}. In particular, when $\ell = 2$ this argument allows us to complete the base case, so we now move on to proving the inductive step.\\

\noindent\textbf{Inductive step.} Suppose now that the claim holds for some $\ell \geq 2$. Let $U, \beta$ and $g$ be the relevant subspace, multilinear form and function for $\ell$. Since $\beta$ satisfies property \textbf{(iii)}, we may use Lemma~\ref{symmArgumentCor} to conclude that $\on{prank}(\beta + \beta \circ (1 \,\, \ell + 1)) \leq O(\exp^{(O(1))} c^{-1})$. When $\ell + 1$ is odd, set $\beta' = \sum_{j \in [\ell + 1]} \beta \circ (j\,\,\ell + 1)$. If $\ell + 1$ is even, so it has to be $\ell + 1 = 4$, we apply Theorem~\ref{mainSymmExtn4} when $k =4$ and Theorem~\ref{mainSymmExtn5} when $k =5$. In all the described cases, we conclude that there exists a further multilinear form $\beta' \colon U^k \to \mathbb{F}_2$, symmetric in the first $\ell + 1$ variables such that $\on{prank}(\beta + \beta') \leq O(\exp^{(O(1))} c^{-1})$. Using Lemma~\ref{closeformsreplacementinverse} we immediately see that $\beta'$ satisfies conditions \textbf{(i)}, \textbf{(iii)} and \textbf{(iv)} for $\ell + 1$. The previous part of this proof allows us to pass to a further multilinear form that satisfies property \textbf{(ii)} as well, completing the proof of the claim.\end{proof}

\noindent\textbf{Step 3.} We may now apply the claim above with $\ell = k$ to see that on a subspace $U$ of codimension $O(\exp^{(O(1))} c^{-1})$ we have $\alpha|_{U\tdt U} = \sigma + \delta$ for a strongly symmetric multilinear form $\sigma \colon U^k \to \mathbb{F}_2$ and a multilinear form $\delta \colon U^k \to \mathbb{F}_2$ of partition rank at most $O(\exp^{(O(1))} c^{-1})$. Pick any projection $\pi \colon G \to U$ and define $\tilde{\sigma}, \tilde{\delta} \colon G \tdt G \to \mathbb{F}_2$ by setting
\[\tilde{\sigma}(x_1, \dots, x_k) = \sigma(\pi(x_1), \dots, \pi(x_k))\hspace{1cm}\text{and}\hspace{1cm}\tilde{\delta}(x_1, \dots, x_k) = \delta(\pi(x_1), \dots, \pi(x_k)).\]
Clearly, $\tilde{\sigma}$ is strongly symmetric, while $\tilde{\delta}$ is of bounded partition rank and $\tilde{\sigma}|_{U \tdt U} = \sigma$, $\tilde{\delta}|_{U \tdt U} = \delta$. Set $\rho = \alpha + \tilde{\sigma} + \tilde{\delta}$ which is a multilinear form satisfying $\rho = 0$ on $U \tdt U$. It remains to show that such a map has low partition rank. Let $c''$ be the density of $U$. Then by Lemma~\ref{unifBound}
\[{c''}^k \leq \exx_{x_1, \dots, x_k \in G}(-1)^{\rho(x_1, \dots, x_k)} \id_U(x_1) \cdots \id_U(x_k) \leq \| (-1)^{\rho} \|_{\square^k} =  (\on{bias} \rho)^{2^{-k}}.\]
By Theorem~\ref{biasedinversethm}, it follows that the partition rank of $\rho$ is at most $O(\exp^{(O(1))} c^{-1})$. Thus,
\[\alpha = \tilde{\sigma} + (\tilde{\delta} + \rho),\]
completing the proof.\end{proof}

Corollary~\ref{inverseu5u6cor} follows from Theorems~\ref{partialInverseTheorem},~\ref{mainthm}, Lemma~\ref{closeformsreplacementinverse} and lower order inverse theorems.

\section{Concluding remarks}

\hspace{12pt}\indent Let us now briefly return to the following question which is central to this paper. Suppose that $\alpha \colon G^{2k} \to \mathbb{F}_2$ is a multilinear form which is symmetric in the first $2k-1$ variables and satisfies $\on{prank} \Big(\alpha + \alpha\circ(2k-1 \,\,2k)\Big) \leq r$. Is $\alpha$ close to a symmetric multilinear form? While we know that the answer is negative, the arguments in the proof of Theorem~\ref{mainSymmExtn5} could be used to show that we may modify $\alpha$ until the low partition rank decomposition of $\alpha + \alpha\circ(2k-1 \,\,2k)$ only has products $\sigma(x_I, x_{2k-1}) \cdot \sigma(x_{[2k-2] \setminus I}, x_{2k})$ for a symmetric multilinear form $\sigma$. In other words, using the places terminology from the proof of Theorem~\ref{mainSymmExtn5}, we first need to make a distinction  between the cases when two places $A$ and $B$ (not necessarily distinct) occur in the same form inside the product, having coefficient $\lambda^{\text{same}}_{A\,\,B}$, or they occur in different forms, having coefficient $\lambda^{\text{diff}}_{A\,\,B}$. It turns out that we may still prove various equalities between coefficients, but the single piece of information that remains elusive is in the case when the places $A$ and $B$ are the same place in a symmetric form. For example looking at the product $\sigma(x_1, \dots, x_{2k-1})\sigma(x_{2k}, \dots)$, which has zero coefficient in $\phi + \phi \circ (1\,\,2k-1) + \phi \circ (1\,\,2k)$, where $\phi = \alpha + \alpha\circ(2k-1 \,\,2k)$, only proves $\lambda^{\text{same}}_{A\,\,A} = 0$ and says nothing about $\lambda^{\text{diff}}_{A\,\,B}$. Of course, the reason for that is the counterexample~\cite{FarSymm} which is multilinear form $\alpha \colon G^4 \to \mathbb{F}_2$ such that $\alpha$ is symmetric in the first 3 variables and satisfies 
\[\alpha(x_1, x_2, x_3, x_4) + \alpha(x_1, x_2, x_4, x_3) = \sigma(x_1, x_3)\sigma(x_2, x_4) + \sigma(x_2, x_3)\sigma(x_1, x_4)\]
for a high rank symmetric bilinear form $\sigma$. The arguments above actually show that the given counterexample is essentially the only way for the symmetry extension to fail. Note also that when $k =4$, we were able to overcome this difficulty by having an additional algebraic property of $\alpha$ that $\alpha(u,u,x_3, x_4) = 0$, which removed the problematic products of forms from $\alpha + \alpha\circ(2k-1 \,\,2k)$. Let us also remark the symmetry-respecting weak regularity lemma (Lemma~\ref{symmRespWRegLemma}) holds for any number of variables.\\

\indent With this in mind, we believe that resolving the following problem, combined with arguments in this paper, should lead to a solution of Problem~\ref{corrtossymm} and thus, paired with Theorem~\ref{partialInverseTheorem}, to a quantitative inverse theorem for Gowers uniformity norms in low characteristic.

\begin{conjecture}\label{finaldiffconj}Suppose that $\alpha \colon G^{2k}\to\mathbb{F}_2$ is a multilinear form which is symmetric in the first $2k-1$ variables such that
\[\alpha(x_{[2k]}) + \alpha(x_{[2k-2]}, x_{2k}, x_{2k-1}) = \sum_{I \in \binom{[2k-2]}{k-1}} \sigma(x_I, x_{2k-1})\sigma(x_{[2k-2] \setminus I}, x_{2k})\]
for a symmetric multilinear form $\sigma \colon G^k \to \mathbb{F}_2$. Suppose that $f \colon G \to \mathbb{D}$ is a function such that
\[\Big| \exx_{x, a_1, \dots, a_{2k}} \mder_{a_1} \dots \mder_{a_{2k}} f(x) \omega^{\alpha(a_1, \dots, a_{2k})}\Big| \geq c.\]
Then $\on{prank} \sigma \leq \exp^{(O_k(1))}(O_{k}(c^{-1}))$.
\end{conjecture}

The bounds in the conjecture are chosen to be in line with other bounds in this paper, but we suspect that the conjecture holds with the bound of the shape $\on{prank} \sigma \leq O_{k} (\log c^{-1})$.\\

Let us more generally pose the following more open-ended question.

\begin{question}Suppose that $\alpha \colon G^{2k} \to \mathbb{F}_2$ is a multilinear form which is symmetric in the first $2k-1$ variables and satisfies $\on{prank} \Big(\alpha + \alpha\circ(2k-1 \,\,2k)\Big) \leq r$. Under what algebraic condition on $\alpha$ can we guarantee to find a symmetric multilinear form $\alpha' \colon G^{2k} \to \mathbb{F}_2$ which satisfies $\on{prank} \Big(\alpha + \alpha'\Big) \leq O_r(1)$?\end{question}

\end{document}